\documentclass[final,onefignum,onetabnum]{siamart190516}

\usepackage{lipsum}
\usepackage{amsfonts}
\usepackage{amssymb}
\usepackage{graphicx}
\usepackage{epstopdf}
\usepackage{algorithmic}
\usepackage{array}
\newcolumntype{L}{>{\centering\arraybackslash}m{2.2cm}}
\newcolumntype{R}{>{\centering\arraybackslash}m{1cm}}
\ifpdf
  \DeclareGraphicsExtensions{.eps,.pdf,.png,.jpg}
\else
  \DeclareGraphicsExtensions{.eps}
\fi

\newsiamremark{remark}{Remark}
\newsiamremark{hypothesis}{Hypothesis}
\crefname{hypothesis}{Hypothesis}{Hypotheses}
\newsiamthm{claim}{Claim}

\headers{stochastic weighted particle method}{S. Lama, J. Zweck, and M. Goeckner}

\title{A higher order moment preserving reduction scheme for the Stochastic Weighted Particle Method \thanks{Submitted to the editors DATE.}}

\author{Sonam Lama\thanks{Department of Mathematics, The University of Texas at Dallas, Richardson, TX (\email{sonam.lama@utdallas.edu}, \email{zweck@utdallas.edu}).}
\and John Zweck\footnotemark[2] \and Matthew Goeckner\thanks{Department of Physics, The University of Texas at Dallas, Richardson, TX (\email{goeckner@utdallas.edu}).} }

\usepackage{amsopn}
\DeclareMathOperator{\diag}{diag}

\ifpdf
\hypersetup{
  pdftitle={A higher order moment preserving reduction scheme for the Stochastic Weighted Particle Method },
  pdfauthor={Sonam Lama, John Zweck, and Matthew Goeckner}
}
\fi

\begin{document}

\maketitle

\begin{abstract}
The Stochastic Weighted Particle Method (SWPM) is a Monte Carlo technique developed by Rjasanow and Wagner that generalizes Bird's Direct Simulation Monte Carlo (DSMC) method for solving the Boltzmann equation. To reduce 
computational cost
due to the gradual increase in the number of stochastic particles in the SWPM, Rjasanow and Wagner proposed several particle reduction schemes designed to preserve specified moments of the velocity distribution. Here, we introduce an improved particle reduction scheme that preserves all moments of the velocity distribution up to the second order, as well as the raw and central heat flux both within each group of particles to be reduced and for the entire system. Furthermore, we demonstrate that with the new reduction scheme the scalar fourth-order moment can be computed more accurately at a reduced computational cost. 
\end{abstract}



\section{Introduction}
\label{sec:intro}
A fundamental problem in the computational modeling of rarefied gases and plasmas is to determine the velocity probability density function (pdf) of each particle species. The evolution of these velocity pdfs is governed by the Boltzmann equation, which models particle transport and collision processes \cite{bittencourt2013fundamentals,boltzmann1970weitere,cercignani1988boltzmann}. Both deterministic and stochastic particle methods are used to solve the Boltzmann equation. Although deterministic methods avoid the uncertainties inherent in stochastic approaches, the cost of computing the Boltzmann collision operator can still be prohibitively high, especially in the low probability tails of the pdf. However, recent theoretical advances in combination with increased computational power have led to the introduction of several promising deterministic spectral methods \cite{gamba2009upper,gamba2014conservative,gamba2017fast,gamba2018galerkin,Mouhot:2006,Pareschi:2002,tran2013nonlinear}. For example, Gamba and Rjasanow recently proposed a Petrov-Gelerkin method whose computational efficiency is comparable to that of stochastic methods \cite{gamba2018galerkin}. Despite the recent reduction of their computational cost, deterministic methods are not as flexible as stochastic methods for the modeling of the diverse range of collision, transport, and boundary surface phenomena, and particle gain and loss mechanisms that occur in experimental settings \cite{johannes1997direct,kushner2009hybrid}.

Accurate modeling of the low probability tails of the velocity distribution is also of interest to experimentalists. For example, reaction rates in plasmas are determined by the overlap between the electron velocity pdf and the electron-impact cross sections of the various species. Therefore, accurate calculation of the low probability tails of the electron velocity pdf is critical. If the plasma is in thermal equilibrium, the electron velocity pdf can often be assumed to be Maxwellian. However, experimental results demonstrate that the Maxwellian assumption is often invalid \cite{allen1978applicability,dicarlo1989solving,sheridan1998electron,sozzi2008measurements,tan1973langmuir}, especially for pulsed plasmas where the velocity pdf may depend strongly on both spatial position and on time \cite{poulose2017driving}. Consequently, there is still a pressing need for improved stochastic particle methods that have both greater computational efficiency and higher accuracy, especially in the higher order moments and in the low probability tails of the distributions.

 Unlike deterministic methods, particle methods are not based on solving the Boltzmann equation directly. Rather they simulate a real system using stochastic particles, each of which represents a group of physical particles that are in close proximity in phase space. Collisions between stochastic particles are designed so as to approximate the collision processes modeled by the Boltzmann equation. Modern particle methods are based on the Direct Simulation Monte Carlo (DSMC) method which was developed by Bird \cite{bird1994molecular}. A convergence proof for this method was given by  Wagner \cite{wagner1992convergence}. The DSMC method has many computational advantages over deterministic methods. However, the computational cost of accurately computing the low probability tails is still very high. 
To resolve the low probability tails
with relatively low computational cost, Rjasanow and Wagner introduced a generalization of the DSMC method which they called the Stochastic Weighted Particle Method (SWPM). 
  
One of the challenges for the SWPM is that the number of stochastic particles gradually increases over the course of the simulation. To reduce the computational cost, Rjasanow and Wagner proposed to use a particle reduction scheme in combination with a clustering technique \cite{rjasanow1998reduction,rjasanow2005stochastic}. With these methods, the particles are partitioned into groups such that the particles are close together, and each group is replaced by a small number of particles. A reduction scheme that does not require clustering was proposed by Vikhansky and Kraft~\cite{vikhansky2005conservative}. 
Their reduction scheme redistributes the statistical weights of the particles 
so as to conserve the mass, momentum and energy of the ensemble. 
They argue that the efficiency of  a particle reduction scheme that relies on clustering 
primarily depends on the computational cost of the clustering algorithm. 
For the clustering algorithm used for the results in this paper, the 
computational cost scales linearly with the number of computational particles. 
In this context, it is important to note that 
the convergence theorem for the SWPM obtained  by 
Rjasanow and Wagner~\cite{rjasanow2005stochastic}
requires that the maximum diameter of the  groups of particles  
converges to zero as the initial number of computational particles increases. 
As a result, there is a theoretical advantage to employing a clustering technique.

The reduction schemes proposed by Rjasanow and Wagner were designed to preserve a specified set of moments of the distribution. The particle reduction scheme of Rjasanow and Wagner that preserves the most moments is a deterministic reduction scheme that preserves the total weight, momentum, energy and central heat flux within each group \cite{rjasanow1998reduction,rjasanow2005stochastic}. The total weight corresponds to the fraction of physical particles represented by the group. Although this reduction scheme preserves the central heat flux of each group, it does not preserve the raw heat flux, and consequently neither the raw nor the central heat flux are preserved for the entire system. 

In this paper, we improve upon the reduction scheme of Rjasanow and Wagner by conserving all of the moments up to the second order (i.e. the full pressure and momentum flux tensors), as well as both the raw and central heat flux, which are third order moments. Conservation of all these moments within each group automatically guarantees that they are conserved for the entire system. 

We performed two series of simulation studies to evaluate the degree to which our new deterministic particle reduction scheme improves upon that of Rjasanow and Wagner's deterministic reduction schemes. First, 
we present results which confirm
that the existing scheme of Rjasanow and Wagner does not conserve the raw heat flux within each group, while our new scheme conserves both raw and central heat flux of each group. Second, we study the convergence rate of the SWPM with the new and existing reduction schemes. In particular, we will present results showing the rate at which the scalar fourth order moment converges to its true value as the number of stochastic particles increases. We compare the results of our new reduction scheme with the existing deterministic schemes to show that our scheme requires a fewer initial number of computational particles and less computational time for the convergence of the scalar fourth order moment compared to the existing reduction schemes. 

In \cref{sec:SWPM}, we review the stochastic weighted particle method, and in \cref{sec:reduction method}, we discuss the reduction schemes of Rjasanow and Wagner and introduce our new reduction scheme. In \cref{sec:convergence}, we briefly show that the assumptions in Wagner's convergence theorem hold for our new particle reduction scheme. In \cref{sec:results}, we present  our numerical results, and finally in \cref{sec:conclusions} we make some conclusions.

\section{The stochastic weighted particle method}
\label{sec:SWPM}
In this section, we review the stochastic weighted particle method for the spatially homogeneous Boltzmann equation. The stochastic weighted particle method is a particle method \cite{rjasanow1996stochastic} that improves upon Bird's DSMC method \cite{bird1994molecular} by decreasing the uncertainty 
in the computation of rare events.
In Bird's method, each stochastic particle represents the same number of physical particles, and the  number of stochastic particles is kept constant throughout the simulation. With the SWPM, the number of physical particles represented by a single stochastic particle varies over the course of the simulation. Each stochastic particle represents a group of physical particles that are in close proximity in phase space. Each stochastic particle is characterized by its velocity and weight. The weight quantifies the proportion of physical particles represented by the given stochastic particle. The SWPM is based on a generalized version of the collision process used in the DSMC method in which only the physical particles corresponding to some portion of the weights of the colliding stochastic particles undergo collisions. For each stochastic collision, this results in the creation of two new stochastic particles whose velocities are given by the post-collision velocities and whose weights quantify the proportion of physical particles involved in the collision process   \cite{rjasanow1998reduction,rjasanow2005stochastic}. The weights of the original pair of colliding stochastic particles are reduced so as to keep the total weight of the system constant. As the number of stochastic particles increases due to collisions, the number of high velocity particles in the low probability tails of the velocity pdf also increases. By periodically applying a clustering technique and a particle reduction scheme, the proportion of particles in the center of the distribution is reduced. The combined effect of these processes is to increase the fraction of stochastic particles occupying the low probability tails of the velocity pdf, which decreases the statistical uncertainty in the tails. 

	We consider the spatially homogeneous Boltzmann equation for a single species of particles with unit mass. This equation, which describes the evolution of the velocity probability density function (pdf), $f$, due to collisions, is given by
\begin{equation} \label{spatially homogeneous}
\frac{\partial f}{\partial t}(\mathbf{v},t)\;=\;\displaystyle\int_{{\mathbb{R}}^3}\,\int_{{S}^2}B(\mathbf{v},\mathbf{w},\mathbf{\Theta})\,\big[(f(\mathbf{v'},t)\,f(\mathbf{w'},t)\,-\,f(\mathbf{v},t)\,f(\mathbf{w},t)\big]\,d\mathbf{\Theta} \, d\mathbf{w},
\end{equation}
with an initial condition of the form
\begin{equation}
f_{0}(\mathbf{v})\;=\;f(\mathbf{v},0).
\end{equation}
Here, 
$S^2$ denotes the unit sphere,
 $B$ is the collision kernel, $t$ is time, $\mathbf{v}$ and $\mathbf{w}$  are the 
 pre-collision velocities, and $\mathbf{v}'$ and $\mathbf{w}'$ are the post-collision velocities.
  For simplicity, for the results in this paper we consider isotropic Maxwell type interactions, for which 
\begin{equation}\label{eq: collision kernel}
B(\mathbf{v},\mathbf{w},\mathbf{\Theta})\;=\;\frac{1}{4\pi}.
\end{equation}
Assuming that the collisions are elastic, the post-collision velocities are given in terms of the pre-collision velocities and the direction vector, $\mathbf{\Theta}$, by
\begin{equation} \label{eq: post coll}
\mathbf{v}'\;=\;\frac{1}{2}\,\big[\mathbf{v}+\mathbf{w}-\mathbf{\Theta}\,|\mathbf{w}-\mathbf{v}|\,\big] \quad \text{and} \quad
\mathbf{w}'\;=\;\frac{1}{2}\,\big[\mathbf{v}+\mathbf{w}+\mathbf{\Theta}\,|\mathbf{w}-\mathbf{v}|\,\big]. 
\end{equation}

The state of the $i$th stochastic particle is given by $(\mathbf{v}_i,g_i )$, where $\mathbf{v}_i$ and $g_i$ are the velocity and weight, respectively. The state of the entire stochastic system is 
\begin{equation}
z\;=\;\{(g_1,\mathbf{v}_1),(g_2,\mathbf{v}_2),\ldots,(g_m,\mathbf{v}_m)\},
\end{equation}
where $m$ is the current number of stochastic particles. To  model a collision between the stochastic particles indexed by $i$ and $j$, we introduce the weight transfer function, $\gamma_{\text{coll}}(z;i,j)$. This function encodes the proportion of physical particles represented by the stochastic particles indexed by $i$ and $j$ that undergo collisions when the state of the system is $z$. The weight transfer function cannot exceed the minimum of the weights of the colliding particles,
\begin{equation}\label{eq: weight transfer function}
0\;\leq \; \gamma_{\text{coll}}(z;i,j)\;\leq \; \text{min}(g_i, g_j).
\end{equation}

During a collision between the stochastic particles indexed by $i$ and $j$, only the fraction of physical particles in the system represented by the weight $\gamma_{\text{coll}}(z;i,j)$ undergo collisions. This process is modeled by adding one or two new stochastic particles to the system. For the results in this paper, we use $\gamma_{\text{coll}}(z;i,j)=\frac{1}{2}\,\text{min}\,(g_i,g_j)$, which always results in two new stochastic particles. In this case, the state, $[J_{\text{coll}}(z;i,j,\mathbf{\Theta})]_k$, of the $k$-th stochastic particle after a collision between particles $i$ and $j$ is given by \cite{rjasanow1998reduction,rjasanow2005stochastic},
\begin{equation} \label{eq: weight transfer process for collision}
[J_{\text{coll}}(z;i,j,\mathbf{\Theta})]_k\;=\;\left\{\begin{aligned}
&(\mathbf{v}_k,g_k),   &&\text{if} \ k \leq m, \ k \notin\{i, j\}, \\
&(\mathbf{v}_i,g_i-\gamma_{\text{coll}}(z;i,j)), &&\text{if} \ k\,=\,i, \\
&(\mathbf{v}_j,g_j-\gamma_{\text{coll}}(z;i,j)), &&\text{if} \ k\,=\,j, \\
&(\mathbf{v}'_i,\gamma_{\text{coll}}(z;i,j)),  &&\text{if} \ k\,=\,m+1, \\
&(\mathbf{v}'_j,\gamma_{\text{coll}}(z;i,j)), &&\text{if} \ k\,=\,m+2, \\
\end{aligned}
\right.
\end{equation}
resulting in a new system state,
\begin{equation}
z\;=\;\{(g_1,\mathbf{v}_1),(g_2,\mathbf{v}_2),\ldots, (g_{m+1},\mathbf{v}_{m+1}),(g_{m+2},\mathbf{v}_{m+2})\}.
\end{equation}

After the collision, the fraction of physical particles corresponding to the weight $\gamma_{\text{coll}}(z;i,j)$ are assigned the post-collision velocities, and the remaining fraction of particles is unchanged. The two new stochastic particles are indexed by $m+1$ and $m+2$. To keep the total weight constant this weight is subtracted from the weights of the colliding stochastic particles, indexed by $i$ and $j$. For elastic collisions given by \cref{eq: post coll}, this stochastic collision process conserves the total weight, momentum and energy. 

To correctly model the evolution of the velocity pdf, we must relate the collision frequency for the stochastic system to that of the physical system. 
The total collision frequency in the physical system is given by 
\begin{equation} \label{eq: coll freq}
\nu\;=\;\displaystyle\int_{{\mathbb{R}}^3}\,\int_{{\mathbb{R}}^3}\,\int_{{S}^2}B(\mathbf{v},\mathbf{w},\mathbf{\Theta})\,f(\mathbf{v},t)\,f(\mathbf{w},t)\,d\mathbf{\Theta} \, d\mathbf{w}\,d\mathbf{v}.
\end{equation}
If we let $\nu_{g_ig_j}$ denote the frequency of collisions between the physical particles that correspond to stochastic particles with states $(\mathbf{v}_i, g_i)$ and $(\mathbf{v}_j, g_j)$, then by \cref{eq: coll freq} we obtain
\begin{equation}\label{eq: act coll freq}
\nu_{g_i g_j}\;=\;g_i\,g_j\,\displaystyle \int_{{S}^2} B(\mathbf{v}_i,\mathbf{v}_j,\mathbf{\Theta})\,d\mathbf{\Theta}.
\end{equation}
Furthermore, if we let $\widetilde\nu_{ij}$ denote the frequency of collisions between particles $i$ and $j$ in the stochastic system, then by the definition of the weight transfer function, we have that
\begin{equation}\label{eq: coll freq ratio}
\widetilde\nu_{ij} \, \gamma_\text{coll}(z;i,j)\;=\;\nu_{g_i g_j}.
\end{equation}
Therefore by \cref{eq: act coll freq}, we obtain
\begin{equation}\label{eq: stoch coll freq}
\widetilde\nu_{ij}\;=\;\frac{g_i\,g_j}{ \gamma_\text{coll}(z;i,j)}\,\displaystyle \int_{{S}^2} B(\mathbf{v}_i,\mathbf{v}_j,\mathbf{\Theta})\,d\mathbf{\Theta},
\end{equation}
and so, by \cref{eq: coll freq,eq: stoch coll freq}, the total collision frequency in the stochastic system is given by
\begin{equation}\label{eq: total collision frequency}
\widetilde\nu(z)\;=\;\frac{1}{2}\, \displaystyle \sum_{i=1}^m\,\sum_{\substack{j=1 \\ j \neq i}}^m \frac{g_i \  g_j }{\gamma_{\text{coll}}(z;i,j)}  \displaystyle \int_{\mathbf{\Theta} \in \mathbf{S^2}} B(\mathbf{v}_i,\mathbf{v}_j,\mathbf{\Theta})\, d\mathbf{\Theta}.
\end{equation}

 Using this frequency, we can obtain the waiting time between stochastic collisions. Since it is memoryless, this waiting time is a Poisson process which follows an exponential distribution. Therefore, the probability that a collision did not occur by time, $t$, is given by the survival function, $P(s>t)\,=\,e^{-\widetilde{\nu}(z)t}$.
Since the survival function has an uniform distribution on $[0,1]$, the time between collisions is given by $\label{eq: time}
\Delta t\,=\,-{\ln(r)}/{\widetilde{\nu}(z)}$, where $r$ is a random number uniformly distributed on $[0,1]$. Once the time interval between collisions is calculated, the time counter is updated.

The probability, $p(z;i,j)$, of a collision between the stochastic particles $i$ and $j$ is given by the ratio of the frequency $\widetilde\nu_{ij}$ of collisions between stochastic particles $i$ and $j$  given in \cref{eq: stoch coll freq} and the total collision frequency given in \cref{eq: total collision frequency}, that is, 
\begin{equation} \label{eq: pair probability}
p(z;k,l)\;=\;\frac{ \frac{g_k \  g_l }{\gamma_{\text{coll}}(z;k,l)}  \displaystyle \int_{\mathbf{\Theta} \in \mathbf{S^2}} B(\mathbf{v}_k,\mathbf{v}_l,\mathbf{\Theta})\, d\mathbf{\Theta}}{\displaystyle \sum_{i=1}^m\,\sum_{\substack{j=1 \\ j \neq i}}^m \frac{g_i \  g_j }{\gamma_{\text{coll}}(z;i,j)}  \displaystyle \int_{\mathbf{\Theta} \in \mathbf{S^2}} B(\mathbf{v}_i,\mathbf{v}_j,\mathbf{\Theta})\, d\mathbf{\Theta}}.
\end{equation}
Once a pair of colliding stochastic particles, $k$ and $l$, has been randomly selected,
the direction vector $\mathbf{\Theta}$ is chosen using the probability
density function
\begin{equation}
\eta(\mathbf{\Theta})=\frac{B(\mathbf{v}_k,\mathbf{v}_l,\mathbf{\Theta})}{\displaystyle \int_{\widetilde{\mathbf{\Theta}} \in {S^2}}B(\mathbf{v}_k,\mathbf{v}_l,\widetilde{\mathbf{\Theta}})\, \ d\widetilde{\mathbf{\Theta}}}.
\end{equation}
In the case of the constant collision kernel given by  \cref{eq: collision kernel}, the probability in \cref{eq: pair probability} further simplifies to 
 \begin{equation} \label{eq: simplified pair probability}
p(z;k,l)\;=\;\frac{ \frac{g_k \  g_l }{\gamma_{\text{coll}}(z;k,l)}}{\displaystyle \sum_{i=1}^m\,\sum_{\substack{j=1 \\ j \neq i}}^m \frac{g_i \  g_j }{\gamma_{\text{coll}}(z;i,j)}},
\end{equation}
and $\eta(\mathbf{\Theta})=\frac{1}{4\,\pi}$.
After the colliding pair of stochastic particles and the direction vector 
have been chosen, the velocities and weights of the colliding particles are updated
using \eqref{eq: weight transfer process for collision}.

The computation of the collision frequency in \cref{eq: total collision frequency} and collision probability in \cref{eq: pair probability} can be computationally expensive since in general these quantities need to be updated after each collision. This issue also arises for the DSMC method when the collision kernel is not constant. To overcome this computational issue, the technique of null collisions was developed by Koura \cite{koura1986null} for the DSMC method. With this technique, an equal maximum collision frequency is assigned to all pairs of particles, which leads to an equal probability of collision for all pairs. Consequently the colliding pair can be selected at random from a uniform distribution.  Once a pair is chosen, we decide whether the collision is an actual one or a null collision based on the probability given by the ratio between the actual collision frequency and the assigned equal collision frequency. Rjasanow and Wagner generalized the technique of the null collisions to the SWPM \cite{rjasanow1998reduction,rjasanow2005stochastic}.

\section{A reduction scheme conserving total weight, momentum, pressure tensor and heat flux}
\label{sec:reduction method}

As we explained in \cref{sec:intro}, one of the challenges for the SWPM is that the number of stochastic particles gradually increases. For computational feasibility, it is necessary to periodically reduce the number of particles. There are two steps in the reduction process. First, the stochastic particles need to be clustered into groups, and then each group of particles needs to be replaced by a small number of particles.

 A number of clustering techniques have been proposed by Rjasanow, Wagner and their collaborators \cite{matheis2003convergence,rjasanow1998reduction,rjasanow2005stochastic}. One of these techniques is based on partitioning particles into two groups with a cutting plane whose normal vector is in the direction of the eigenvector corresponding to the largest eigenvalue of the covariance matrix of the particles \cite{rjasanow1998reduction,rjasanow2005stochastic}. This partitioning method is performed iteratively on each of the partitioned groups using the group's covariance matrix. The iteration continues until the product of the total weight and the standard deviation of the particle speeds within each group is minimized, which results in a roughly uniform number of stochastic particles in each group. We use this clustering method for the results in this paper. Rjasanow and Wagner also proposed several stochastic and deterministic particle reduction schemes to replace each group by a group with a small number of particles. These schemes are based on conserving a specific set of moments of the distribution within each group. The details of these reduction schemes can be found in \cite{rjasanow1998reduction,rjasanow2005stochastic}.
  
We are interested in deterministic reduction schemes that conserve as many moments as possible, so that the structure of the velocity pdf is preserved.  
In this paper, we propose a particle reduction scheme that conserves all the moments of the velocity pdf  up to second order, given in \cref{tab: Moments}, together with the raw and central heat flux, which are the most physically relevant third order moments. The raw heat flux vector, $\mathbf{h}$, is computed relative to the origin, while the central heat flux, $\mathbf{q}$, is relative to the drift velocity, $\mathbf{V}$. They are given by
 \begin{equation}\label{eq: Total raw and central heat flux}
\mathbf{h}\;=\;\frac{1}{2}\, \displaystyle \sum_{i=1}^{m}\,g_i\,\mathbf{v}_i\,|\mathbf{v}_i|^2 \quad \text{and} \quad \mathbf{q}\;=\;\frac{1}{2}\, \displaystyle \sum_{i=1}^{m}\,g_i\,(\mathbf{v}_i-\mathbf{V})\,|\mathbf{v}_i-\mathbf{V}|^2.
\end{equation}
In the following discussion, when we refer to third order moments we simply mean the raw and central heat flux.
   \begin{table}[!htbp]
    {\footnotesize
    \caption{Moments of the velocity pdf}
    \label{tab: Moments}
    \begin{center}
    \begin{tabular}{|c|c|c|c|c|} 
     \hline
     Moment order & Raw Moment& Symbol & Central Moment &Symbol  \\
        \hline     
        Zero & Total Weight & $\varrho$ & --- & ---\\  
       First  & Momentum & $\varrho\,\mathbf{V}$ &--- &---\\
       Second & Momentum Flux Tensor & $\Pi$ &Pressure Tensor & $P$\\
       \hline
    \end{tabular}
  \end{center}
  }
\end{table}    

The reduction scheme of Rjasanow and Wagner that preserves the most moments and is closest to our scheme is the one that preserves the total weight, momentum, energy, and central heat flux \cite{rjasanow1998reduction,rjasanow2005stochastic}. With this scheme, although the central heat flux is conserved within each group, the momentum flux tensor and the pressure tensor are not. Only the total energy, which is the trace of the momentum flux tensor, is conserved. As a consequence, the raw heat flux for each group is not conserved, and therefore, the raw and central heat flux of the entire system are also not conserved. 
 
To conserve both the raw and central moments of a group, it is necessary and sufficient to conserve either of these moments and all of the lower order moments. Because of the additive property of raw moments, if a raw moment is conserved within each group then it must also be conserved for the entire system. In particular, since the total weight and momentum are raw moments, conservation of $\varrho$ and $\varrho\,\mathbf{V}$ within each group ensures that these two moments are conserved for the entire system. Therefore, if we could conserve the total weight, momentum, pressure tensor and central heat flux within each group, then all of the raw and central moments up to the second order together with the raw and central heat flux would be  conserved for the entire system. 

We formalize this idea as follows. Let $m$ be the number of stochastic particles in the system,  and suppose that the particles have been partitioned into $\widehat{n}$ groups with $m_l$ stochastic particles in the $l$-th group, $G_l$. Let $g_{l,i}$ and $\mathbf{v}_{l,i}$ denote the weight and velocity of the $i$-th particle in the $l$-th group. Then, the total weight, $\varrho_l$, momentum, $\varrho_l\,\mathbf{V}_l$, momentum flux tensor, $\Pi_l$, and raw heat flux, $\mathbf{h}_l$, for the $l$-th group are given by
\begin{equation}\label{eq: Raw moments}
\begin{aligned}
\varrho_l &\;=\;\displaystyle \sum_{i=1}^{m_l} g_{l,i}, & \varrho_l\,\mathbf{V}_l&\;=\;\displaystyle \sum_{i=1}^{m_l} g_{l,i}\,\mathbf{v}_{l,i},   \\ \Pi_l&\;=\;\displaystyle \sum_{i=1}^{m_l}\,g_{l,i} \,\mathbf{v}_{l,i}\,\mathbf{v}_{l,i}^T, & \,\mathbf{h}_l&\;=\; \frac{1}{2}\,\displaystyle \sum_{i=1}^{m_l}\,g_{l,i}\,\mathbf{v}_{l,i}\,|\mathbf{v}_{l,i}|^2,
\end{aligned}
\end{equation}
where $\mathbf{V}_l$ is the drift velocity of the $l$-th group. 
The pressure tensor, $P_l$,  and the central heat flux, $\mathbf{q}_l$, of the $l$-th group are given by
\begin{equation} \label{eq: Pressure Tensor}
P_l\;=\;\displaystyle \sum_{i=1}^{m_l} g_{l,i}\,(\mathbf{v}_{l,i}-\mathbf{V}_l)\,(\mathbf{v}_{l,i}-\mathbf{V}_l)^T \quad \text{and} \quad \,\mathbf{q}_l\;=\;\frac{1}{2}\, \displaystyle \sum_{i=1}^{m_l}\, g_{l,i}\,\big(\mathbf{v}_{l,i}-\mathbf{V}_l \big)\,\big|\mathbf{v}_{l,i}-\mathbf{V}_l \big|^2.
\end{equation}
The energy, $E_l$, and temperature, $T_l$, are given by
\begin{equation} \label{eq: energy and temperature}
E_l\;=\;\displaystyle \sum_{i=1}^{m_l} g_{l,i}\,|\mathbf{v}_{l,i}|^2, \quad \text{and} \quad 3\,\varrho_l\,T_l\;=\;\displaystyle \sum_{i=1}^{m_l} g_{l,i}\,|\mathbf{v}_{l,i}-\mathbf{V}_l|^2,
\end{equation}
where the quantity on the right hand side  of the formula for $T_l$ is the trace of the pressure tensor. The energy is given in terms of the temperature by
\begin{equation}
E_l\;=\;\varrho_l\,|\mathbf{V}_l|^2+3\,\varrho_l\,T_l.
\end{equation}
 The raw moments of the entire system are given by
 \begin{equation} \label{eq: Raw moments system}
 \begin{aligned}
 \varrho\;=\;\displaystyle \sum_{l=1}^{\hat{n}} \varrho_l, \quad \varrho\,\mathbf{V}\;=\;\displaystyle \sum_{l=1}^{\hat{n}}\varrho_l\,\mathbf{V}_l,   \quad \Pi&\;=\;\displaystyle \sum_{l=1}^{\hat{n}} \Pi_l, \quad \text{and} \quad \mathbf{h}\;=\;\displaystyle \sum_{l=1}^{\hat{n}} \mathbf{h}_l.
 \end{aligned}
 \end{equation}
Here $\varrho$ is the total weight, $\mathbf{V}$ is the drift velocity, $\Pi$ is the momentum flux tensor and $\mathbf{h}$ is the raw heat flux of the entire system. 

The relationship between the raw and central second order moments is given by 
 \begin{equation} \label{eq: raw and central second order moment}
 P_l\;=\;\Pi_l- \varrho_l\,\mathbf{V}_l\,\mathbf{V}_l^T.
 \end{equation}
Using this relationship, we observe that for a reduction scheme to preserve both of the second order moments, $P_l$ and $\Pi_l$, it is sufficient to conserve the total weight, $\varrho_l$, the momentum, $\varrho_l\,\mathbf{V}_l$, and either $P_l$ or $\Pi_l$. Since the raw moments are additive (see \cref{eq: Raw moments system}), conservation of the total weight, momentum and momentum flux tensor within each group leads to the conservation of the these moments for the entire system. Using \cref{eq: raw and central second order moment} for the entire system, we conclude that the pressure tensor for the entire system is also conserved. Therefore, a reduction scheme that conserves the total weight, momentum and either of the second order moments for each group leads to the conservation of the moments up to second order for the entire system. 

Similarly, the relationship between the raw and central third order moment can be determined using \cref{eq: Raw moments,eq: energy and temperature,eq: raw and central second order moment} giving the equation,
\begin{equation}\label{eq: raw and central third order moment}
\begin{aligned}
\mathbf{q}_l\;=\;& \mathbf{h}_l-P_l\,\mathbf{V}_l-\frac{1}{2}\,\varrho_l\,\mathbf{V}_l\,\big|\mathbf{V}_l \big|^2-\frac{3}{2}\,\varrho_l\,T_l\,\mathbf{V}_l, 
\end{aligned}
\end{equation}
which relates the raw and central heat flux to each other via the lower order moments. As above, to conserve both the raw and central moments of a group up to third order it is sufficient to conserve the total weight, momentum, pressure tensor, and either of the third order moments of the group. Similarly, using the additivity property of the raw moments, and the relationships between the moments given by \cref{eq: raw and central second order moment} and  \cref{eq: raw and central third order moment} for the entire system, the pressure tensor and central heat flux of the system are also conserved together with all the raw moments. This verifies our claim that to conserve the raw and central moments of the system up to the third order during a reduction process, it is sufficient to conserve the total weight, momentum, pressure tensor and central heat flux of each group.

 Next, we present a novel particle reduction scheme that conserves the total weight, momentum, pressure tensor, and central heat flux in a group.
First, we outline the idea behind the conservation of these moments. Before describing this scheme, we briefly recall that for the reduction scheme that preserves the total weight and momentum of a group, we simply replace all the stochastic particles in the group by a single stochastic particle with the given weight and momentum \cite{rjasanow2005stochastic}. The next higher order moments are the momentum flux tensor and the pressure tensor. Since, the pressure tensor is a $3 \times 3$ real symmetric positive semi-definite matrix, it can be diagonalized using an orthonormal basis of normalized eigenvectors, with the non-negative eigenvalues as its diagonal entries. This simplifies the problem, as we only have to conserve the diagonal entries of the pressure tensor. To conserve the pressure tensor in this new orthonormal basis, each group can be replaced by a group with between one and three pairs of particles. The number of pairs of particles depends on the number of nonzero eigenvalues. Specifically, we choose to assign an equal portion of the total weight to each pair of particles. For each pair, the velocity of one of the particles relative to the drift velocity is chosen to be in the direction of an eigenvector with nonzero eigenvalue, while the other particle moves in the opposite direction. The magnitudes of these velocity pairs relative to the drift velocity of the group are equal, and  are chosen to ensure the conservation of each of the diagonal entries of the pressure tensor in the new basis, which leads to the conservation of the pressure tensor. The total weight and momentum of the group are conserved as a consequence of the choices we made.

To additionally conserve the raw and central heat flux, we utilize a degree of freedom in the choice of weights and in the magnitudes of the velocities relative to the drift velocity. We choose the sum of the weights of the particles for each pair to be an equal portion of the total weight. For each pair, the weights of the two particles and the magnitudes of their velocities relative to the drift velocity are not required to be equal. These quantities are determined by solving the conditions required to conserve the weight, momentum, pressure tensor, and central heat flux in the new basis. Once the post reduction velocities are determined, the transformation of these velocities to the standard basis leads to the conservation of the moments in the standard basis. 
 
 The following theorem summarizes our new particle reduction scheme for the conservation of the total weight, momentum, pressure tensor, and central heat flux of a group.
 \begin{theorem} \label{T: Theorem1}
 Let $G_l$ be a group of stochastic particles. Suppose that the pressure tensor, $P_l$, has $k$ non-zero eigenvalues, $\lambda_1, \ldots, \lambda_k$, for some $k \in \{1,2,3\}$, and an associated orthonormal set of eigenvectors, $\mathbf{\Theta}_1, \ldots, \mathbf{\Theta}_k$, with the direction of $\mathbf{\Theta}_i$ chosen so that $\widehat{q}_{l,i}=\mathbf{\Theta}_i^T\,\mathbf{q}_l>0$. Let $\widetilde{G}_l$ be the reduced group of $2k$ stochastic particles whose weights and velocities, $(\widetilde{\mathbf{v}}_i,\widetilde{g}_i)$, for $i=1, \ldots, 2k$ are given by
\begin{equation} \label{eq: reduced particles}
\begin{aligned}
 \widetilde{\mathbf{v}}_i &\;=\;\mathbf{V}_l+\gamma_i \,\sqrt{\frac{k\,\lambda_i}{\varrho_l}}\mathbf{\Theta}_i, \quad &\widetilde{g}_i&\;=\;\frac{\varrho_l}{k}\,\frac{1}{1+\gamma_i^2}, \\
 \widetilde{\mathbf{v}}_{i+k}&\;=\;\mathbf{V}_l-\frac{1}{\gamma_i} \,\sqrt{\frac{k\,\lambda_i}{\varrho_l}}\mathbf{\Theta}_i, \quad &\widetilde{g}_{i+k}&\;=\;\frac{\varrho_l}{k}\,\frac{\gamma_i^2}{1+\gamma_i^2}, \quad \text{for} \quad i\;=\;1, \ldots, k
\end{aligned}
\end{equation}
where, 
\begin{equation} \label{eq: gamma parameter}
\gamma_i\;=\;\frac{\sqrt{\varrho_l}\,\widehat{q}_{l,i}}{\sqrt{k}\,\lambda_i^{\frac{3}{2}}} +\sqrt{1+\frac{\varrho_l\,\widehat{q}_{l,i}^2}{k\,\lambda_i^3}}, \quad \text{for} \quad  i\,=\,1, \ldots , k.
\end{equation}
Then, $\widetilde{G}_l$ preserves the total weight, momentum, pressure tensor, and central heat flux of $G_l$, which leads to the preservation of all the moments up to the second order as well as the raw and central heat flux of $G_l$.

 \end{theorem}
 
\begin{proof}
 We consider the case where the eigenvalues of the pressure tensor are all nonzero. We let the reduced group, $\widetilde{G}_l$, consist of three pairs of particles, with each pair of the form,
\begin{equation} \label{eq: pairs assumption}
\begin{aligned}
&\widetilde{\mathbf{v}}_i\;=\;\mathbf{V}_l+\alpha_i\,\mathbf{\Theta}_i,  \quad \widetilde{\mathbf{v}}_{i+3}\;=\;\mathbf{V}_l-\alpha_{i+3}\,\mathbf{\Theta}_i, \quad  \text{and} \quad \widetilde{g}_i+\widetilde{g}_{i+3}\;=\;\frac{\varrho_l}{3}, \quad \! i \in \{1,2,3 \},
\end{aligned}
\end{equation}
for some $\alpha_i \in \mathbb{R}$ and $\mathbf{\Theta}_i \in S^2$. We derive the conditions on the unknown parameters,  $\widetilde{g}_i$, $\alpha_i$, and $\mathbf{\Theta}_i$, so as to conserve the total weight, momentum, pressure tensor and heat flux. 

By construction, the total weight is conserved, 
\begin{equation}\label{eq: zero moment}
\displaystyle \sum_{i=1}^{6}\,\widetilde{g}_i\;=\;\varrho_l.
\end{equation}
Similarly, if we impose the condition
\begin{equation} \label{eq: alpha}
\widetilde{g}_i\,\alpha_i\;=\;\widetilde{g}_{i+3}\,\alpha_{i+3}, \quad \text{for} \quad i\in \{1, 2, 3\},
\end{equation}
we find that the momentum of the group is conserved, since
\begin{equation} \label{eq: first moment}
\displaystyle \sum_{i=1}^{6}\,\widetilde{g}_i\,\widetilde{\mathbf{v}}_i\;= \;\sum_{i=1}^3\widetilde{g}_i(\mathbf{V}_l+\alpha_i\,\mathbf{\Theta}_i)+\widetilde{g}_{i+3}(\mathbf{V}_l-\alpha_{i+3}\,\mathbf{\Theta}_i)
=  \varrho_l\,\mathbf{V}_l.
\end{equation}

Next, to conserve the pressure tensor, $P_l$, we use the fact that it is a $3\times3$ real symmetric matrix with positive eigenvalues. Therefore, there is a diagonal matrix $D=\diag[\lambda_1, \lambda_2, \lambda_3]$ and an orthonormal matrix $Q=[\mathbf{\Theta_1},\mathbf{\Theta_2},\mathbf{\Theta_3}]$ such that
\begin{equation}
D\,=\,Q^T\,P_l\,Q.
\end{equation}
That is, each $\{\lambda_i,\mathbf{\Theta}_i \}$ is an eigenpair of the matrix $P_l$.
The condition, $\widetilde{P}_l\,=\,P_l$, that the reduction scheme preserves the pressure tensor is therefore equivalent to the condition 
\begin{equation}
\begin{aligned}
D\;=&\;Q^{T}\,\Big[\displaystyle \sum_{i=1}^6 \widetilde{g}_i\,\big(\widetilde{\mathbf{v}}_i -\mathbf{V}_l \big) \big(\widetilde{\mathbf{v}}_i-\mathbf{V}_l \big)^{T} \Big]\,Q \\
=&\;\displaystyle \sum_{i=1}^3\big(\widetilde{g}_i\,{\alpha_i}^2+\widetilde{g}_{i+3}\,\alpha_{i+3}^2 \big)\,\big(Q^T\,\mathbf{\Theta}_i\big)\,\big(Q^T\,\mathbf{\Theta}_i\big)^T.
\end{aligned}
\end{equation}
Therefore, to conserve the pressure tensor, we require that
\begin{equation} \label{eq: alpha square}
\widetilde{g}_i\,{\alpha_i}^2+\widetilde{g}_{i+3}\,\alpha_{i+3}^2\;=\;\lambda_i, \quad \text{for} \quad i \in \{1, 2, 3\}.
\end{equation}

In the basis of eigenvectors, the central heat flux, $\widehat{\mathbf{q}}_l$, is given by
\begin{equation}
\widehat{\mathbf{q}}_l\;=\;Q^T\,\mathbf{q}_l.
\end{equation}
As in the statement of the theorem, we choose the direction of $\mathbf{\Theta}_i$ so that the $i$-th component, $\widehat{q}_{l,i}\,=\,\mathbf{\Theta}_i^T\, \mathbf{q}_l$, of $\widehat{\mathbf{q}}_l$ is positive. To conserve the central heat flux in the new basis, we have
\begin{equation}
\widehat{\mathbf{q}}_l\;=\;\frac{1}{2} \displaystyle \sum_{i=1}^3 \big(\widetilde{g}_i\,\alpha_i^3-\widetilde{g}_{i+3}\,\alpha_{i+3}^3\big)\,Q^T \mathbf{\Theta}_i \;=\;\frac{1}{2} \displaystyle \sum_{i=1}^3 \big(\widetilde{g}_i\,\alpha_i^3-\widetilde{g}_{i+3}\,\alpha_{i+3}^3\big)\,\mathbf{e}_i,
\end{equation}
and we obtain
\begin{equation} \label{eq: alpha cube}
\widehat{q}_{l,i}\;=\;\frac{1}{2}\big[\widetilde{g}_i\,\alpha_i^3-\widetilde{g}_{i+3}\,\alpha_{i+3}^3\big].
\end{equation}

Next, to solve for $\alpha_i$ and $\widetilde{g}_i$, we apply a technique used by Rjasanow and Wagner \cite{rjasanow1998reduction, rjasanow2005stochastic} . We introduce a new parameter, $\gamma_i$, and express $\alpha_i$ as 
\begin{equation} \label{eq: alpha 1}
\alpha_i\;=\;\gamma_i\,\sqrt{\frac{3\,\lambda_i}{\varrho_l}}.
\end{equation}
Substituting \cref{eq: alpha} into \cref{eq: alpha square} we obtain 
\begin{equation}
\lambda_i\;=\;\frac{\widetilde{g}_i}{\widetilde{g}_{i+3}}\,\alpha_i^2\,\Big[\widetilde{g}_{i+3}+\widetilde{g}_i \Big]. 
\end{equation}
Substituting the expression for $\alpha_i$ given by \cref{eq: alpha 1}, and using \cref{eq: pairs assumption} for the sum of weights, we obtain
\begin{equation}
\frac{\widetilde{g}_i}{\widetilde{g}_{i+3}}\,\gamma_i^2\;=\;1.
\end{equation}
Using this relationship in \cref{eq: pairs assumption}, we obtain
\begin{equation}\label{eq: g14 and alpha14}
\begin{aligned}
\widetilde{g}_i\;=\;\frac{\varrho_l}{3}\,\frac{1}{1+\gamma_i^2}, \quad \alpha_i\;=\;\gamma_i\,\sqrt{\frac{3\,\lambda_i}{\varrho_l}}, \quad \text{and} \quad
\widetilde{g}_{i+3}\;=\;\frac{\varrho_l}{3}\,\frac{\gamma_i^2}{1+\gamma_i^2}, \quad \alpha_{i+3}\;=\;\frac{1}{\gamma_i}\,\sqrt{\frac{3\,\lambda_i}{\varrho_l}}.
\end{aligned}
\end{equation}

To determine $\gamma_i$, we substitute \cref{eq: g14 and alpha14} into \cref{eq: alpha cube} to obtain
\begin{equation}
\widetilde{g}_i\,\alpha_i^3-\widetilde{g}_{i+3}\,\alpha_{i+3}^3\;=\;\sqrt{\frac{3}{\varrho_l}}\,\frac{\lambda_i^{\frac{3}{2}}}{\gamma_i}\, \big(\gamma_i^2-1\big).
\end{equation}
Now by  \cref{eq: alpha cube}, $\widehat{q}_{l,i}\,=\,\frac{1}{2} \big[\widetilde{g}_i\,\alpha_i^3-\widetilde{g}_{i+3}\,\alpha_{i+3}^3 \big] >0$. Therefore $\gamma_i>1$ and
\begin{equation}
\widehat{q}_{l,i}\;=\;\frac{1}{2}\sqrt{\frac{3}{\varrho_l}}\,\frac{\lambda_i^{\frac{3}{2}}}{\gamma_i}\, \big(\gamma_i^2-1\big)\; \implies \;\gamma_i^2-2\,\frac{\sqrt{\varrho_l}\,\widetilde{q}_i}{\sqrt{3}\,\lambda_i^{\frac{3}{2}}}\,\gamma_i-1\;=\;0.
\end{equation}
Solving for the positive root, we obtain
\begin{equation} \label{eq: gamma}
\gamma_i\;=\;\frac{\sqrt{\varrho_l}\,\widehat{q}_{l,i}}{\sqrt{3}\,\lambda_i^{\frac{3}{2}}} +\sqrt{1+\frac{\varrho_l\,\widehat{q}_{l,i}^2}{3\,\lambda_i^3}}.
\end{equation}
Therefore, the post reduction particles are given by \cref{eq: reduced particles} and \cref{eq: gamma parameter} as required.

If the pressure tensor has at least one zero eigenvalue, the moments can be conserved with fewer than six particles. The reason is that there is no need to introduce particles whose heat flux is in the direction of the eigenvectors corresponding to the zero eigenvalues. In this situation, the result follows similarly to the calculations above.	
\end{proof}

\section{Theoretical convergence of SWPM with the new reduction scheme}
\label{sec:convergence}
In this section, we show that our new reduction scheme satisfies the assumptions in Wagner's convergence theorem for the SWPM \cite[Thm. (3.22)]{rjasanow2005stochastic}. This theorem provides a collection of assumptions which guarantee that the sequence of empirical measures of the Markov process produced by the SWPM converges to the weak solution of the Boltzmann equation as $n \to \infty$. These assumptions on the reduction scheme are given by \cite[eq. (3.162)]{rjasanow2005stochastic}, and \cite[eq. (3.164)]{rjasanow2005stochastic}.  
According to Rjasanow and Wagner, assumption \cite[eq. (3.162)]{rjasanow2005stochastic} assures that the reduction is sufficiently precise, and \cite[eq. (3.164)]{rjasanow2005stochastic} restricts the increase in energy during reduction. Since the energy is conserved in our new reduction scheme, the second assumption related to the energy is satisfied. For assumption \cite[eq. (3.162)]{rjasanow2005stochastic}, the arguments given by Rjasanaw and Wagner for their deterministic reduction schemes also apply to our new deterministic reduction scheme. Therefore, for this assumption to hold for our new reduction scheme, it is sufficient to show that the inequality given by \cite[eq. (3.273)]{rjasanow2005stochastic} holds. This inequality states that
\begin{equation}\label{eq:inequality}
\bigg|\displaystyle \int_{Z_l} \Phi(\widetilde{z}_l)\,p_{\text{red}}(z_l;d\widetilde{z}_l)-\Phi(z_l)\bigg|\;\leq \;{||\varphi||}_L\bigg[\sum_{i=1}^{m_l}g_i\,\big|\mathbf{V}_l-\mathbf{v}_i\big|+\varrho_l\,\sqrt{3\,T_l}\bigg].
\end{equation}
Here,
\begin{equation}
z_l\;=\;\{(g_1,\mathbf{v}_1),(g_2,\mathbf{v}_2),\ldots, (g_{m_l},\mathbf{v}_{m_l})\}
\end{equation}
is the state of a group $G_l$ prior to reduction, $\widetilde{z}_l$ is the post-reduction state, and $p_{\text{red}}(z_l;d\widetilde{z}_l)$ is a measure that gives the probability that the post-reduction state lie in the volume element, $d\widetilde{z}_l$. The function $\Phi$, which approximates the velocity pdf,  is given by
\begin{equation}
\Phi(z_l)\;= \;\sum_{i=1}^{m_l} g_i\,\varphi(\mathbf{v}_i),
\end{equation}
for the particles in the group $G_l$. Here $\varphi$ is an arbitrary test function. The norm for the test function, $||\varphi||_L$,  is defined as 
\begin{equation}\label{eq: test function norm}
 ||\varphi||_L\;=\;\max \bigg \{||\varphi||_\infty, \sup_{\mathbf{v}\,\neq\, \mathbf{w}\, \in \,\mathbb{R}^3} \frac{|\varphi(\mathbf{v})-\varphi(\mathbf{w})|}{|\mathbf{v}-\mathbf{w}|} \bigg\}.
\end{equation}
In the inequality \cref{eq:inequality}, $\displaystyle \int_{Z_l} \Phi(\widetilde{z}_l)\,p_{\text{red}}(z_l;d\widetilde{z}_l)$ gives the expectation of $\Phi$ for the reduced system. 

For our new deterministic reduction scheme, in the spatially homogeneous case, for each group only one state is possible after reduction. Therefore, in the case where all three eigenvalues of the pressure tensor are positive, $p_{\text{red}}(z_l;d\widetilde{z}_l)\,=\,\delta_{J_{\text{red}(z_l)}}(d\widetilde{z}_l)$, where ${[J_{\text{red}(z_l)}]}_i\,=\,(\widetilde{\mathbf{v}}_i(z_l),\widetilde{g}_i(z_l))$, for $i\,=\,1,\ldots,6$, is the post reduction state given by \cref{T: Theorem1}.
Therefore,
\begin{equation}
\displaystyle \int_{Z_l} \Phi(\widetilde{z}_l)\,p_{\text{red}}(z_l;d\widetilde{z}_l)\;=\;\Phi(J_{\text{red}}(\widetilde{z}_l)) 
\;=\; \sum_{j=1}^{6} \widetilde{g}_j\,\varphi(\widetilde{\mathbf{v}}_j).
\end{equation}
Since $\displaystyle \sum_{j=1}^6 \widetilde{g}_j\,=\,\varrho_l$, and applying the triangle inequality, we obtain
\begin{equation}
\begin{aligned}
\bigg|\displaystyle \int_{Z_l} \Phi(\widetilde{z}_l)\,p_{\text{red}}(z_l;d\widetilde{z}_l)-\Phi(z)\bigg|\,\,= & \,\, \bigg|\sum_{j=1}^{6} \widetilde{g}_j\,\varphi(\widetilde{\mathbf{v}}_j)- \sum_{i=1}^{m_l} g_i\,\varphi(\mathbf{v}_i)\bigg| \\
\leq&\,\,\sum_{j=1}^{6} \bigg| \frac{\widetilde{g}_j}{\varrho_l}\, \sum_{i=1}^{m_l} g_i\,\varphi(\widetilde{\mathbf{v}}_j)- \frac{\widetilde{g}_j}{\varrho_l} \sum_{i=1}^{m_l} g_i\,\varphi(\mathbf{v}_i)\bigg| \\
\leq&\,\,\sum_{j=1}^{6} \frac{\widetilde{g}_j}{\varrho_l}\, \sum_{i=1}^{m_l} g_i\,\bigg|\varphi(\widetilde{\mathbf{v}}_j)- \varphi(\mathbf{v}_i)\bigg| \\
\leq&\,\,{||\varphi||}_L\,\sum_{j=1}^{6} \frac{\widetilde{g}_j}{\varrho_l}\, \sum_{i=1}^{m_l} g_i\,\bigg|\widetilde{\mathbf{v}}_j- \mathbf{v}_i\bigg|,
\end{aligned}
\end{equation}
where the final inequality follows from \cref{eq: test function norm}.
Furthermore, using \cref{eq: pairs assumption}, the triangle inequality, and the fact that $\displaystyle \sum_{i=1}^{m_l} g_i\,=\,\varrho_l $, we obtain
\begin{equation}
\begin{aligned}
&\,\,\sum_{j=1}^{6}\frac{\widetilde{g}_j}{\varrho_l}\,\sum_{i=1}^{m_l} g_i\,\bigg|\widetilde{\mathbf{v}}_j- \mathbf{v}_i\bigg|\\
\leq& \,\,\sum_{i=1}^{m_l}g_i\,\big|\mathbf{V}_l-\mathbf{v}_i\big|+\sum_{j=1}^{6}\widetilde{g}_j\,\alpha_j  \\
=& \,\,\sum_{i=1}^{m_l}g_i\,\big|\mathbf{V}_l-\mathbf{v}_i\big|+\bigg(\sum_{j=1}^{6}\widetilde{g}_j^2\,\alpha_j^2+2\,\sum_{j=1}^6 \,\sum_{k>j} \widetilde{g}_j\,\widetilde{g}_k\,\alpha_j\,\alpha_k \bigg)^{\frac{1}{2}} \\
\leq& \,\,\sum_{i=1}^{m_l}g_i\,\big|\mathbf{V}_l-\mathbf{v}_i\big|+\bigg(\sum_{j=1}^{6}\widetilde{g}_j^2\,\alpha_j^2+\sum_{j=1}^6 \,\sum_{k>j} \widetilde{g}_j\,\widetilde{g}_k\,(\alpha_j^2+\alpha_k^2) \bigg)^{\frac{1}{2}}\\
 = &\,\,\sum_{i=1}^{m_l}g_i\,\big|\mathbf{V}_l-\mathbf{v}_i\big|+\bigg({\varrho}_l\,\sum_{j=1}^{6}\widetilde{g}_j\,\alpha_j^2\bigg)^{\frac{1}{2}} \\
= & \,\,\sum_{i=1}^{m_l}g_i\,\big|\mathbf{V}_l-\mathbf{v}_i\big|+\varrho_l\,\sqrt{3\,T_l}.
\end{aligned}
\end{equation}
 Here, the final equality is obtained from \cref{eq: energy and temperature} and \cref{eq: pairs assumption}. Therefore, the desired inequality \cref{eq:inequality} holds for our new reduction scheme, and Wagner's convergence theorem applies in this context.

\section{Numerical results}
\label{sec:results}
In this section, we discuss our numerical results. 
The algorithm was implemented in C++ and all simulations 
were performed on a desktop machine with a 3.6~GHz single processor.  
We verified that the total times for the particle collisions and
for the clustering and particle reductions both
scale linearly with the initial number of computational particles, $m_0$.
The time taken to simulate the  clustering and
particle reductions was  approximately four times larger 
than the time taken to simulate  the particle collisions. However, as we will show in \cref{table:timings}, for the results in \cref{ABseries,fig1: s,fig2: s} below, the
total computational time is only about 30 seconds for $N=100$ ensembles with 
$m_0 = 10,240$ particles per ensemble.

First, to numerically verify the conclusions of \cref{T: Theorem1}, we study the sum over all the groups of the reduction errors for the raw and central heat flux. For this study, we consider an initial Maxwellian distribution with temperature, $T=1$, and drift velocity, $\mathbf{V}=\langle 0,0,0 \rangle$. We used a single ensemble to obtain these results, and the initial number of 
computational particles
is chosen to be $m_0 =$ 10,240. Once the number of computational particles reaches $4m_0$, we reduce it to $\widetilde{m} \approx \frac{m_0}{4}$, which was the strategy that produced the largest errors for the deterministic reduction schemes of
Rjasanow and Wagner~\cite{rjasanow1998reduction}. We chose this strategy to demonstrate that our method performs well even under this condition.

 \begin{table}[!htbp]
    {\footnotesize
    \caption{Maximum average relative errors in the central and raw heat flux for the three reduction schemes.}
    \label{tab: table errors}
    \begin{center}
    \begin{tabular}{|c|L|L|} 
     \hline
      Reduction Scheme & Central Heat Flux Error & Raw Heat Flux Error \\
        \hline       
       Energy  &1 & 0.01743 \\
       Energy and Central Heat Flux (Ct. HF) & 3.00844e-15& 0.015072\\
       Pressure Tensor (PT) and Central Heat Flux (Ct. HF)  &2.12406e-15 &6.81888e-16 \\
       \hline
    \end{tabular}
  \end{center}
  }
\end{table}

 In \cref{tab: table errors}, we compare three reduction schemes. All three schemes conserve the total weight and momentum, and in addition to these moments, the reduction schemes conserve the moments associated to their names. The first two schemes, energy conservation, and energy and central heat flux conservation (Ct. HF), are the reduction schemes of Rjasanow and Wagner, and the third one is our reduction scheme which conserves the pressure tensor (PT) and central heat flux (Ct. HF). 
To compare these schemes, we compute the relative 2-norm errors for each of the third order moments of each group and take their average over all the groups, that is we let
\begin{equation}\label{eq:total error}
 \mathcal{E}\;=\;\frac{1}{\#\text{Grps}}\displaystyle \sum_{l=1}^{\#\text{Grps}}\frac{ {|| \mathfrak{m}_{\text{After},l}- \mathfrak{m}_{\text{Before},l}||}_{2}}{{||\mathfrak{m}_{\text{Before},l}||}_2}.
 \end{equation}
We obtained these average relative errors for the first ten reductions, and show the maximum of these errors in \cref{tab: table errors}. The errors for the pressure tensor and central heat flux scheme are smaller than $10^{-14}$, which is negligible. However, in the third column of the table, we observe that for the energy and central heat flux scheme the raw heat flux error is about $2\times 10^{13}$ times larger than that for the pressure tensor and central heat flux scheme. 
These results support the theory in \cref{sec:reduction method} that the energy and central heat flux scheme only conserves the central heat flux in each group, and does not conserve the raw heat flux, while the pressure tensor and central heat flux scheme conserves both. Furthermore, the energy scheme has the largest error for both third order moments. 
Since this scheme replaces a group by two particles with equal weights and opposite velocities relative to the drift velocity, the central heat flux of the group after reduction is zero. This observation explains why the relative error in the central heat flux is 1
for the energy reduction scheme.

   In \cite{rjasanow2005stochastic}, Rjasanow and Wagner observed that the higher order moments of a distribution are conserved statistically when averaged over a large number of ensembles, even if the reduction scheme only conserves the lower order moments. However, they found that the existing deterministic reduction schemes require a larger initial number of computational particles for the convergence of the scalar fourth order moment than for the lower moments. To examine this, for each reduction scheme we studied the convergence of (1,1)-component of the momentum flux tensor, $\Pi_{1,1}$, the second component of the raw heat flux, $\mathbf{h}_2$, and the scalar fourth order moment,  
 \begin{equation}
s\;=\;\displaystyle \sum_{i=0}^{m} g_i |\mathbf{v}_i|^4,
\end{equation}
as we increase $m_0$. For this study, we chose the initial condition to be a mixture of Maxwellian distributions, since for this pdf there is an analytical formula for the given moments as a function of time, $t$ \cite{rjasanow2005stochastic}. The initial distribution is given by
 \begin{equation}
 f_0 (\mathbf{v})\,=\,\alpha\,M_{\mathbf{V}_1,T_1}(\mathbf{v})+(1-\alpha)\,M_{\mathbf{V}_2,T_2}(\mathbf{v}), 
 \end{equation}     
where $M_{\mathbf{V}_1,T_1}(\mathbf{v})$ and $M_{\mathbf{V}_2,T_2}(\mathbf{v}) $ are Maxwellian distributions with drift velocities $\mathbf{V}_1$  and $\mathbf{V}_2$, and temperatures $T_1$ and $T_2$, respectively. We chose $\alpha\,=\,0.5$, $\mathbf{V}_1\,=\,\langle-2,2,0 \rangle$, $\mathbf{V}_2\,=\,\langle2,0,0 \rangle$, and $T_1\,=\,T_2\,=\,1$. We performed two sets of simulations in which we studied the short term (transient) 
behavior of $s$ in the time interval $[0,3]$. 
For this study, we calculated the relative error of the moments and the half-width of the $99.9\%$ confidence interval as a function of time in the interval $[0,3]$. The relative error for a moment $\mathfrak{m}$ is given by
\begin{equation}\label{eq:MomentError}
E\;=\; \frac{|\mathfrak{m}_{\text{anal}}-\overline{\mathfrak{m}}|}{|\mathfrak{m}_{\text{anal}}|},
\end{equation}
where $\overline{\mathfrak{m}}\,=\,\displaystyle \frac{1}{N} \sum_{i=1}^N\,\mathfrak{m}_{i}$ is the average of the simulated moments over the $N$ ensembles. Similarly, the half-width of the relative confidence interval is given by
\begin{equation}\label{eq:CI}
CI\;=\;\frac{z_{(1-\frac{\alpha}{2})}}{|\mathfrak{m}_{\text{anal}}|}\,\sqrt{\frac{\sigma^2}{N}},
\end{equation}
where  $\sigma^2\,=\,\displaystyle \frac{\sum_{i=1}^N \big(\mathfrak{m}_{i}-\overline{\mathfrak{m}} \big)^2 }{N-1}$ is the variance of the simulated moments and $z_{(1-\frac{\alpha}{2})}$ is the $z$-score for the confidence interval with $\alpha\,=\,10^{-3}$. A statistical simulation computes a moment accurately if 
$E < CI$, that is if there is a high probability that $\mathfrak{m}_{\text{anal}}$
lies in the confidence interval centered at $\overline{\mathfrak{m}}$, and that
this confidence interval is relatively narrow.

 \begin{figure}[!htbp] 
   \label{ABseries} 
   \minipage[t]{0.5 \textwidth}
     \centering
     \includegraphics[width=\textwidth]{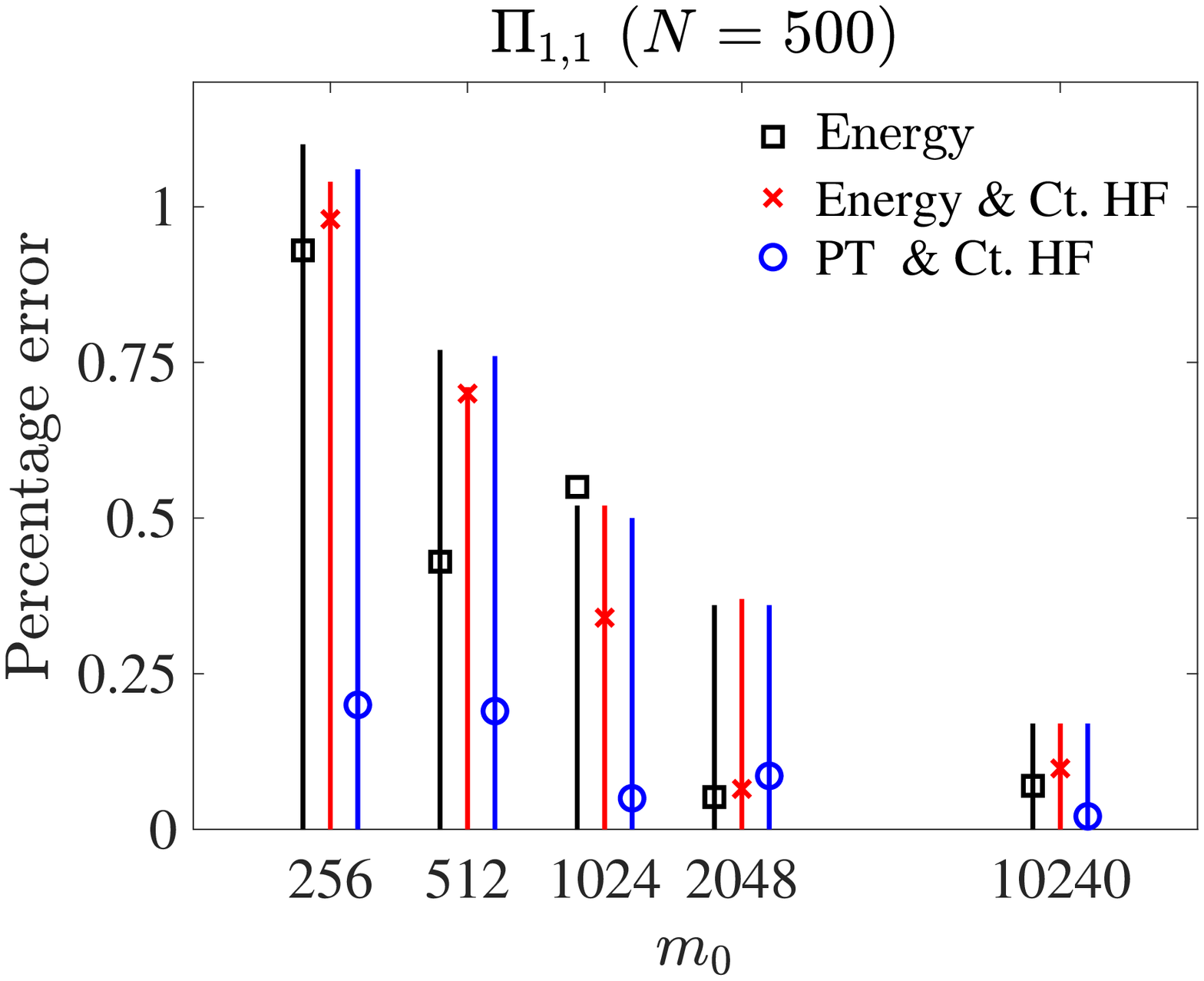} 
       \endminipage 
   \minipage[t]{0.5 \textwidth}
     \centering
      \includegraphics[width=\textwidth]{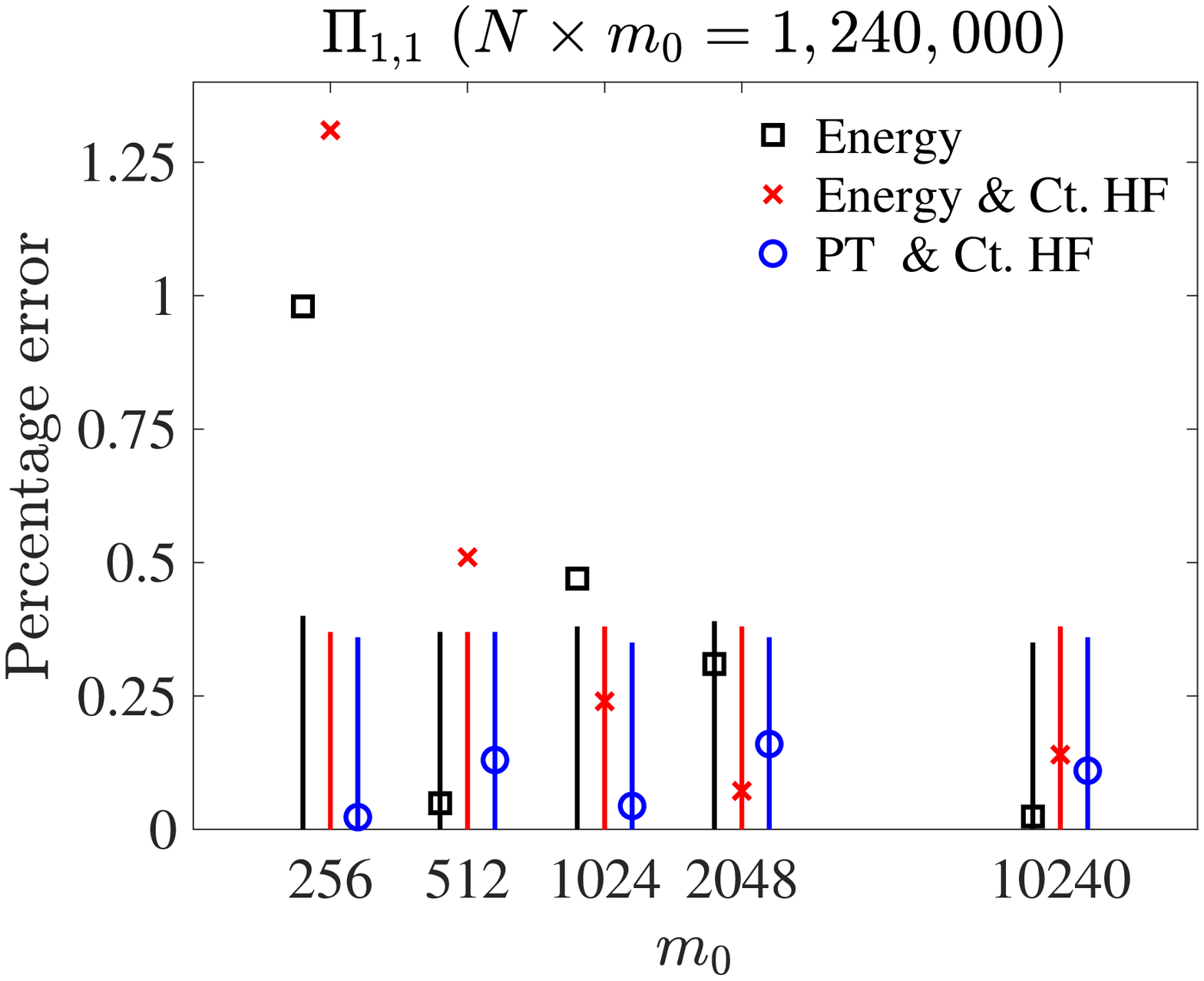}
       \endminipage 
       \medskip \hfill
    \minipage[t]{0.5 \textwidth}
     \centering
      \includegraphics[width=\textwidth]{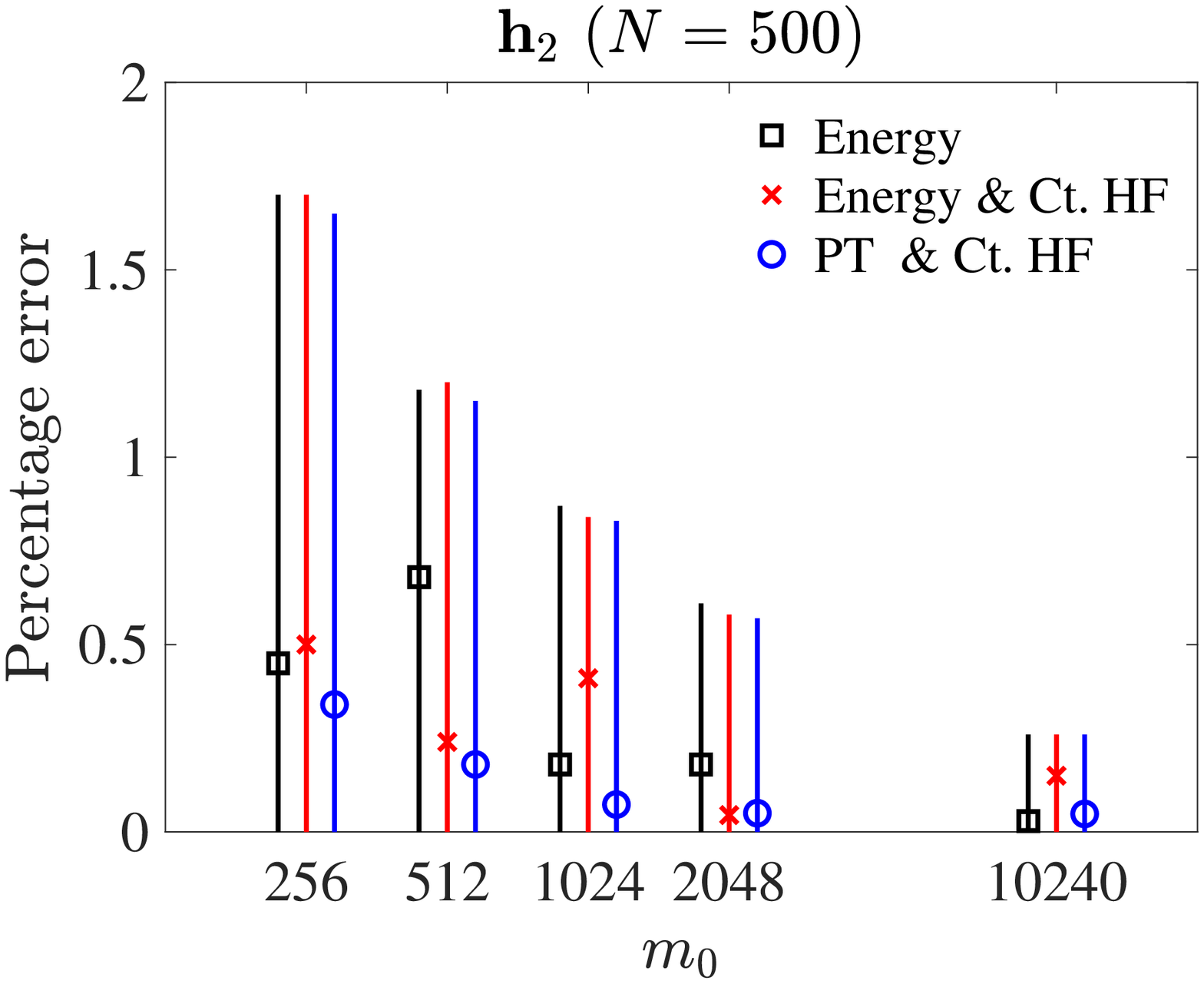}
       \endminipage 
          \minipage[t]{0.5 \textwidth}
     \centering
      \includegraphics[width=\textwidth]{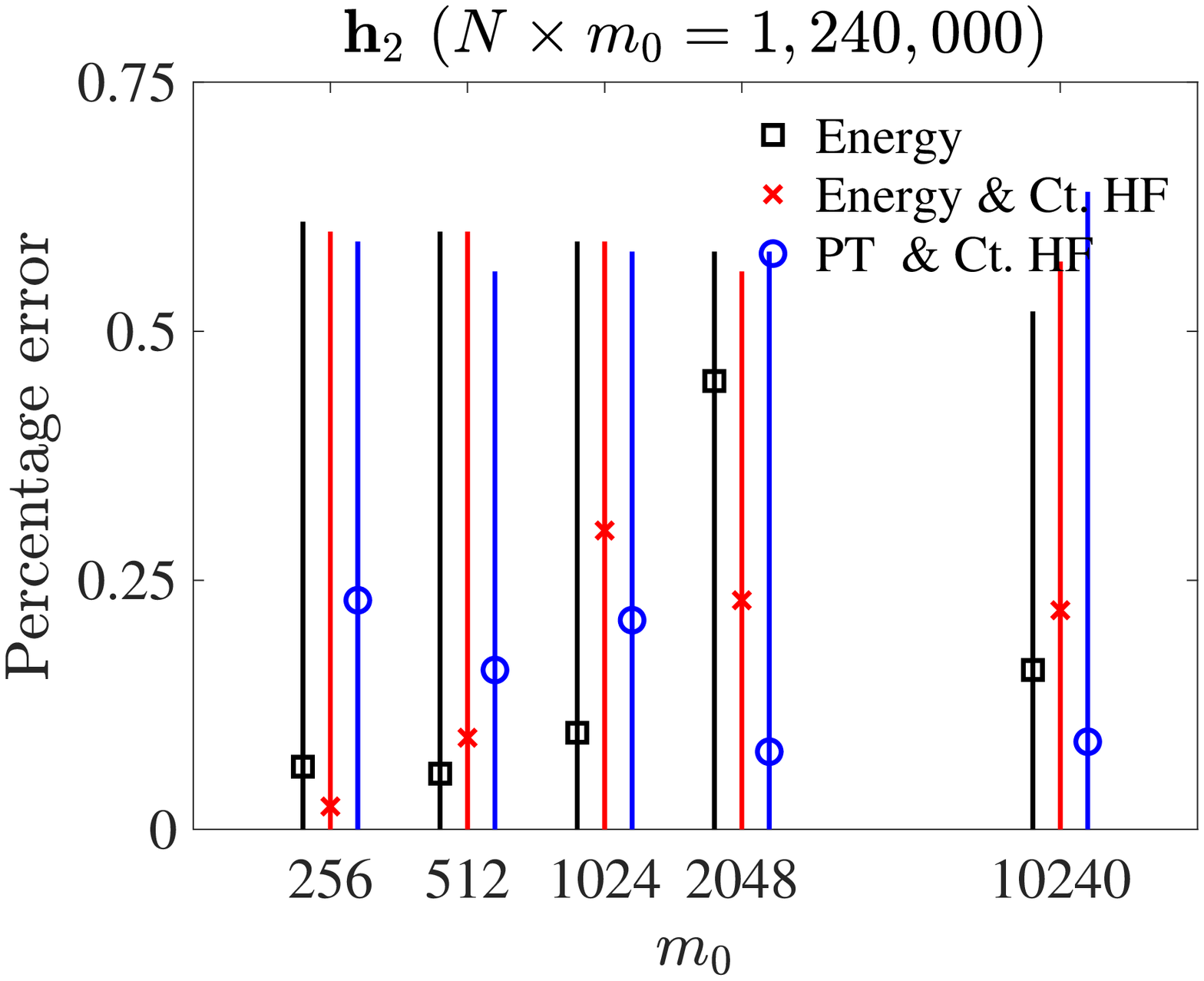}
       \endminipage 
       \medskip \hfill
         \minipage[t]{0.5 \textwidth}
     \centering
      \includegraphics[width=\textwidth]{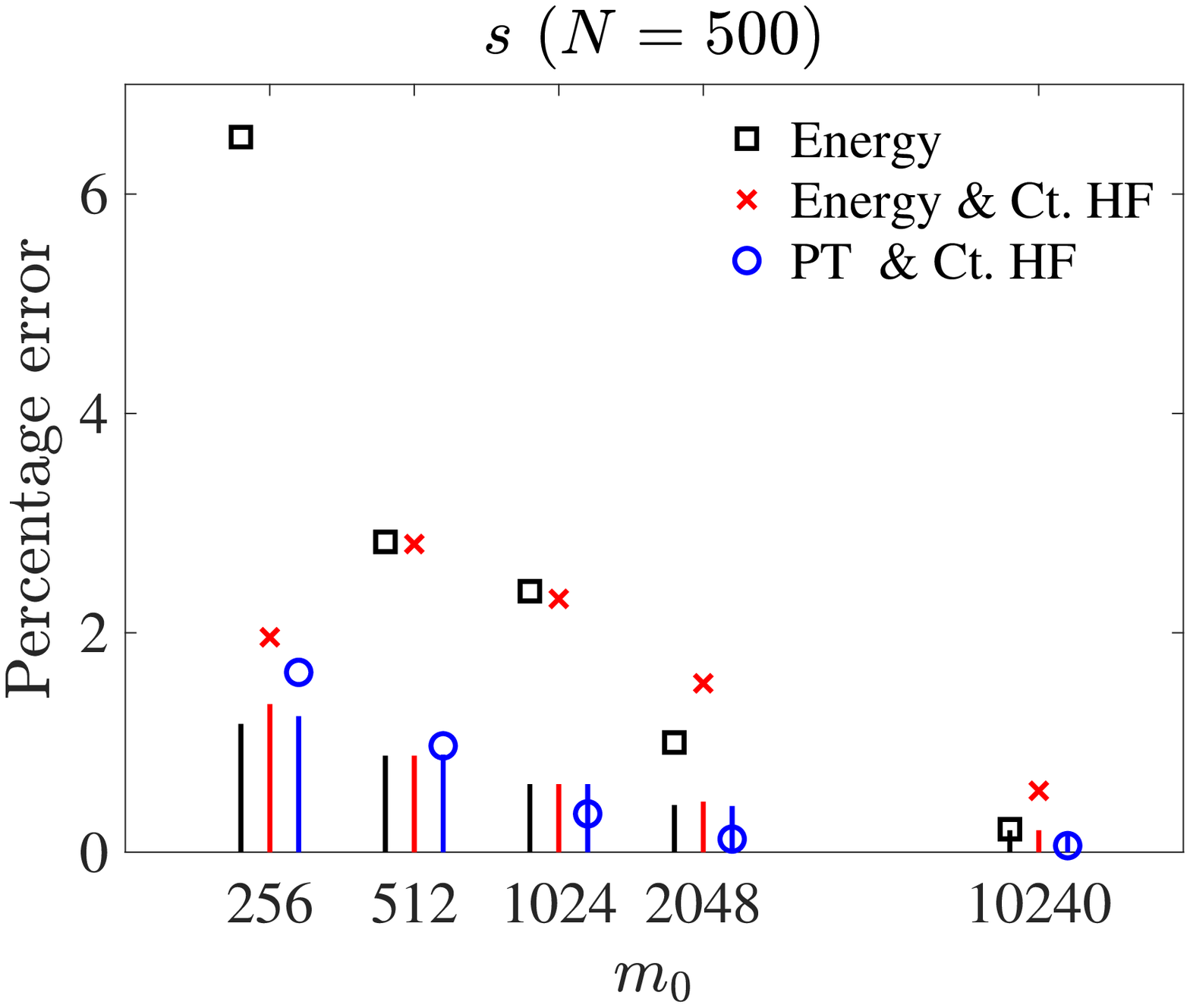}
       \endminipage 
          \minipage[t]{0.5 \textwidth}
     \centering
      \includegraphics[width=\textwidth]{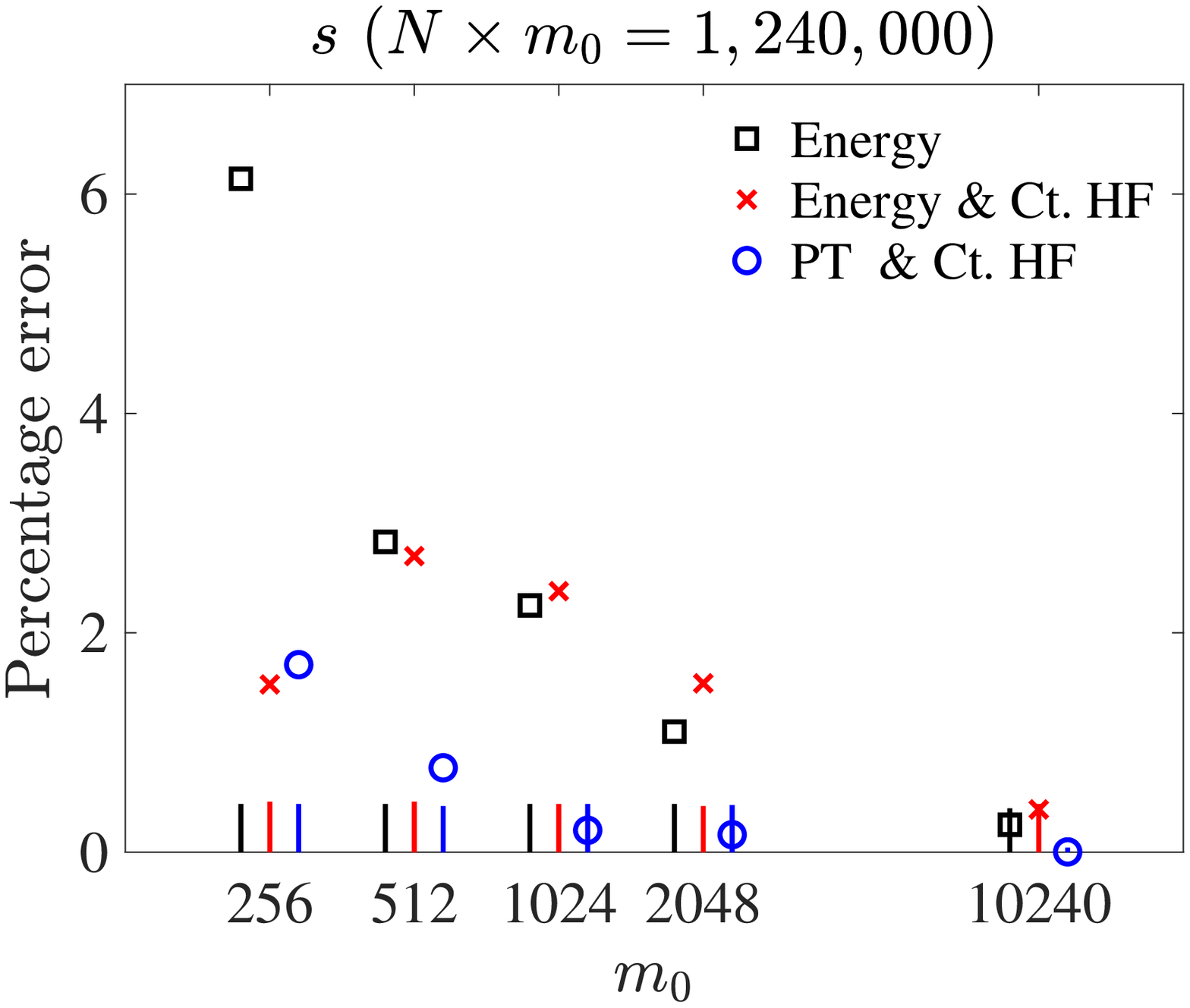}
       \endminipage 
       \medskip \hfill
  \caption{
  Percentage relative error, $E$, in \eqref{eq:MomentError} at time $t=3$ for selected moments, 
  $\mathfrak m$,  of the velocity pdf. 
We compare the performance of  the SWPM with the 
  three different reduction schemes shown in the legends.
  We show results for the $(1,1)$-component of the
  momentum, $\Pi_{1,1}$, (top row), the second component of the raw heat flux, $\mathbf h_2$, (middle row), and the scalar fourth-order moment, $s$, (bottom row).
  We plot the errors using symbols and the half confidence
  intervals with vertical lines, so that $E < CI$ when the symbol lies on the line.
  In the left column, we plot  $E$ as a function of the number of particles, $m_0$, 
  per ensemble for $N=500$ ensembles. In the right column, we plot
  $E $ as a function of $m_0$ when $N$ is chosen so that
  $N\times m_0 = 1,240,000$.}
\end{figure}

 In the first set of simulations, we used $N =$ 500 ensembles and various initial numbers of computational particles, $m_0$.  In the left column of \cref{ABseries}, we show the
 relative error, $E$, in  \eqref{eq:MomentError}  and confidence interval, $CI$, in
 \eqref{eq:CI} at time $t=3$  for 
 the $(1,1)$-component of the momentum, $\Pi_{1,1}$, (top), the second component of the raw heat flux, $\mathbf h_2$, (middle), and the scalar fourth-order moment, $s$, (bottom).
The  percentage  relative error, 
for the energy, the energy and central heat flux, and the pressure tensor and 
central heat flux reduction schemes are shown using the symbols in the legends. 
The half-width of the relative confidence intervals are shown using the corresponding vertical lines.
These quantities are plotted for the different values of $m_0$, which is displayed using a 
 logarithmic scale. 
 For each value of $m_0$, we have offset the results for the three reduction schemes from each other to aid comprehension. 
 
First, we observe that for each moment, the confidence intervals  primarily depend on 
$m_0$ rather than on the reduction scheme. Furthermore, with one slight exception,
the errors for $\Pi_{1,1}$ and $\mathbf{h}_2$ are within the confidence intervals, even for a small number of computational particles.  For $\Pi_{1,1}$, this result is to be expected  since all three reductions  schemes are designed to conserve momentum. 
However, as we saw in \cref{tab: table errors}, the energy and  energy and central heat flux 
reduction schemes do not preserve the raw heat flux. 
Therefore, the accuracy of the computation of
 $\mathbf{h}_2$ with these two schemes is simply due to statistical averaging
 over the  500 ensembles. Significantly, in most cases the errors for 
 $\Pi_{1,1}$ and $\mathbf{h}_2$ are smaller with the pressure tensor and central heat flux
 scheme than with the other two reduction schemes. 
 
 The main advantage to be gained from using the new pressure tensor and heat flux
reduction scheme can be seen in the results for the scalar fourth-order moment, $s$ (see the bottom left panel of \cref{ABseries}). 
 With our method,  the error in $s$ lies within the confidence interval
 for  $m_0 \geq $ 1,024. However, with the other two methods
 the errors are larger than the width of the confidence interval,
 even for $m_0 =$ 10,240. Therefore, the energy conservation, and energy and central heat flux conservation reduction schemes require more than $10$ times the initial number of computational particles as the pressure tensor and heat flux conservation scheme to approximate the scalar fourth-order moment with the same degree of accuracy. As we see in \cref{table:timings}, this requires at least seventeen times the computational time. 

 \begin{table}[!htbp]
  {\footnotesize
    \caption{Total computational time for the simulation results shown in \cref{ABseries}.
    These results were obtained with the pressure tensor and central heat flux
    scheme. The computational times for the two reduction schemes of Rjasanow and Wagner were similar ($\approx \pm 10\%$).} 
    \label{table:timings}
    \begin{center}
    \begin{tabular}{|c|c|c|} 
 \cline{2-3}
      \multicolumn{1}{c|}{}& \multicolumn{1}{c|}{$N=500$}& \multicolumn{1}{c|}{$N\times m_0 = 1,240,000$}\\
     \hline
      $m_0$ &  $t$ (sec) & $t$ (sec) \\
       \hline       
      256 & 2.61 & 20.85 \\
      512 &  6.06 & 22.02\\
       1,024 &13.76 & 24.93 \\
       2,048  & 31.26& 31.26 \\
       10,240 &  205.56 & 39.05\\
       \hline
    \end{tabular}
  \end{center}
  }
\end{table}

 \begin{figure}[!htbp] 
   \label{fig1: s} 
   \minipage[t]{0.33 \textwidth}
     \centering
     \includegraphics[width=\textwidth]{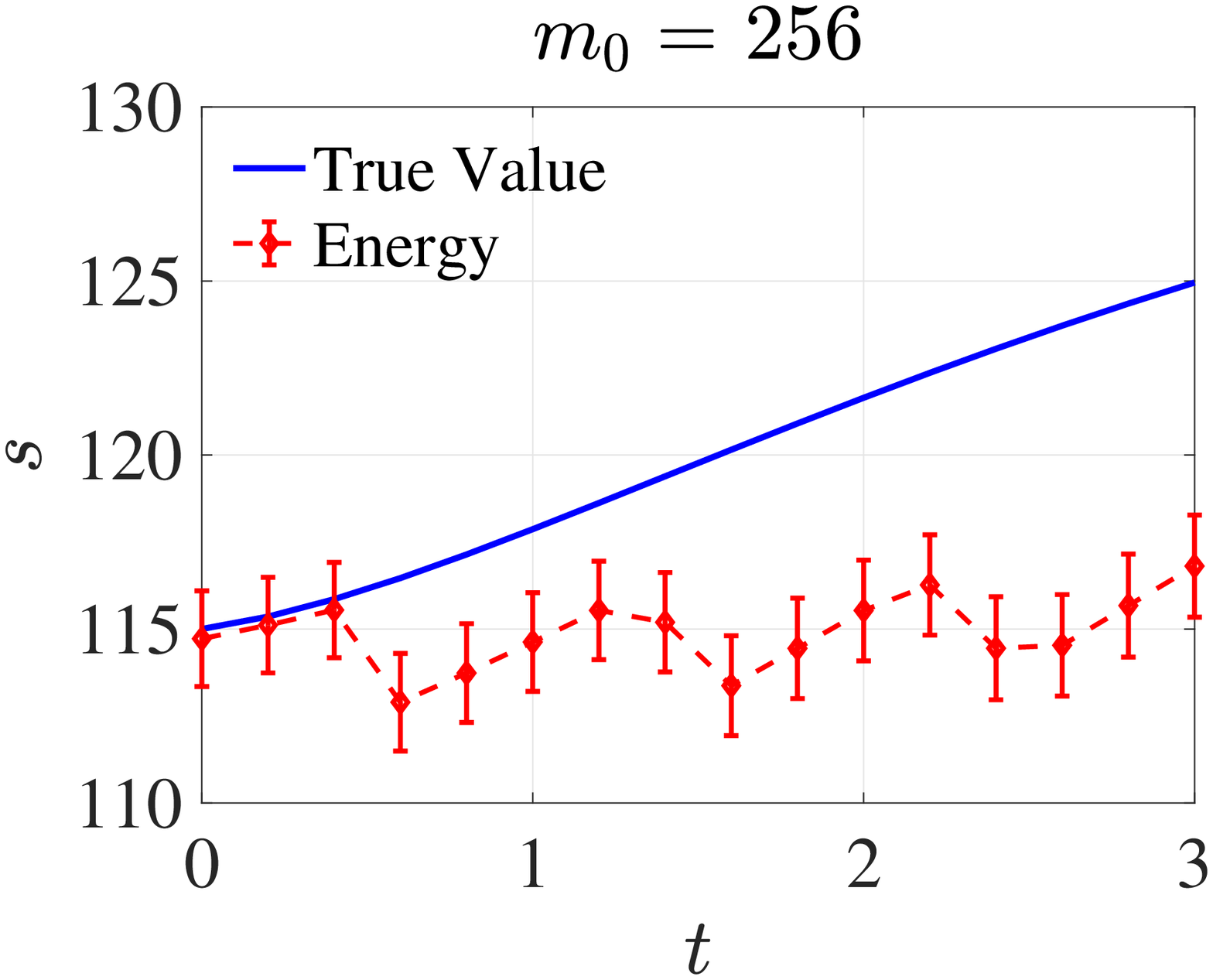} 
       \endminipage 
   \minipage[t]{0.33 \textwidth}
     \centering
      \includegraphics[width=\textwidth]{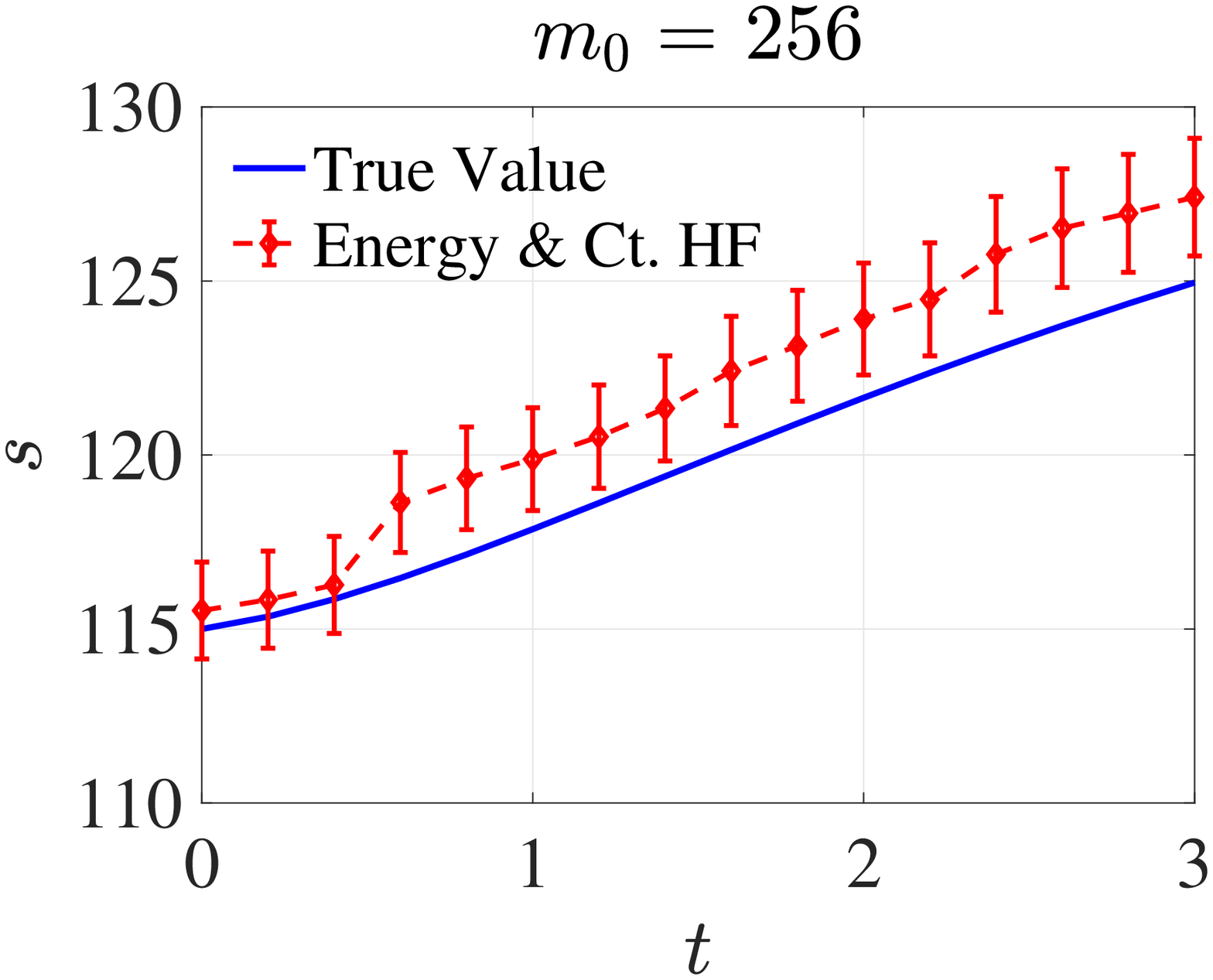}
       \endminipage 
    \minipage[t]{0.33 \textwidth}
     \centering
      \includegraphics[width=\textwidth]{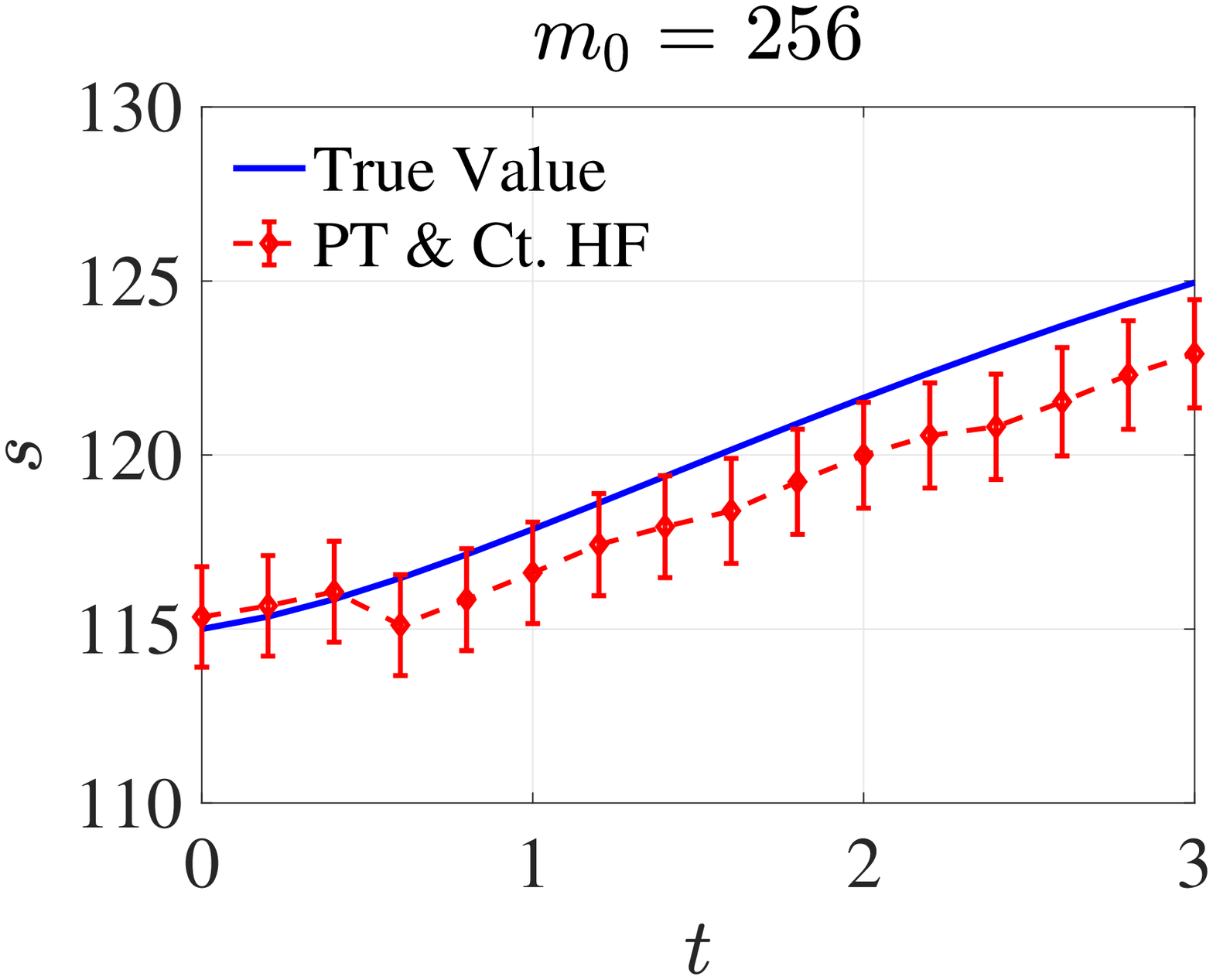}
       \endminipage \medskip \hfill
       \minipage[t]{0.33 \textwidth}
     \centering
     \includegraphics[width=\textwidth]{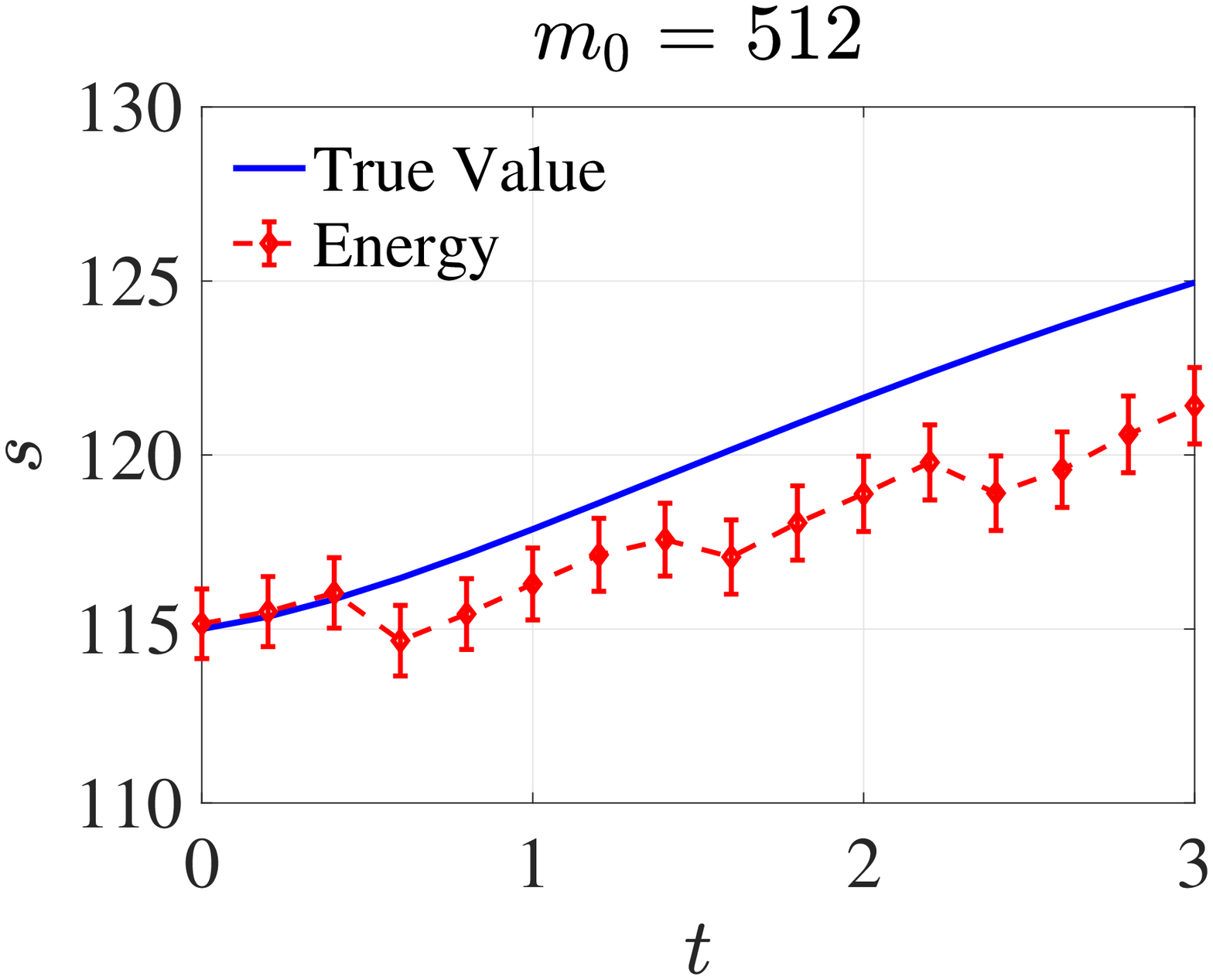} 
       \endminipage
   \minipage[t]{0.33 \textwidth}
     \centering
      \includegraphics[width=\textwidth]{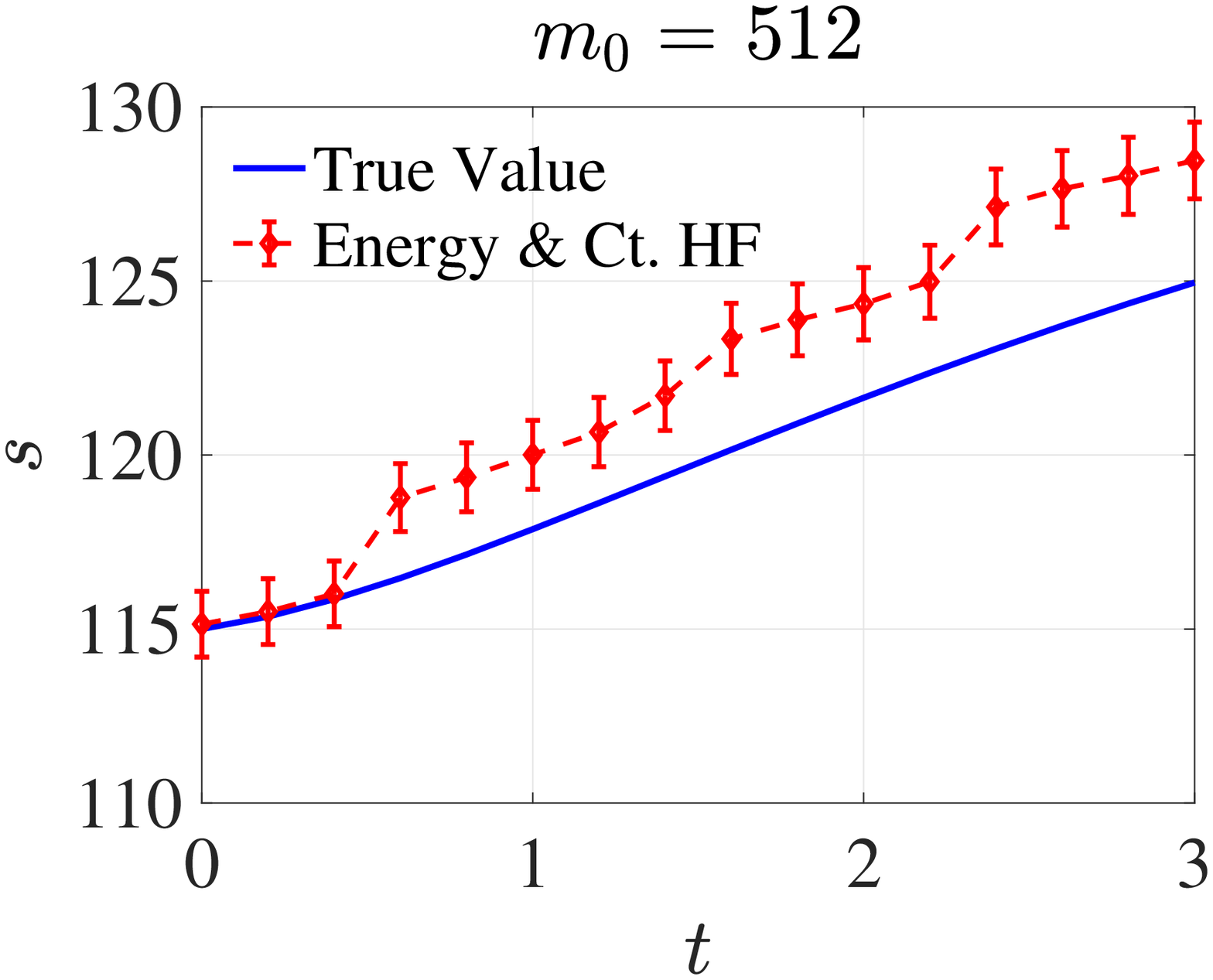}
       \endminipage 
    \minipage[t]{0.33 \textwidth}
     \centering
      \includegraphics[width=\textwidth]{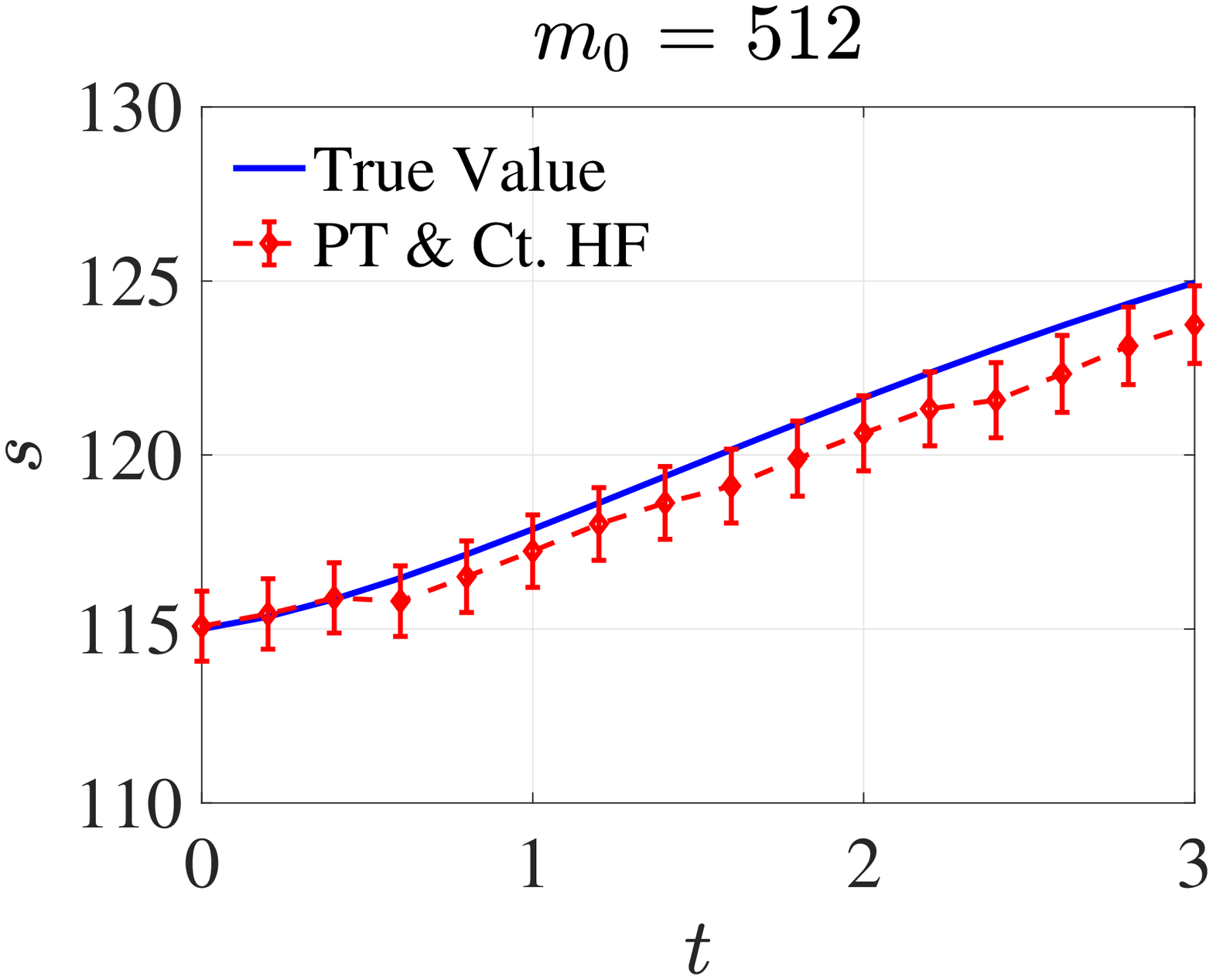}
       \endminipage \medskip \hfill 
 \minipage[t]{0.33 \textwidth}
     \centering
     \includegraphics[width=\textwidth]{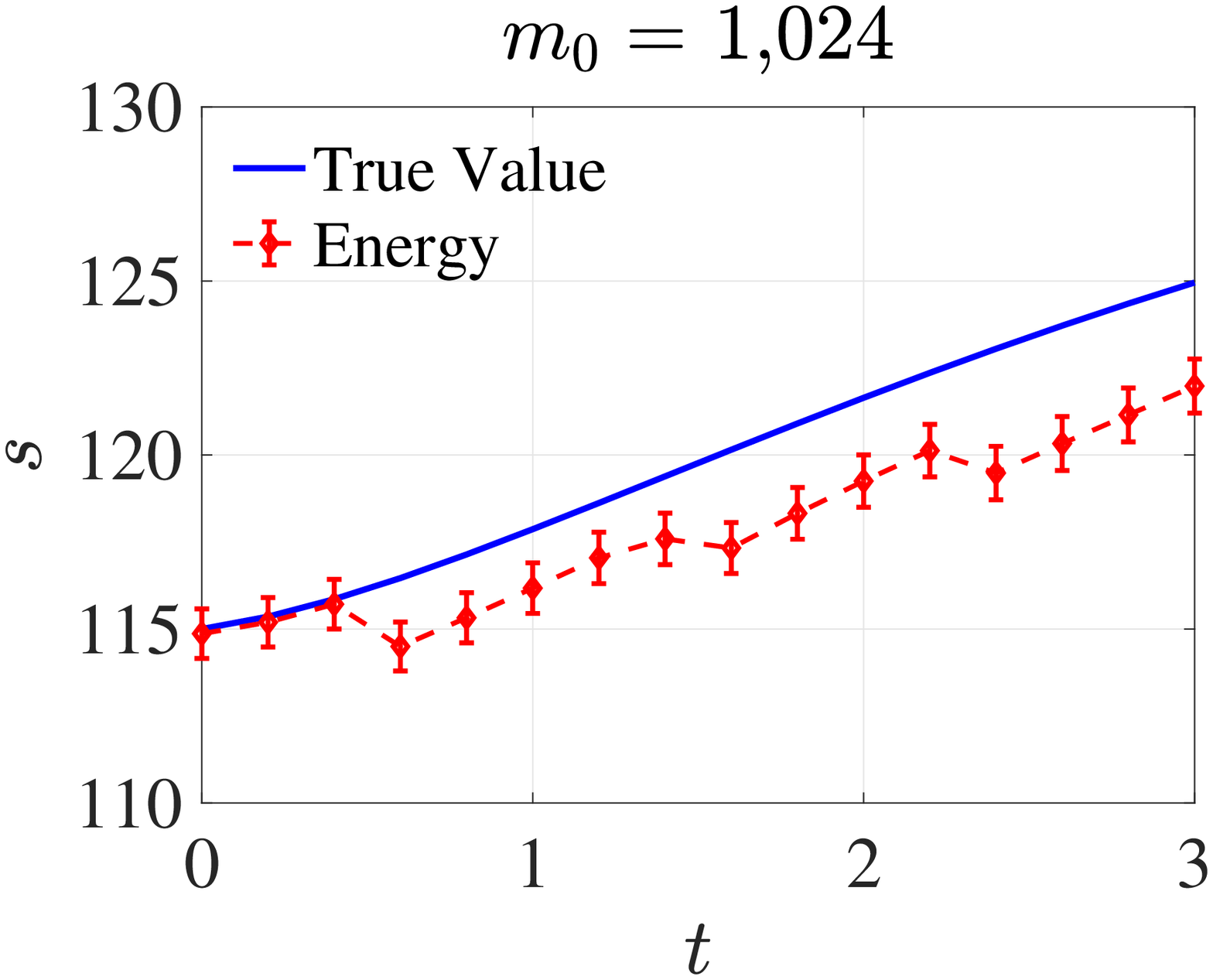} 
       \endminipage
   \minipage[t]{0.33 \textwidth}
     \centering
      \includegraphics[width=\textwidth]{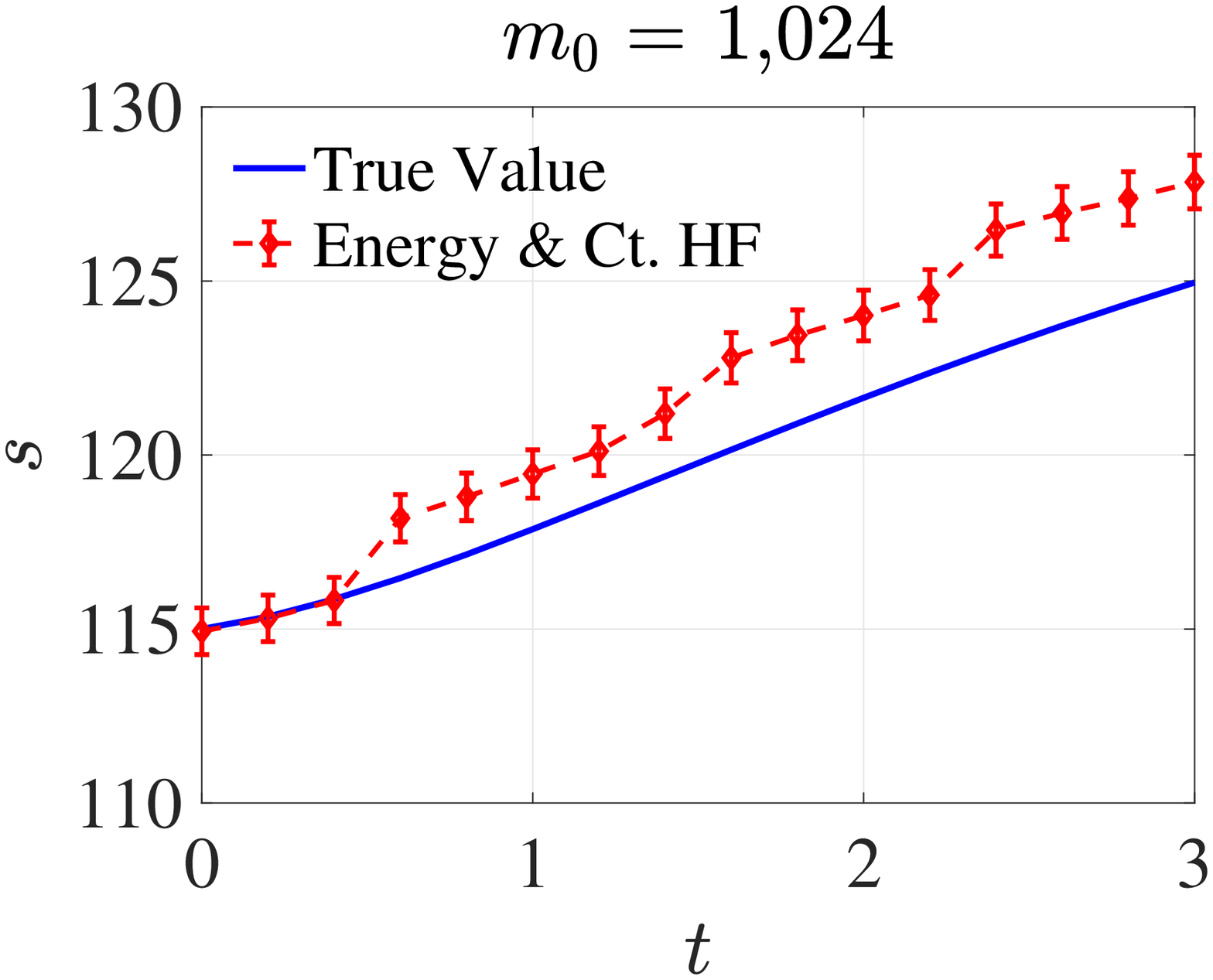}
       \endminipage 
    \minipage[t]{0.33 \textwidth}
     \centering
      \includegraphics[width=\textwidth]{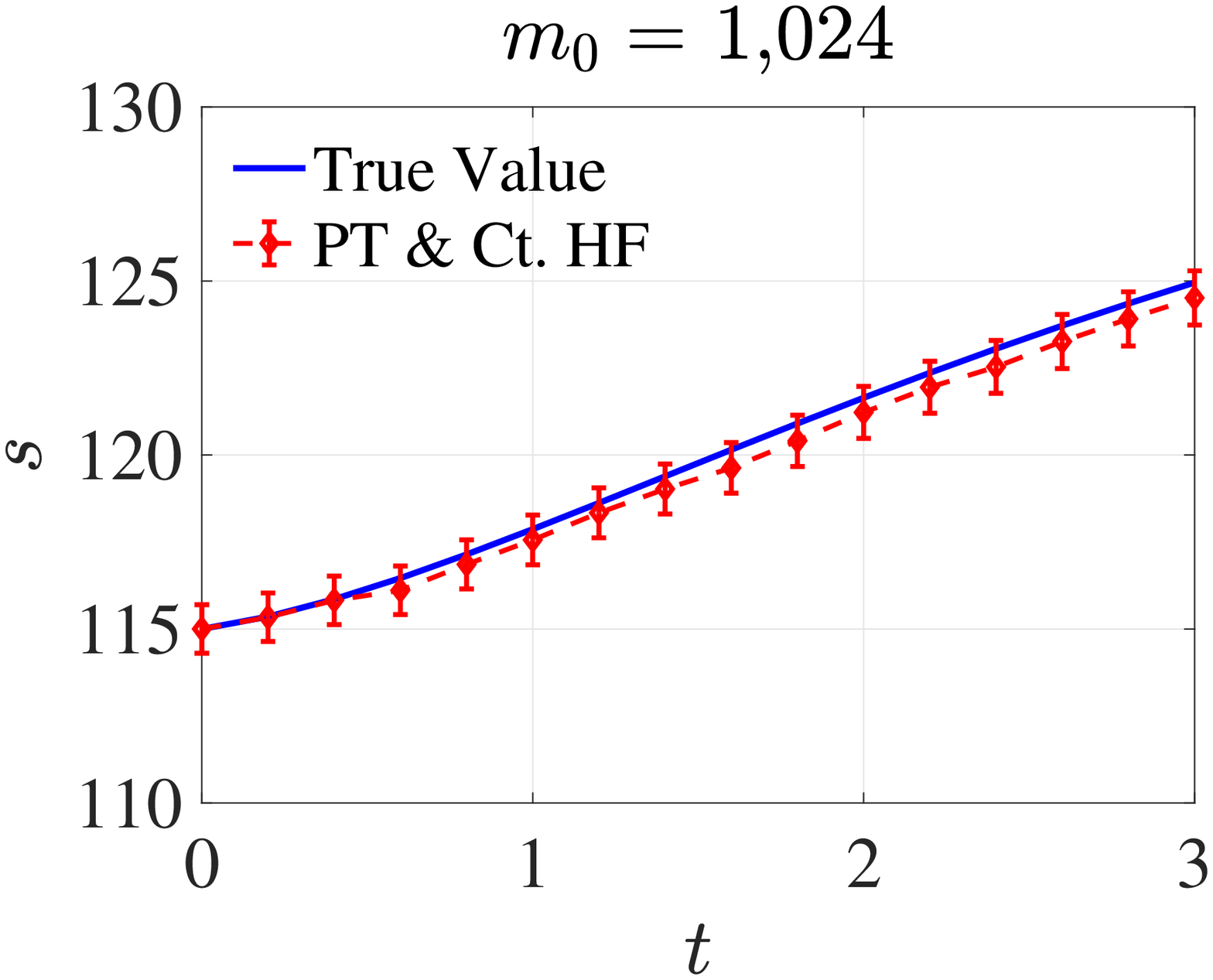}
       \endminipage \medskip \hfill
 \minipage[t]{0.33 \textwidth}
     \centering
     \includegraphics[width=\textwidth]{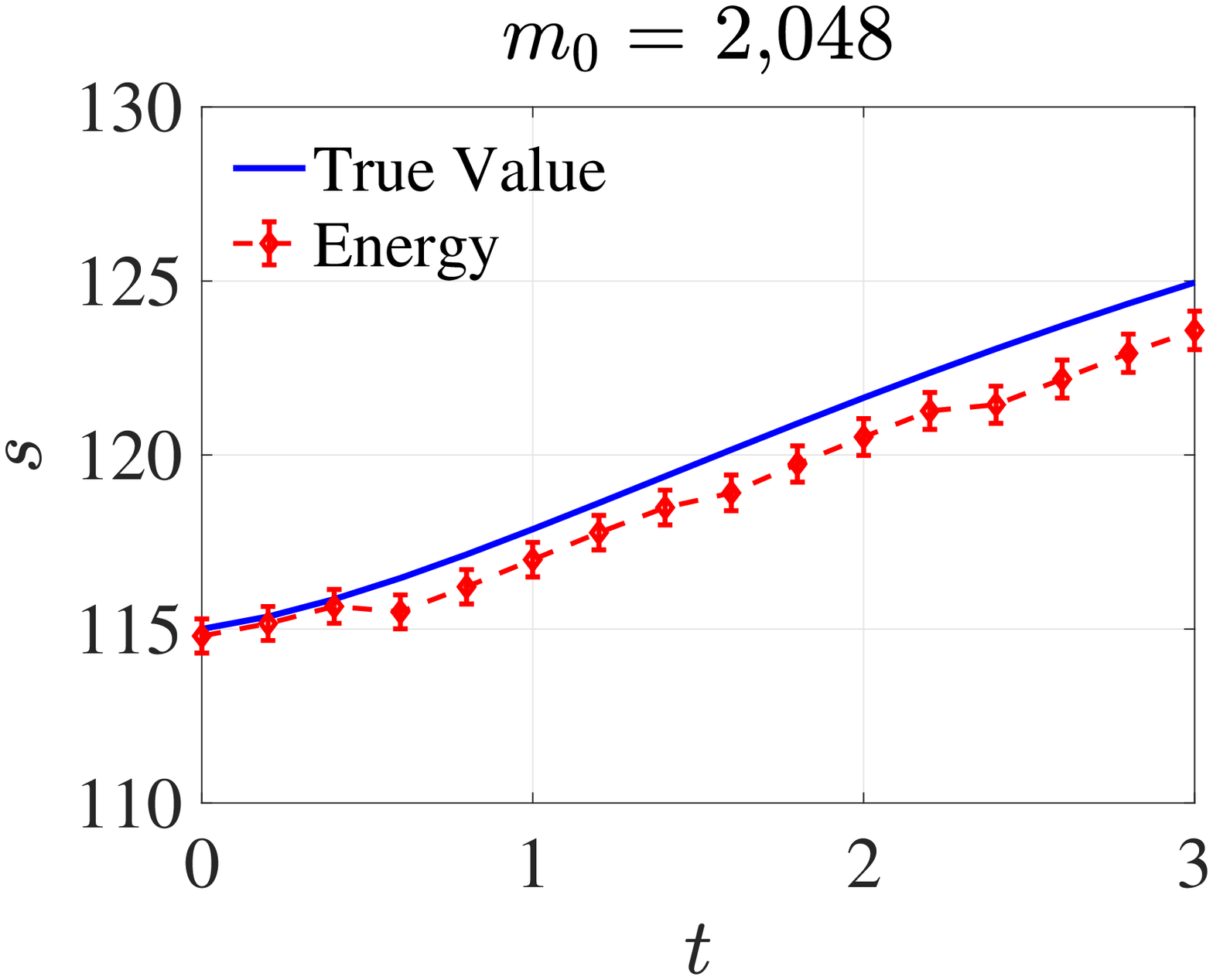} 
       \endminipage
   \minipage[t]{0.33 \textwidth}
     \centering
      \includegraphics[width=\textwidth]{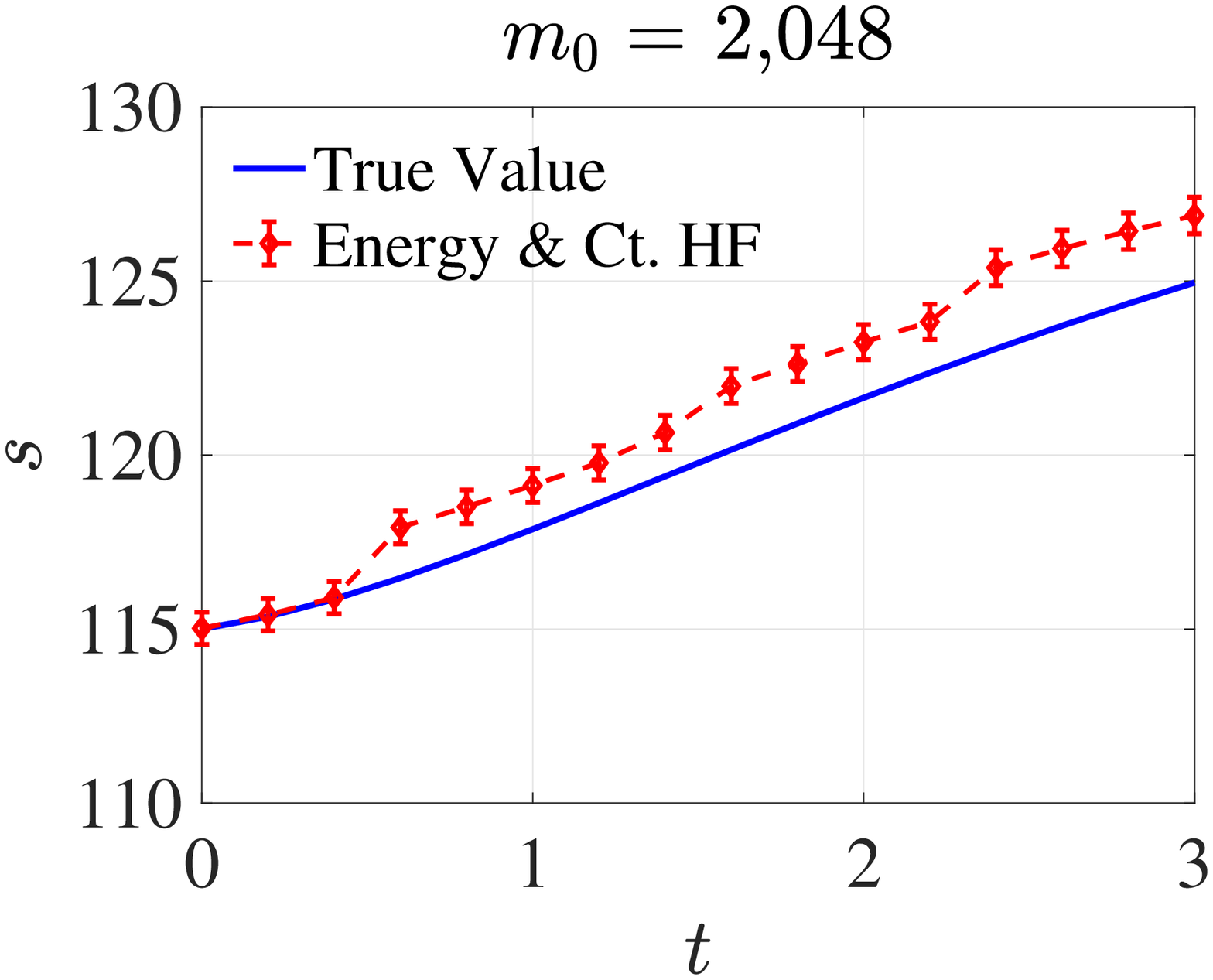}
       \endminipage 
    \minipage[t]{0.33 \textwidth}
     \centering
      \includegraphics[width=\textwidth]{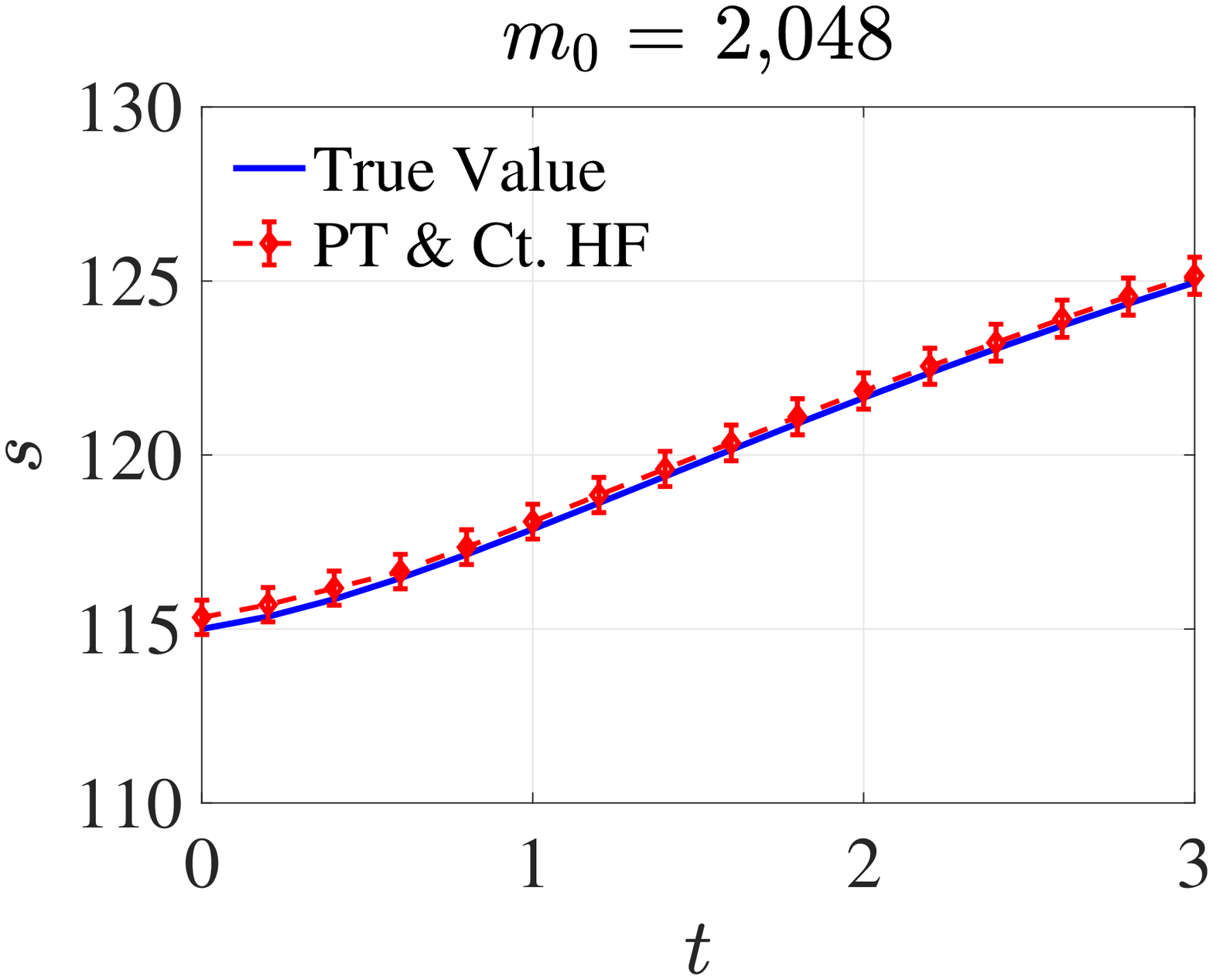}
       \endminipage \medskip \hfill 
 \minipage[t]{0.33 \textwidth}
     \centering
     \includegraphics[width=\textwidth]{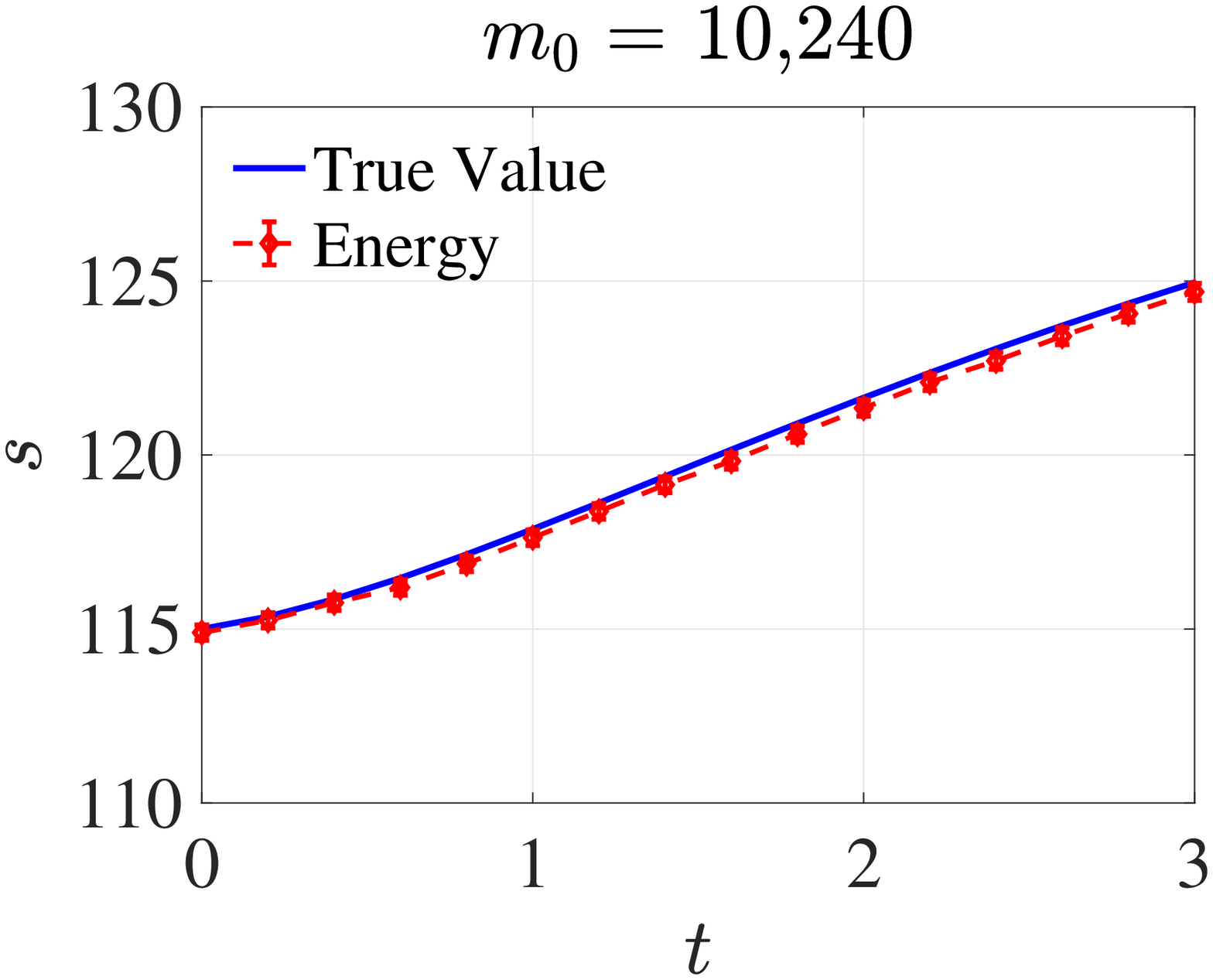} 
       \endminipage
   \minipage[t]{0.33 \textwidth}
     \centering
      \includegraphics[width=\textwidth]{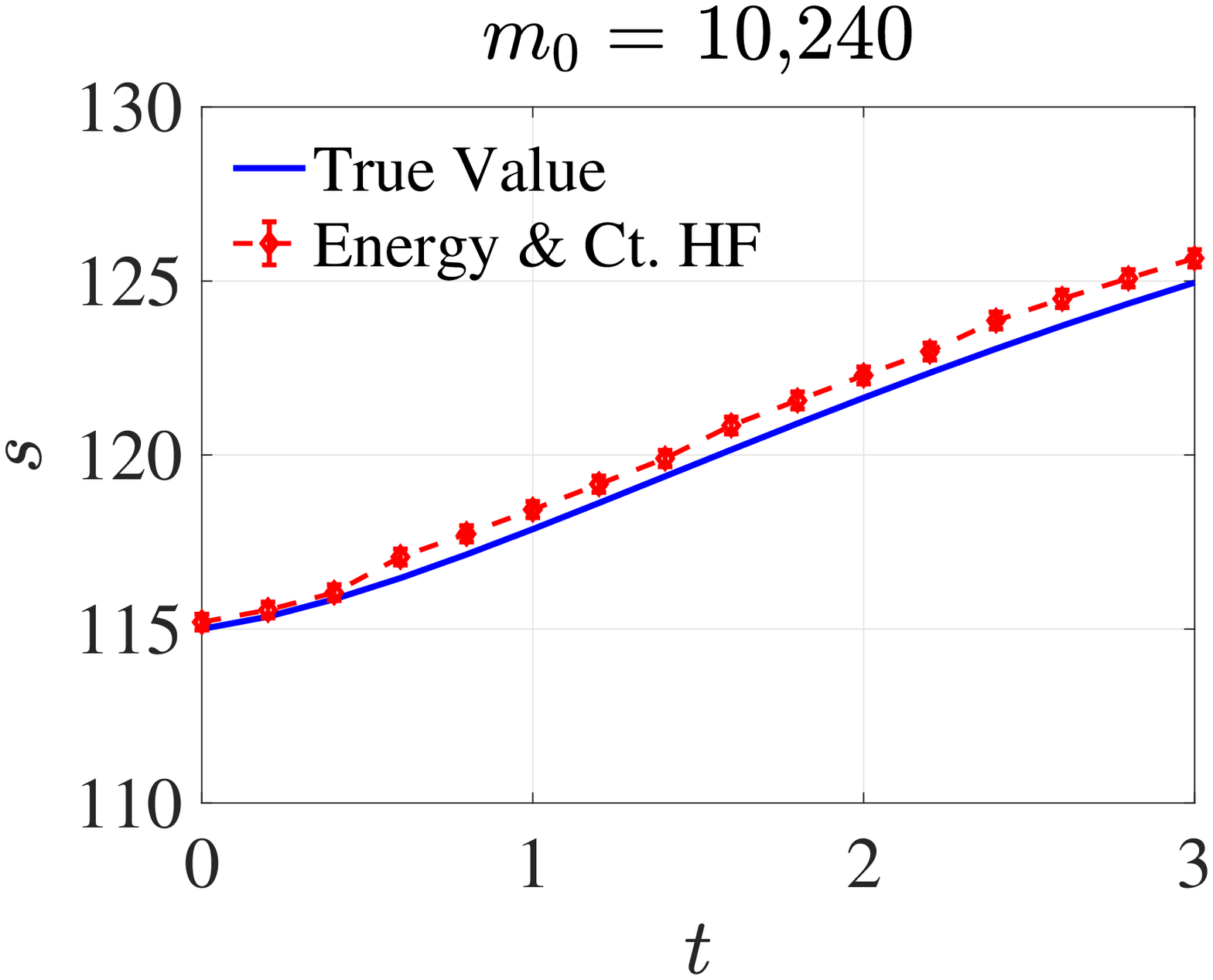}
       \endminipage 
    \minipage[t]{0.33 \textwidth}
     \centering
      \includegraphics[width=\textwidth]{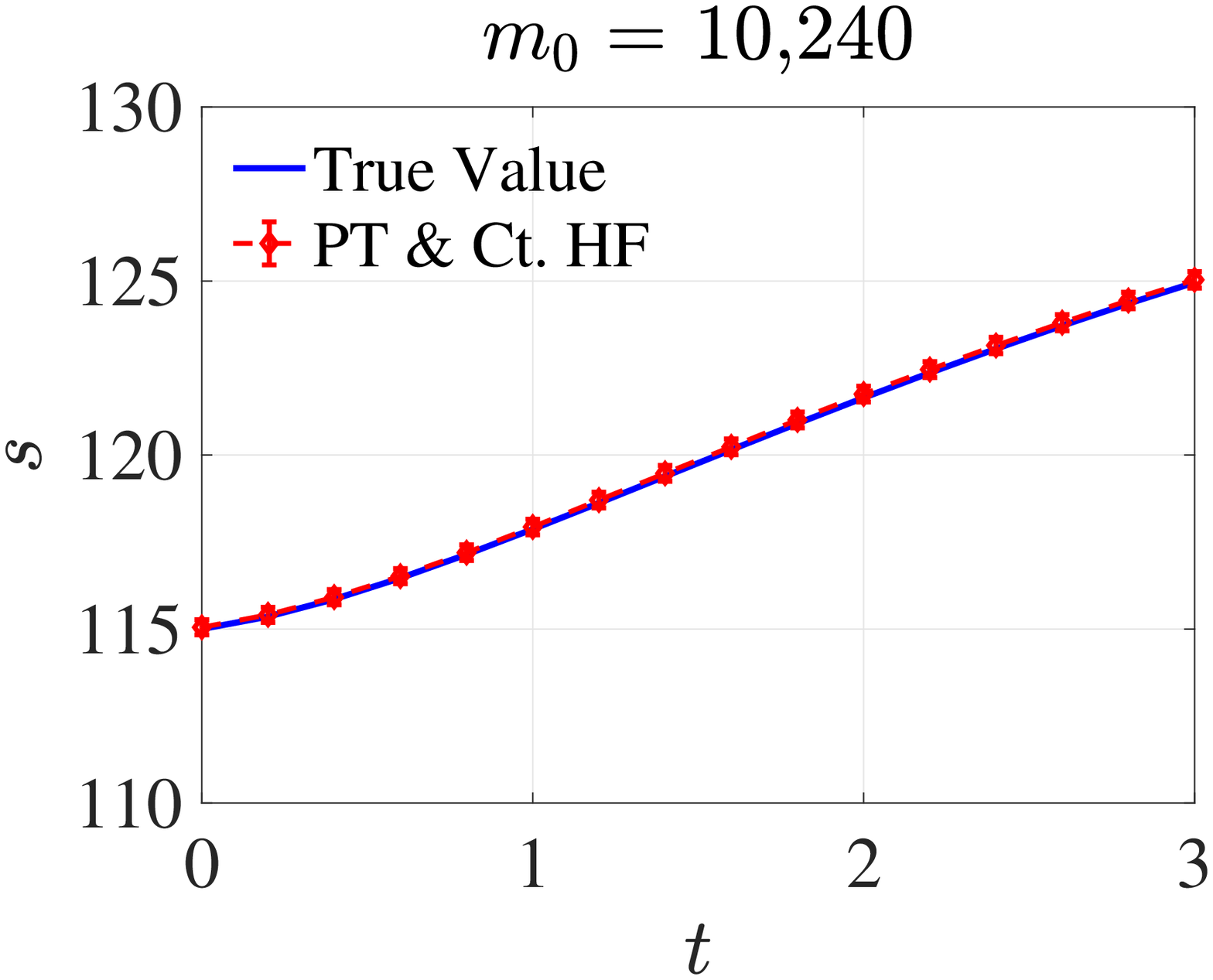}
       \endminipage \hfill 
     \caption{Evolution of the scalar fourth order moment, $s$, as a function of time together with $99.9\%$ confidence intervals. In the different rows we show the results for different initial numbers of computational particles, $m_0$, per ensemble. We used $N =$ 500 ensembles in all the panels. We show the results for the energy scheme (left column), energy and central heat flux scheme (middle column), and pressure tensor and central heat flux scheme (right column).}
\end{figure}

To further examine how accurately the three reductions schemes compute the
scalar fourth-order moment, in  \cref{fig1: s} we plot the
evolution of  $s$ as a function of time, together with $99.9\%$ confidence intervals. 
The numerical results are shown with red-dashed lines and the true values
are shown with solid blue lines. The results for the energy and 
the energy and central heat flux reduction schemes are 
shown in the left and middle columns. 
The numerical results are visually close to the true values only for $m_0=10,240$ (bottom left and middle panels). On the other hand, with the 
pressure tensor and central heat flux scheme (right column), the numerical results
are reasonably accurate across the entire time range for $m_0=1,240$.
In particular, examining each column in turn, we see that the convergence of $s$  is significantly faster for the pressure tensor and central heat flux conservation scheme, than for the other two reduction schemes.

\begin{figure}[!htbp] 
   \label{fig2: s}
   \minipage[t]{0.33 \textwidth}
     \centering
     \includegraphics[width=\textwidth]{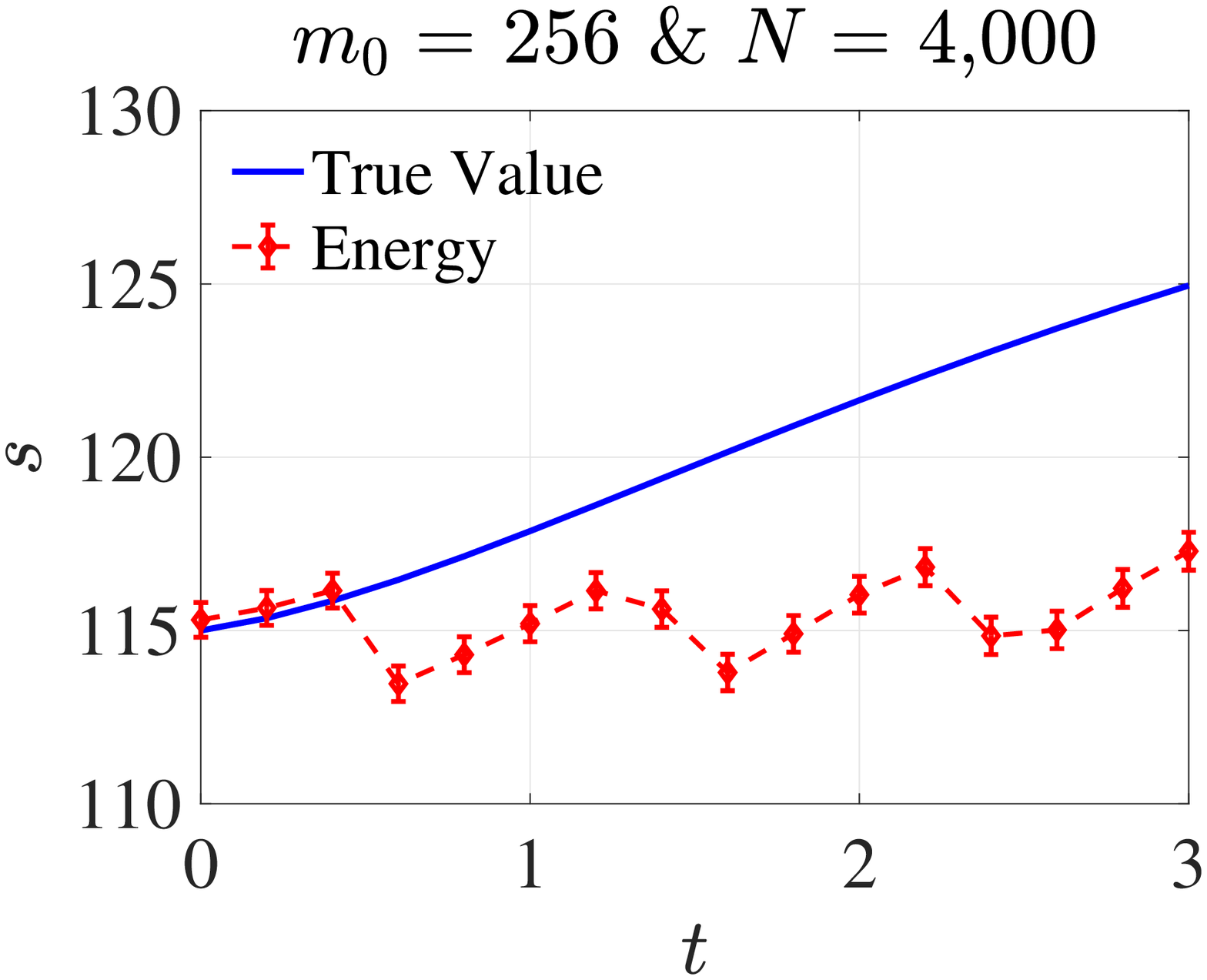} 
       \endminipage 
   \minipage[t]{0.33 \textwidth}
     \centering
      \includegraphics[width=\textwidth]{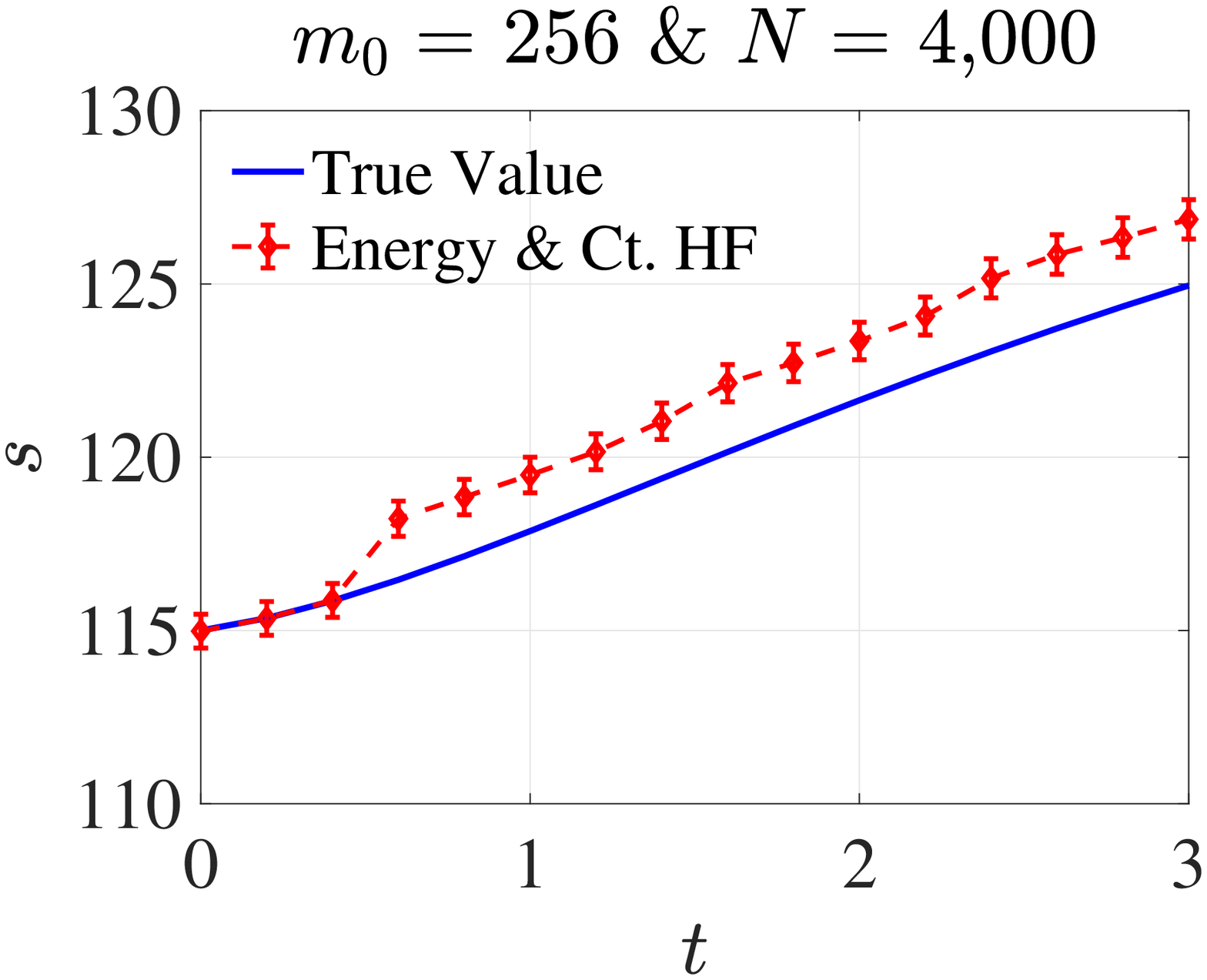}
       \endminipage 
    \minipage[t]{0.33 \textwidth}
     \centering
      \includegraphics[width=\textwidth]{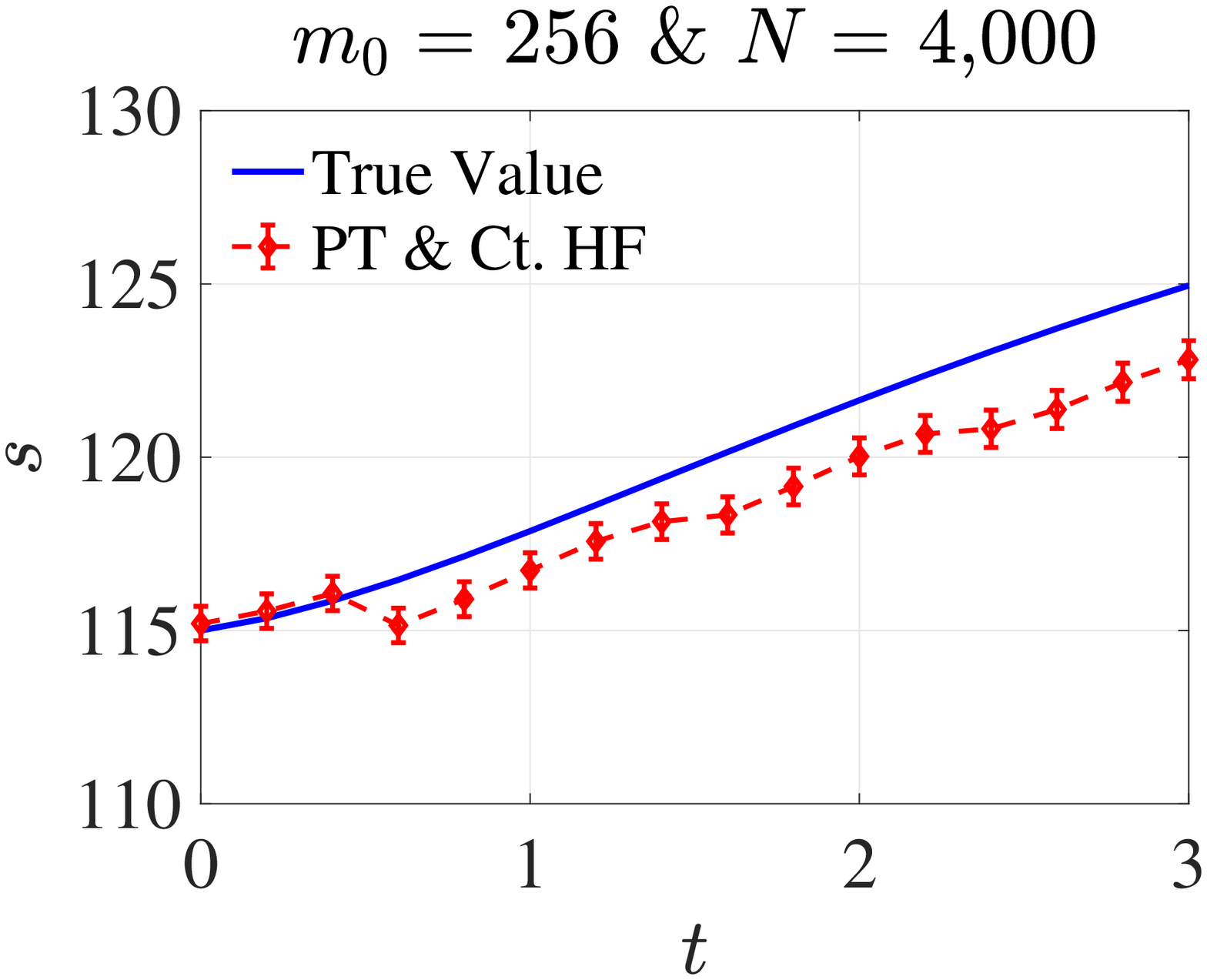}
       \endminipage \medskip \hfill
       \minipage[t]{0.33 \textwidth}
     \centering
     \includegraphics[width=\textwidth]{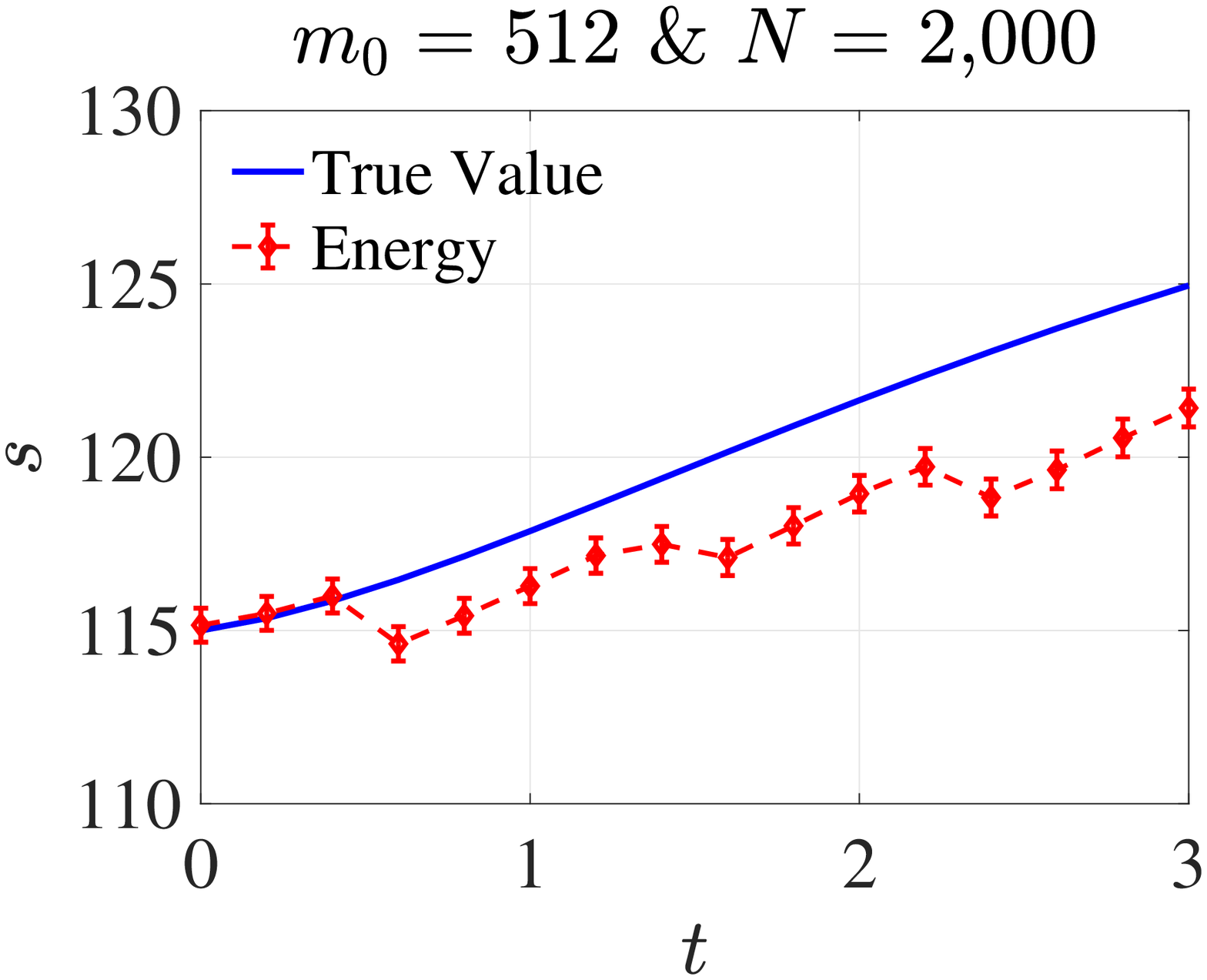} 
       \endminipage
   \minipage[t]{0.33 \textwidth}
     \centering
      \includegraphics[width=\textwidth]{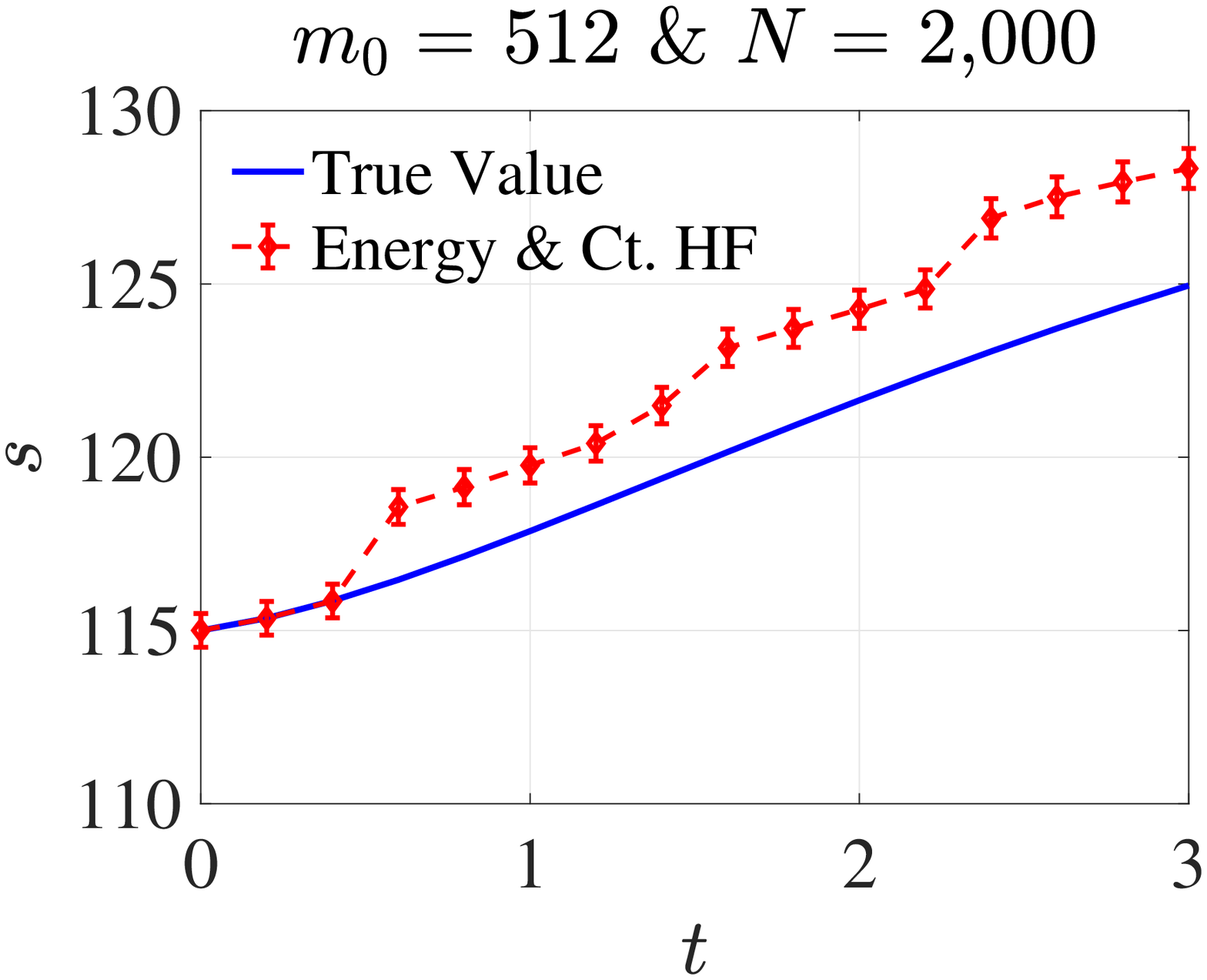}
       \endminipage 
    \minipage[t]{0.33 \textwidth}
     \centering
      \includegraphics[width=\textwidth]{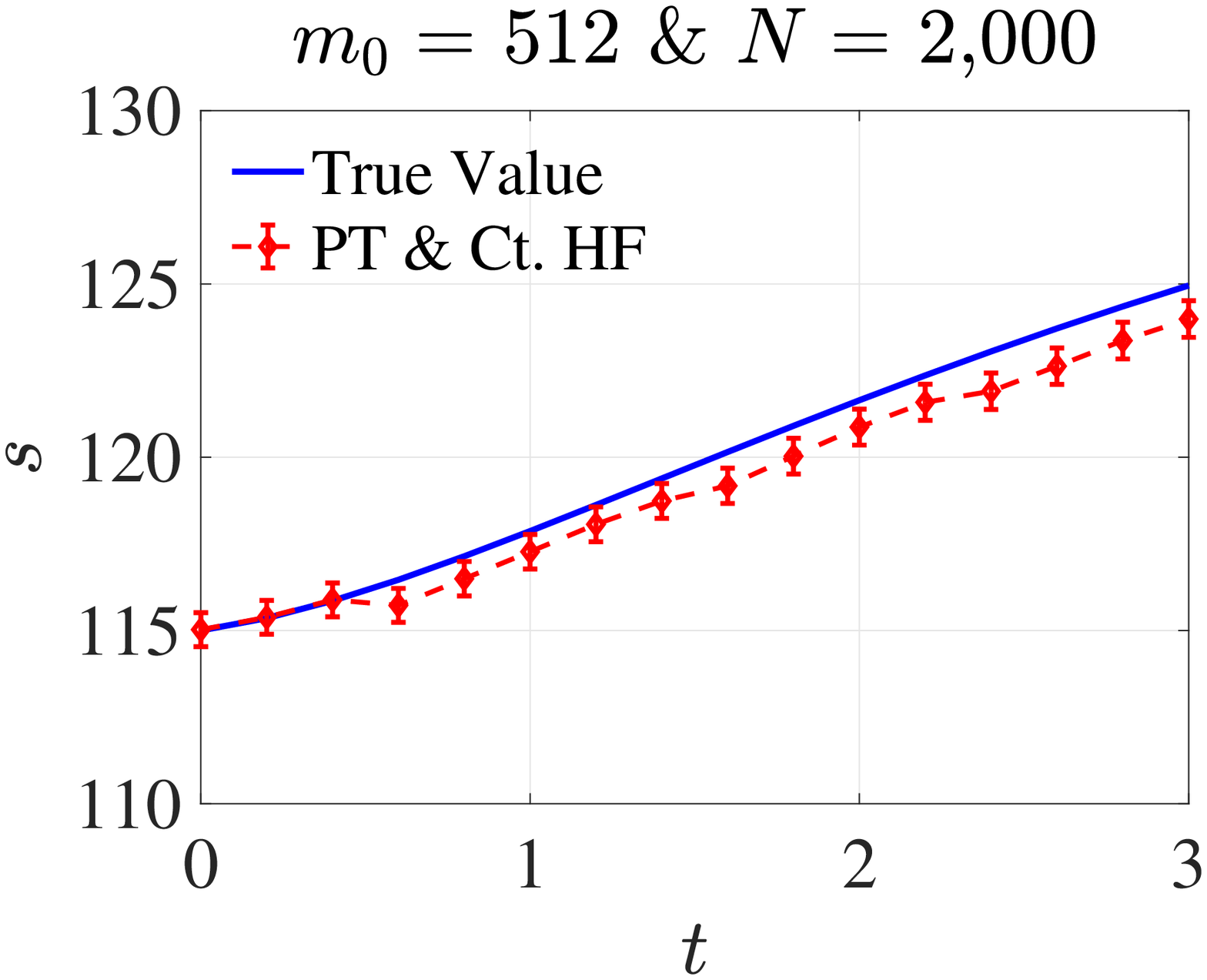}
       \endminipage \medskip \hfill
 \minipage[t]{0.33 \textwidth}
     \centering
     \includegraphics[width=\textwidth]{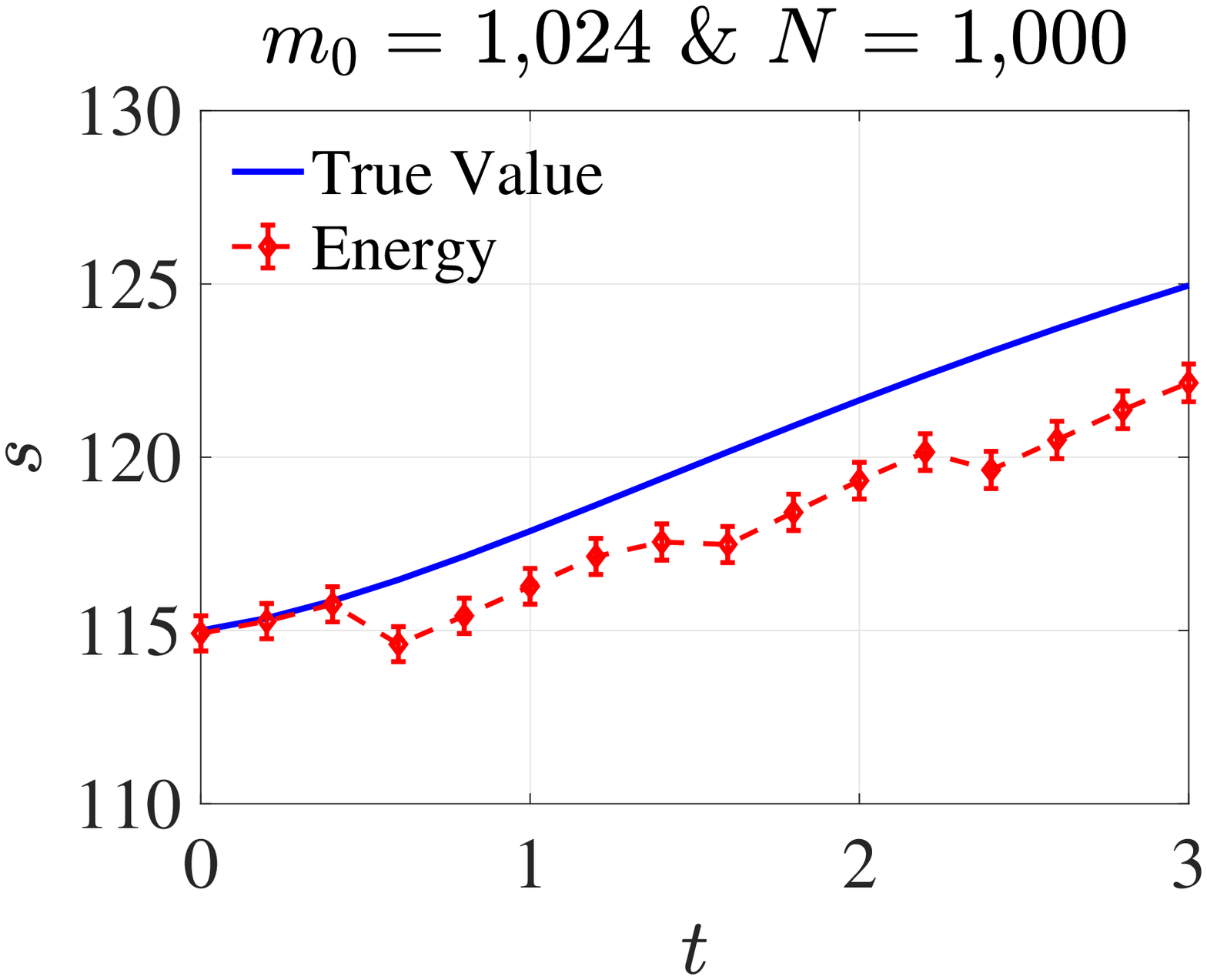} 
       \endminipage
   \minipage[t]{0.33 \textwidth}
     \centering
      \includegraphics[width=\textwidth]{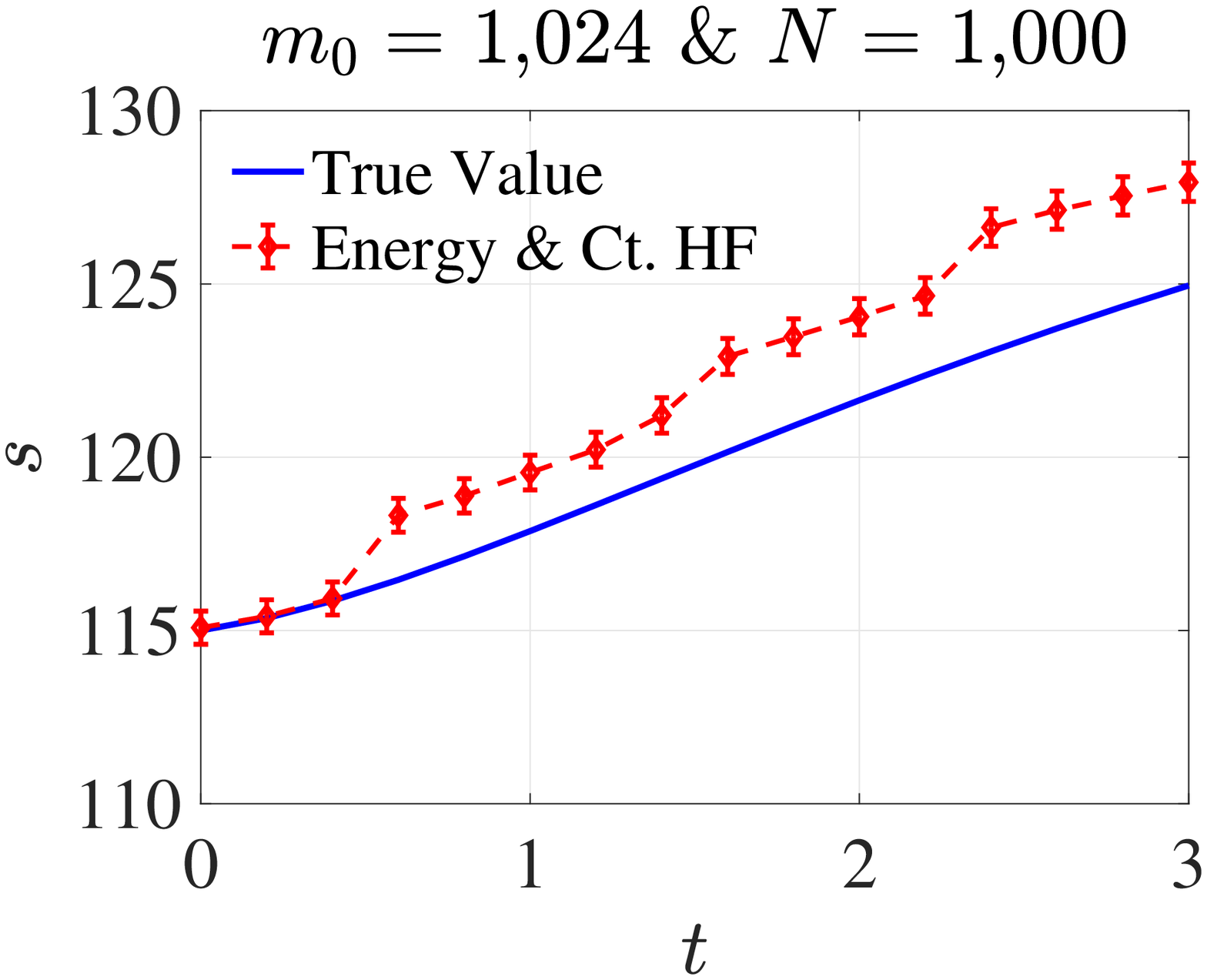}
       \endminipage 
    \minipage[t]{0.33 \textwidth}
     \centering
      \includegraphics[width=\textwidth]{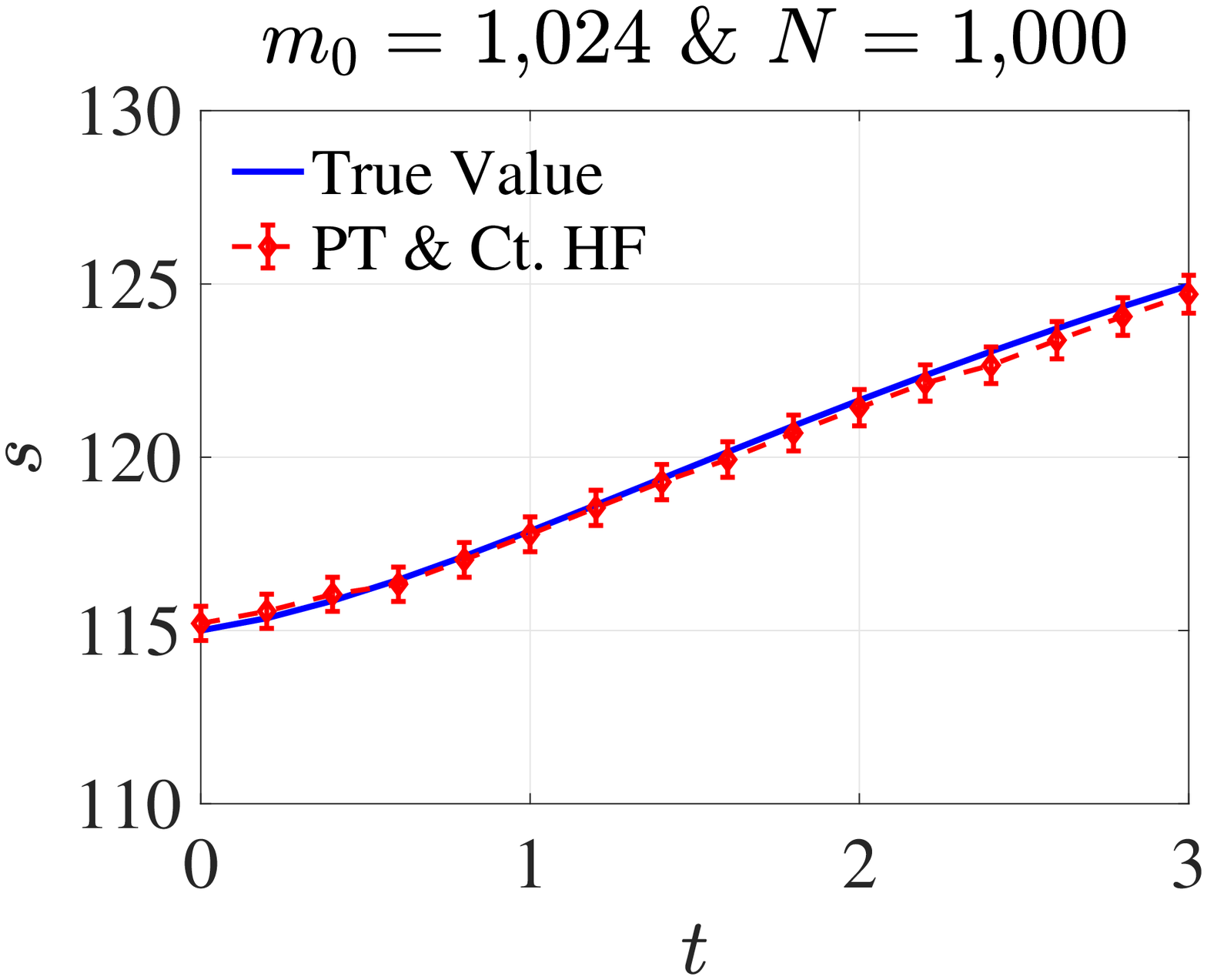}
       \endminipage \medskip \hfill
 \minipage[t]{0.33 \textwidth}
     \centering
     \includegraphics[width=\textwidth]{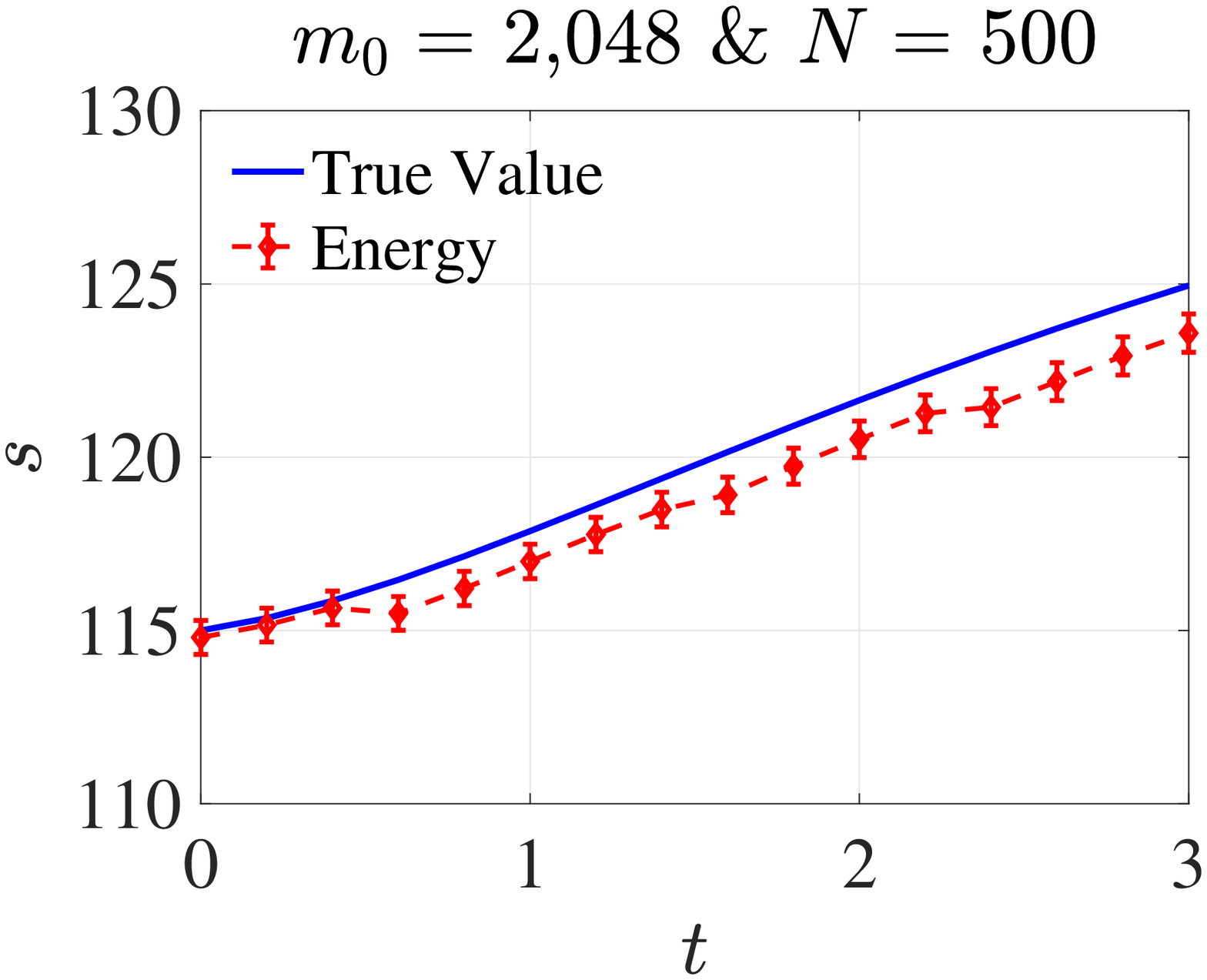} 
       \endminipage
   \minipage[t]{0.33 \textwidth}
     \centering
      \includegraphics[width=\textwidth]{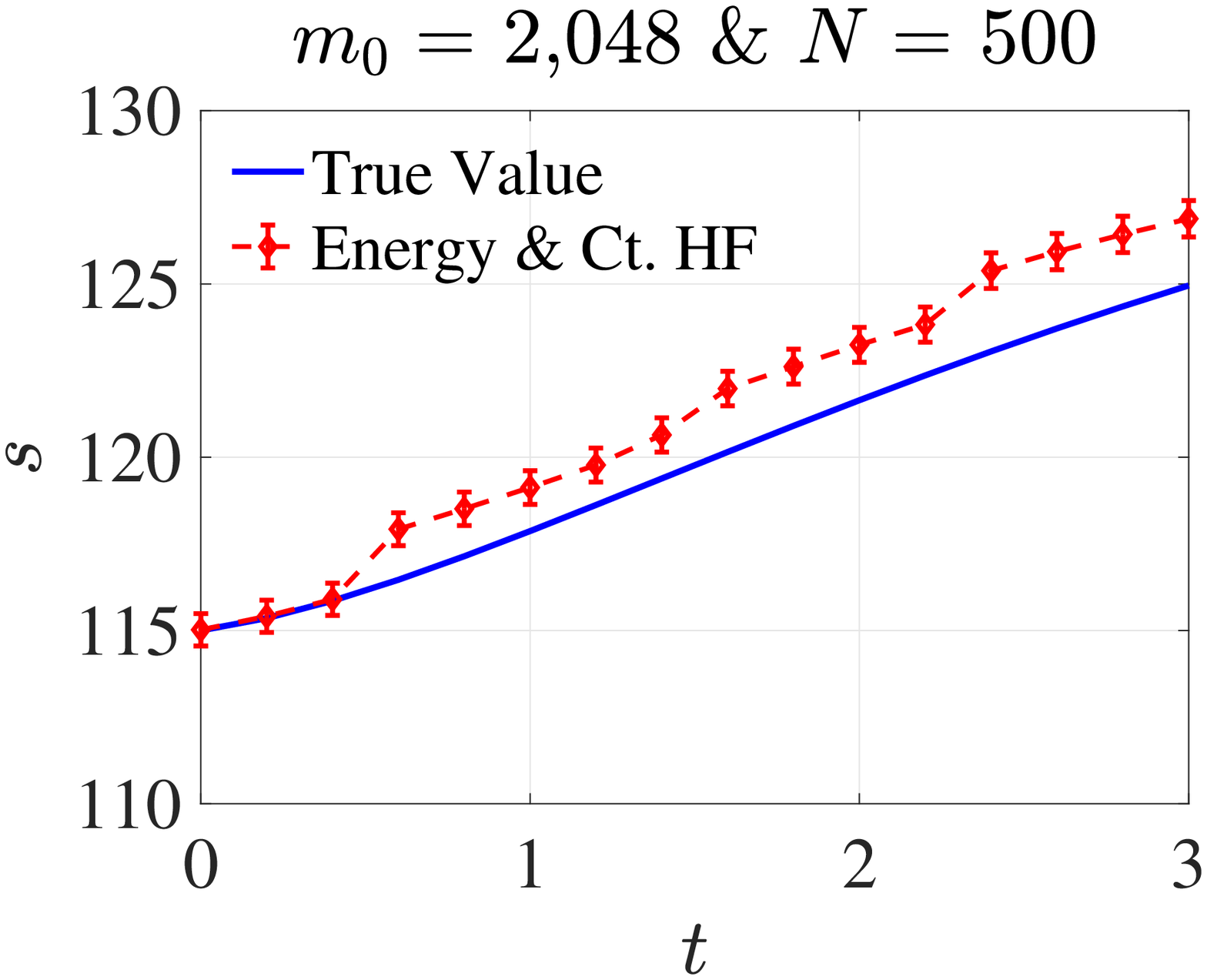}
       \endminipage 
    \minipage[t]{0.33 \textwidth}
     \centering
      \includegraphics[width=\textwidth]{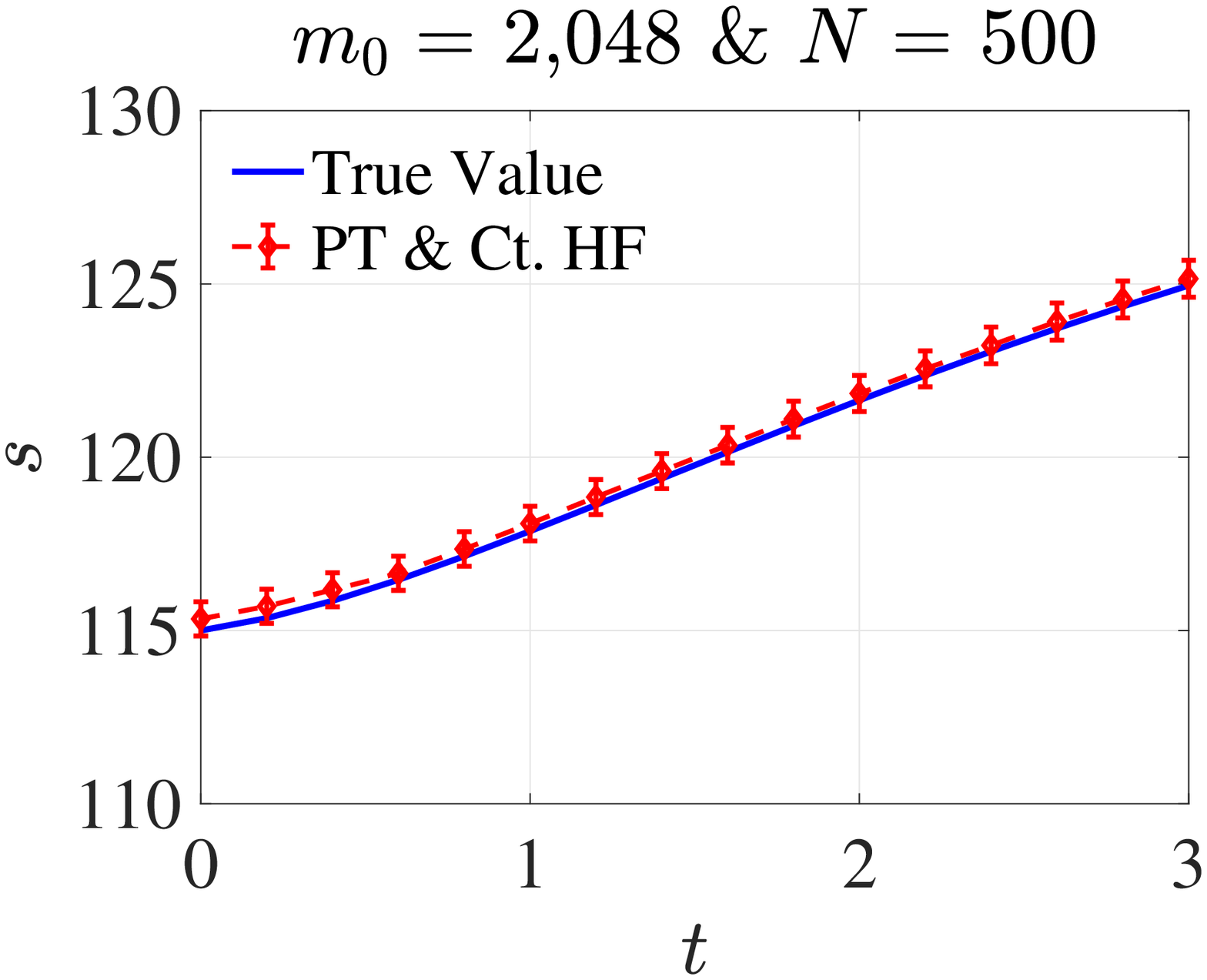}
       \endminipage \medskip \hfill
 \minipage[t]{0.33 \textwidth}
     \centering
     \includegraphics[width=\textwidth]{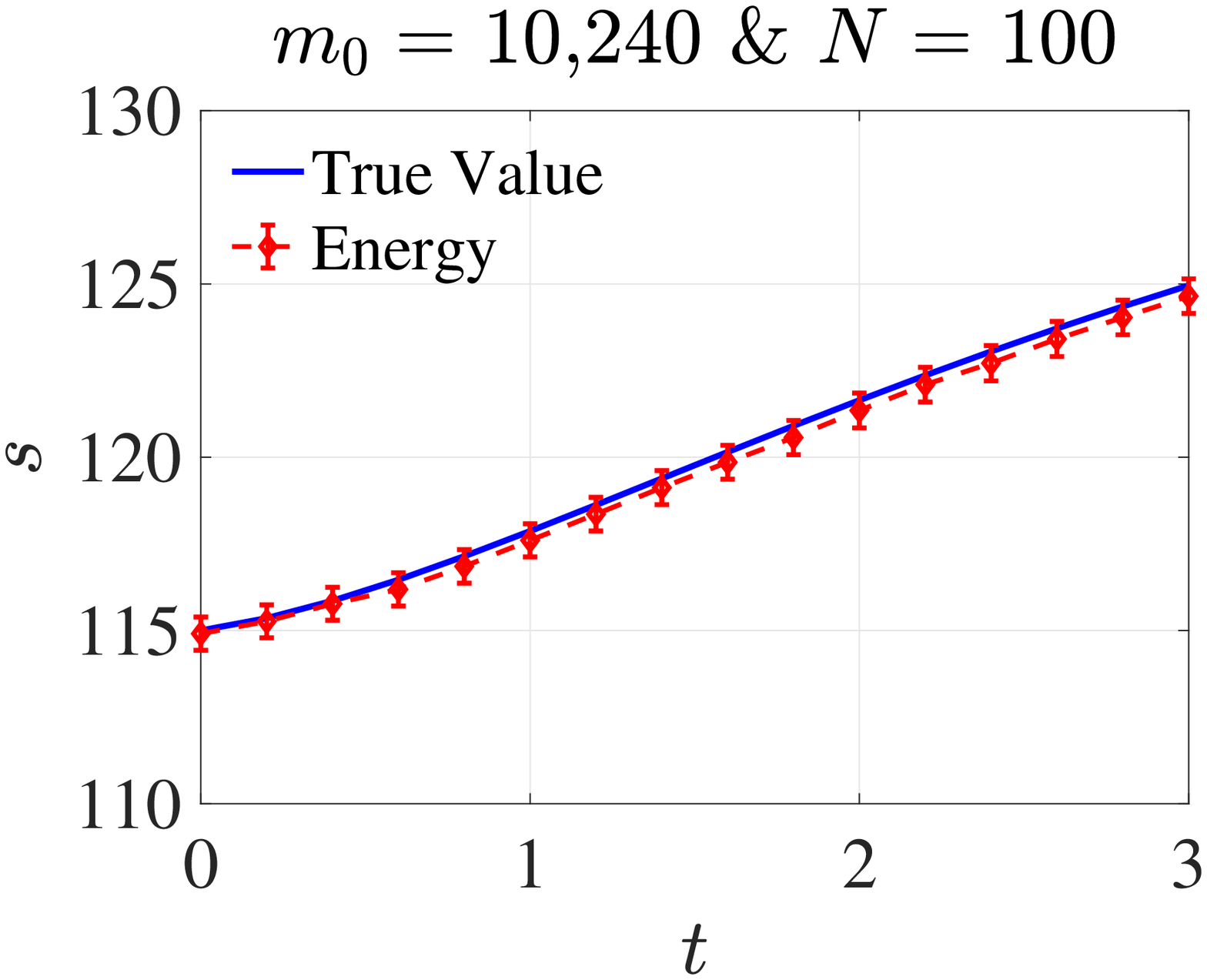} 
       \endminipage
   \minipage[t]{0.33 \textwidth}
     \centering
      \includegraphics[width=\textwidth]{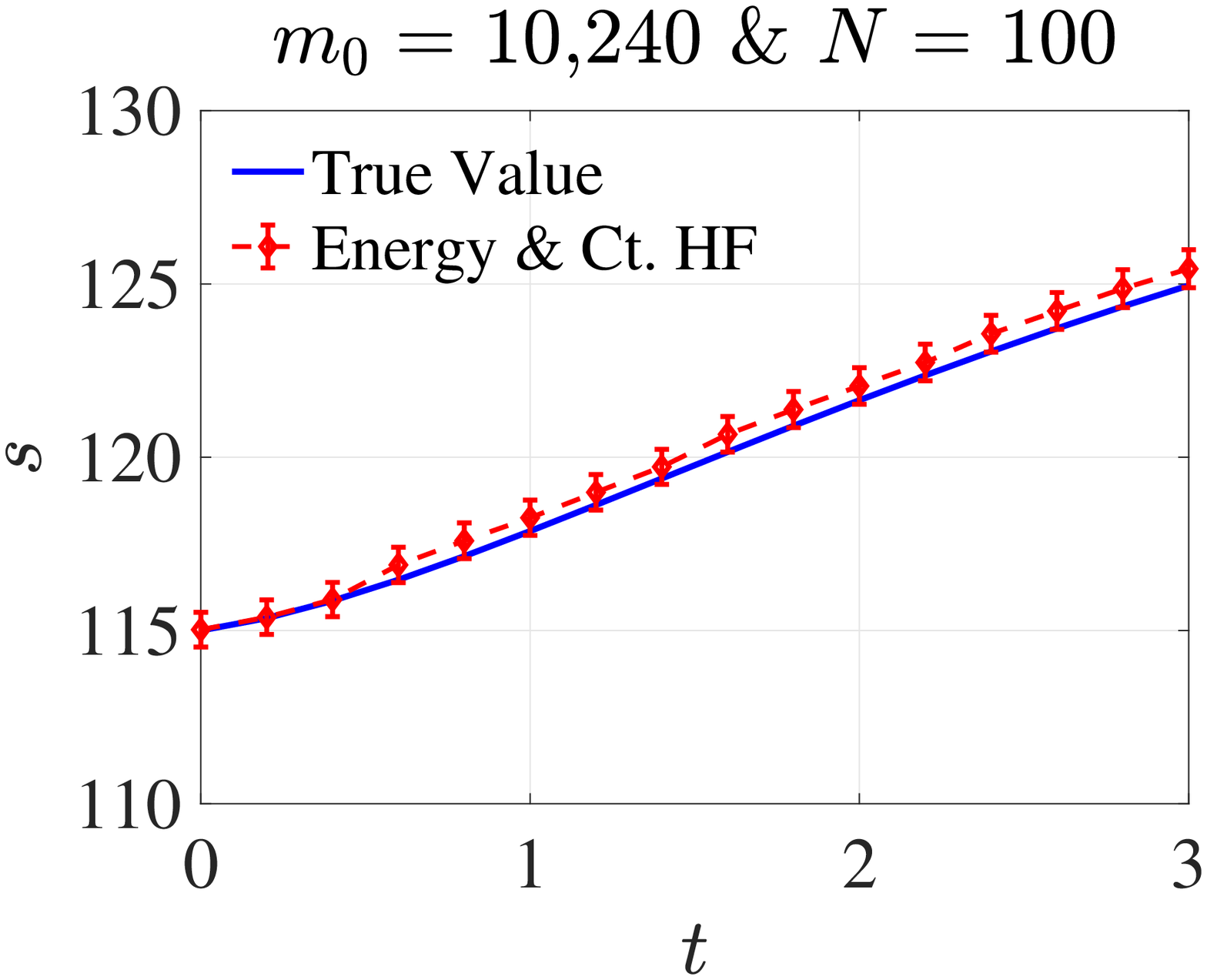}
       \endminipage 
    \minipage[t]{0.33 \textwidth}
     \centering
      \includegraphics[width=\textwidth]{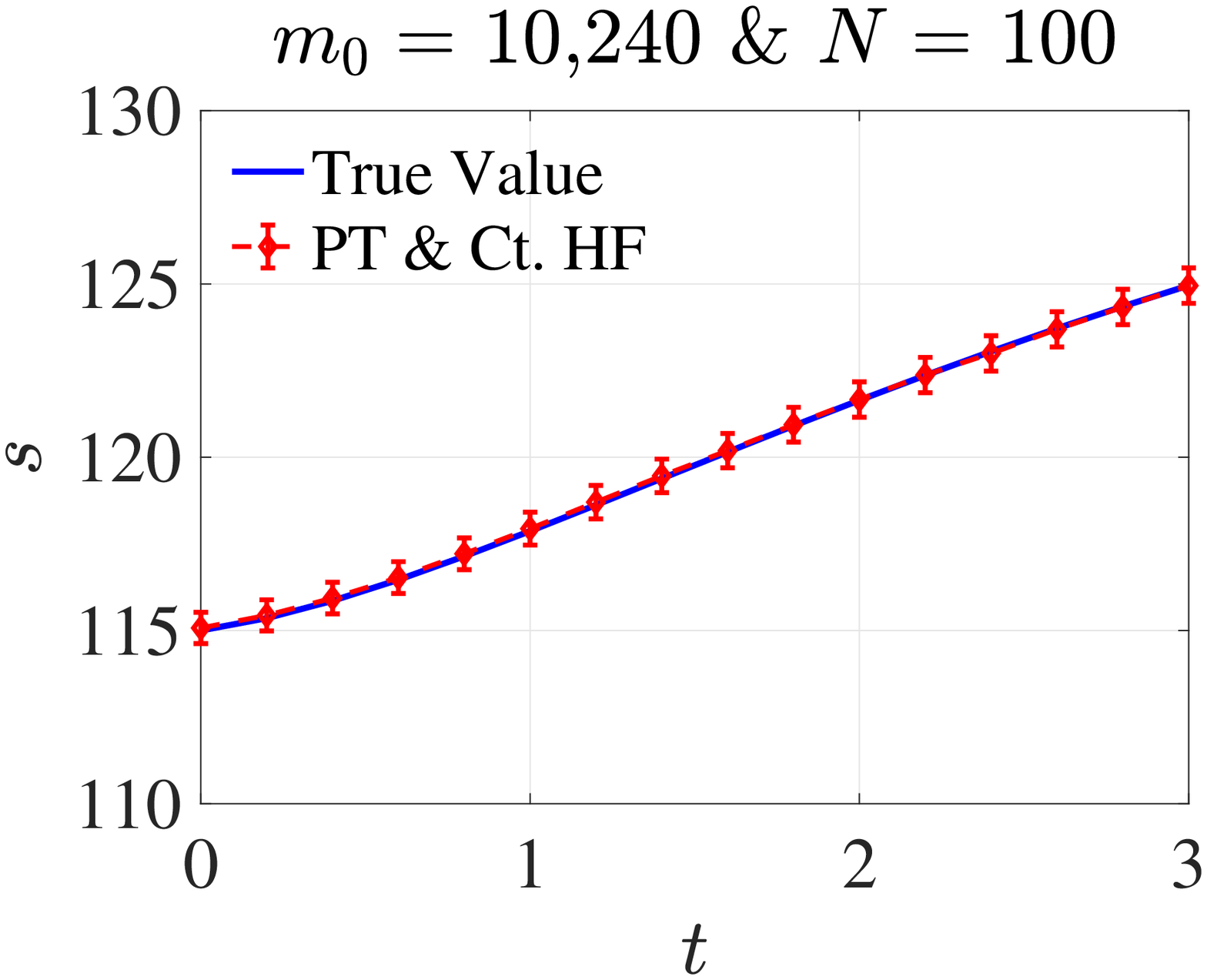}
       \endminipage \hfill
    \caption{Evolution of the scalar fourth order moment, s, as a function of time, together with $99.9\%$ confidence intervals, for different initial numbers of computational particles, $m_0$, per ensemble and different number of ensembles, $N$. For these results the total number of computational particles, $m_0 \times N$, was kept constant.  We show the results for the energy scheme (left column), energy and central heat flux scheme (middle column), and pressure tensor and central heat flux scheme (right column).}
\end{figure}

To summarize our conclusions so far, the new  reduction scheme provides improved accuracy at a significantly reduced computational cost.
 To provide additional evidence for this conclusion, we performed a second set of simulations where we fixed the total number of computational particles, $m_0 \times N$, to be 1,024,000. This value was kept constant to obtain approximately equal sized confidence intervals for all the simulations. 
The results for these simulations are shown in the right column of \cref{ABseries}
and in \cref{fig2: s}.
Comparing the errors and the half width of the confidence intervals for $\Pi_{1,1}$  in \cref{ABseries} (top right panel), we observe that our scheme
is accurate even with $m_0 = 256$, while the other two schemes require a larger initial number of computational particles. On the other hand, for $\mathbf{h}_2$
the results obtained with all three reduction schemes are acceptable  for all the values of $m_0$. In the case of the fourth order moment, for both reduction schemes of Rjasanow and Wagner,  the error lies within the confidence interval only for $m_0 =$ 10,240. On the other hand, the errors for our scheme lies within the confidence intervals for $m_0 \geq 1,024$.
In the right column of \cref{ABseries}, for each moment, the confidence intervals only depend on the total number of computational particles, $m_0 \times N$, and not on $m_0$ or the choice of reduction scheme.  Furthermore, in \cref{table:timings} we see that as $m_0$ decreases the computational time decreases. 
As a consequence, for each  moment, for the same level of accuracy the computational time for our reduction scheme is significantly less than that for the other two reduction schemes. We also observe this phenomenon in \cref{fig2: s}. 
For example, we observe the same degree of accuracy in the fourth order moment for our reduction scheme with $m_0 =$ 1,024 and $N =$ 1,000 (the right panel in the third row) as for the other two methods with $m_0 =$ 10,240 and $N =$ 100 (the left and middle panels in the last row). However, the computational time of 43 seconds for the two reduction schemes of Rjasanow and Wagner is reduced by $42\%$ to 25 seconds for our scheme.

\section{Conclusions}
\label{sec:conclusions}
We have confirmed
 that the reduction scheme of Rjasanow and Wagner that conserves total weight, momentum, energy and central heat flux of a group does not conserve the raw heat flux in each group. Consequently, the raw and central heat flux of the entire system are not conserved. We resolved this problem by devising a new reduction scheme that conserves the total weight, momentum, pressure tensor and heat flux within each group. Conservation of these moments within a group results in the conservation of all of the moments up to the second order, and both raw and central heat flux among the third order moments of a group. This further leads to the preservation of these moments for the entire system.

To examine the accuracy of our new reduction scheme, we performed simulation studies to analyze the convergence of $\Pi_{1,1}$, $\mathbf{h}_2$, and the scalar fourth order moment for the existing and new reduction schemes. The new reduction scheme leads to the convergence of these moments, particularly the scalar fourth order moment, with significantly less computational cost compared to the existing deterministic reduction schemes. This shows that the preservation of additional moments in the new reduction scheme conserves the higher moments with better accuracy, and also minimizes the reduction error.

Although the conservation of higher-order moments reduces the systematic error introduced by the reduction process,
the clustering technique must also be carefully designed in order 
to accurately and efficiently compute the low-probability
tails of the  velocity pdf. 
Specifically,  since the tails occupy a proportionately large volume of phase space, we
need to ensure that  particles in the tails
that are assigned to the same group are sufficiently close together.   
In  a forthcoming article we will
use the  proof of the convergence theorem for the SWPM~\cite{rjasanow2005stochastic}
to develop  such a clustering algorithm. In combination with the 
reduction scheme introduced in this paper, we will demonstrate that this 
leads to a more   efficient method for the computation of tail functionals.


\bibliographystyle{siamplain}
\bibliography{references}

\end{document}


\maketitle

For these simulations, we show the results for the energy reduction scheme in \cref{tab: table1}. In the first column we show the initial number of computational particles, $m_0$. In the second column we show the maximum error of the $(1,1)$-component of the momentum flux tensor, $\Pi_{1,1}$, over the time interval $[0,3]$, and in the third column we show the maximum half width of the confidence interval for $\Pi_{1,1}$. Similarly, in the fourth and fifth columns we show the same information for the the second component of the raw heat flux, $\mathbf{h}_2$, and in the sixth and seventh columns we show it for the scalar fourth moment, $s$. In the last column we show the computational time in seconds. We observe that the errors for $\Pi_{1,1}$ and $\mathbf{h}_2$ are within the confidence interval even for a small number of computational particles. However, the errors for the fourth moment $s$ are larger than the half width of the confidence interval even for $m_0 =$ 10,240.  

   In \cref{tab: table2}, we show the same data as in \cref{tab: table1} for the energy and central heat flux conservation reduction scheme, and in \cref{tab: table3} we show the results for the pressure tensor and central heat flux conservation reduction scheme. 

  In \cref{fig1: M11}, for each of the three reduction schemes, we compare the time evolution of (1,1)-component of the momentum flux tensor, $\Pi_{1,1}$, to the analytical formula for $\Pi_{1,1}$ given in \cite{rjasanow2005stochastic}. In the left column of \cref{fig1: M11}, we plot the time evolution of the momentum flux tensor component for the energy conservation reduction scheme for the values of $m_0$ given in the first column of \cref{tab: table1}. In the middle and right columns of \cref{fig1: M11}, we show the corresponding results for the energy and central heat flux conservation scheme, and the pressure tensor and central heat flux conservation scheme, respectively. For all the three reduction schemes, we see that the momentum flux tensor component is approximated even for $m_0 = 256$.
  
 Similarly, in \cref{fig1: h2} we compare the time evolution of second component of the raw heat flux, and in \cref{fig1: s} we compare the time evolution  of the fourth order moment. In \cref{fig1: h2}, we see that all three reduction schemes approximate the heat flux component. However, the plots for the pressure tensor and central heat flux conservation scheme on the right seem to overlap with the analytical curve more precisely than for the other two reduction schemes. 

\begin{table}[!htbp]
    {\footnotesize
    \caption{Energy Conservation Scheme ($N =$ 500 ensembles).}
    \label{tab: table1}
    \begin{center}
    \begin{tabular}{|R|R|R|R|R|R|R|R|} 
     \cline{2-7}
      \multicolumn{1}{c|}{}& \multicolumn{2}{c|}{$\Pi_{1,1}$}& \multicolumn{2}{c|}{$\mathbf{h}_{2}$}&  \multicolumn{2}{c|}{$s$} \\
     \hline
      $m_0$ & ${||E||}_{\infty}$& $CI_{\max}$ & ${||E||}_{\infty}$ & $CI_{\max}$ & ${||E||}_{\infty}$ & $CI_{\max}$ & $t$ (sec) \\
       \hline       
      256 & 0.0093 &0.0110 & 0.0045 & 0.0170 & 0.0652 & 0.0117 &  2.84\\
      512 & 0.0043 & 0.0077 & 0.0068 & 0.0118 & 0.0283 & 0.0088 & 6.37\\
       1,024 & 0.0055 &  0.0052 & 0.0018 & 0.0087 & 0.0238 & 0.0062 & 14.31 \\
       2,048 & 0.00052 & 0.0036 & 0.0018 & 0.0061 &  0.0100 & 0.0043 & 33.37 \\
       10,240 & 0.0007 & 0.0017 & 0.00029 & 0.0026 & 0.0021 & 0.0020 & 223.25 \\
       \hline
    \end{tabular}
  \end{center}
  }
\end{table}    

  \begin{table}[!htbp]
  {\footnotesize 
    \caption{Energy and Central Heat Flux Conservation Scheme ($N =$ 500 ensembles).}
    \label{tab: table2}
    \begin{center}
    \begin{tabular}{|R|R|R|R|R|R|R|R|} 
      \cline{2-7}
      \multicolumn{1}{c|}{}& \multicolumn{2}{c|}{$\Pi_{1,1}$}& \multicolumn{2}{c|}{$\mathbf{h}_{2}$}&  \multicolumn{2}{c|}{$s$} \\
     \hline
      $m_0$ & ${||E||}_{\infty}$& $CI_{\max}$ & ${||E||}_{\infty}$ & $CI_{\max}$ & ${||E||}_{\infty}$ & $CI_{\max}$ & $t$ (sec) \\
       \hline       
      256 & 0.0098 & 0.0104 & 0.0050 & 0.0170 & 0.0196 & 0.0135 & 3.03 \\
      512 & 0.0070 & 0.0071 & 0.0024 & 0.0120 & 0.0281 & 0.0088 & 6.84 \\
       1,024 & 0.0034 & 0.0052 & 0.0041 & 0.0084 & 0.0231 & 0.0062 & 15.16 \\
       2,048 & 0.00065 & 0.0037 & 0.00045 & 0.0058 & 0.0154 & 0.0046 & 34.38 \\
       10,240 & 0.00098 & 0.0017 & 0.0015 & 0.0026 & 0.0056 & 0.0020 & 226.81 \\
       \hline
    \end{tabular}
  \end{center}
  }
\end{table}    

 \begin{table}[!htbp]
  {\footnotesize
    \caption{Pressure Tensor and Central Heat Flux Conservation Scheme ($N =$ 500 ensembles).} 
    \label{tab: table3}
    \begin{center}
    \begin{tabular}{|R|R|R|R|R|R|R|R|} 
 \cline{2-7}
      \multicolumn{1}{c|}{}& \multicolumn{2}{c|}{$\Pi_{1,1}$}& \multicolumn{2}{c|}{$\mathbf{h}_{2}$}&  \multicolumn{2}{c|}{$s$} \\
     \hline
      $m_0$ & ${||E||}_{\infty}$& $CI_{\max}$ & ${||E||}_{\infty}$ & $CI_{\max}$ & ${||E||}_{\infty}$ & $CI_{\max}$ & $t$ (sec) \\
       \hline       
      256 & 0.0020 & 0.0106 & 0.0034 & 0.0165 & 0.0164 & 0.0124 & 2.61 \\
      512 & 0.0019 & 0.0076 & 0.0018 & 0.0115 & 0.0097 & 0.0089 & 6.06\\
       1,024 & 0.0005 & 0.0050 & 0.00073 & 0.0083 & 0.0035 & 0.0062 & 13.76\\
       2,048 & 0.00086 & 0.0036 & 0.0005 & 0.0057 & 0.0012 & 0.0042 & 31.26\\
       10,240 & 0.00021 & 0.0017 & 0.00048 & 0.0026 & 0.0006 & 0.0019 & 205.56\\
       \hline
    \end{tabular}
  \end{center}
  }
\end{table}

 \begin{table}[!htbp]
 {\footnotesize
  \caption{Energy Conservation Scheme ($m_0 \times N =$ 1,240,00).}
     \label{tab: 1table1}
       \begin{center}
    \begin{tabular}{|R|R|R|R|R|R|R|R|R|} 
     \cline{3-8}
      \multicolumn{2}{c|}{}& \multicolumn{2}{c|}{$\Pi_{1,1}$}& \multicolumn{2}{c|}{$\mathbf{h}_{2}$}&  \multicolumn{2}{c|}{$s$} \\
     \hline
      $m_0$ & $N$ & ${||E||}_{\infty}$& $CI_{\max}$ & ${||E||}_{\infty}$ & $CI_{\max}$ & ${||E||}_{\infty}$ & $CI_{\max}$ & $t$ (sec) \\
       \hline       
      256 & 4,000 & 0.0098 & 0.0040 & 0.00063 & 0.0061 & 0.0614 & 0.0044 & 22.32\\
      512 & 2,000 & 0.00049 & 0.0037 & 0.00056 & 0.0060 & 0.0283 & 0.0044 & 25.44\\
       1,024 & 1,000 & 0.0047 & 0.0038 & 0.00097 & 0.0059 & 0.0225 & 0.0044 & 28.46\\
       2,048 & 500 & 0.0031 & 0.0039 & 0.0045 & 0.0058 & 0.0110 & 0.0044 & 33.37\\
       10,240 & 100 & 0.00024 & 0.0035 & 0.0016 & 0.0052 & 0.0025 & 0.0040 & 42.74\\
       \hline
    \end{tabular}
  \end{center}
  }
  \end{table}    
  
   \begin{table}[!htbp]
  {\footnotesize
    \caption{Energy and Central Heat Flux Conservation Scheme ($m_0 \times N =$1,240,00).}
    \label{tab: 1table2}
    \begin{center}
    \begin{tabular}{|R|R|R|R|R|R|R|R|R|} 
     \cline{3-8}
      \multicolumn{2}{c|}{}& \multicolumn{2}{c|}{$\Pi_{1,1}$}& \multicolumn{2}{c|}{$\mathbf{h}_{2}$}&  \multicolumn{2}{c|}{$s$} \\
     \hline
      $m_0$ & $N$ & ${||E||}_{\infty}$& $CI_{\max}$ & ${||E||}_{\infty}$ & $CI_{\max}$ & ${||E||}_{\infty}$ & $CI_{\max}$ & $t$ (sec) \\
       \hline       
      256 & 4,000 & 0.0131 & 0.0037 & 0.00023 & 0.0060 & 0.0153 & 0.0046 & 24.25\\
      512 & 2,000 & 0.0051 & 0.0037 & 0.00092 & 0.0060 & 0.0270 & 0.0046 & 26.9\\
       1,024 & 1,000 & 0.0024 & 0.0038 & 0.0030 & 0.0059 & 0.0238 & 0.0044 & 30.02\\
       2,048 & 500 & 0.00072 & 0.0038 & 0.0023 & 0.0056 & 0.0154 & 0.0042 & 34.38\\
       10,240 & 100 & 0.0014 & 0.0038 & 0.0022 & 0.0057 & 0.0039 & 0.0044 & 44.36\\
       \hline
    \end{tabular}
  \end{center}
  }
\end{table}    

 \begin{table}[!htbp]
  {\footnotesize
    \caption{Pressure Tensor and Central Heat Flux Conservation Scheme ($m_0 \times N =$ 1,240,00).} 
    \label{tab: 1table3}
    \begin{center}
    \begin{tabular}{|R|R|R|R|R|R|R|R|R|} 
 \cline{3-8}
      \multicolumn{2}{c|}{}& \multicolumn{2}{c|}{$\Pi_{1,1}$}& \multicolumn{2}{c|}{$\mathbf{h}_{2}$}&  \multicolumn{2}{c|}{$s$} \\
     \hline
      $m_0$ & $N$ & ${||E||}_{\infty}$& $CI_{\max}$ & ${||E||}_{\infty}$ & $CI_{\max}$ & ${||E||}_{\infty}$ & $CI_{\max}$ & $t$ (sec) \\
       \hline       
      256 & 4,000 & 0.00023 & 0.0036 & 0.0023 & 0.0059 & 0.0171 & 0.0044 & 20.85\\
      512 & 2,000 & 0.0013 & 0.0037 & 0.0016 & 0.0056 & 0.0077 & 0.0042 & 22.02 \\
       1,024 & 1,000 & 0.00044 & 0.0035 & 0.0021 & 0.0058 & 0.0020 & 0.0044 & 24.93\\
       2,048 & 500 & 0.0016 & 0.0036 & 0.00078 & 0.0058 & 0.0016 & 0.0043 & 31.26\\
       10,240 & 100 & 0.0011 & 0.0036 & 0.00088 & 0.0064 & 0.00003 & 0.00041 & 39.05\\
       \hline
    \end{tabular}
  \end{center}
  }
\end{table}

\begin{figure}[!htbp] 
  \label{fig1: M11} 
  \minipage[t]{0.33 \textwidth}
    \centering
    \includegraphics[width=\textwidth]{Figures/figM(1,1)Energy(256)} 
  \endminipage 
  \minipage[t]{0.33 \textwidth}
    \centering
    \includegraphics[width=\textwidth]{Figures/figM(1,1)RW(256)}
  \endminipage 
  \minipage[t]{0.33 \textwidth}
    \centering
    \includegraphics[width=\textwidth]{Figures/figM(1,1)Our(256)}
  \endminipage \medskip \hfill
  \minipage[t]{0.33 \textwidth}
    \centering
    \includegraphics[width=\textwidth]{Figures/figM(1,1)Energy(512)} 
  \endminipage
  \minipage[t]{0.33 \textwidth}
    \centering
    \includegraphics[width=\textwidth]{Figures/figM(1,1)RW(512)}
  \endminipage 
  \minipage[t]{0.33 \textwidth}
    \centering
    \includegraphics[width=\textwidth]{Figures/figM(1,1)Our(512)}
  \endminipage \medskip \hfill 
  \minipage[t]{0.33 \textwidth}
    \centering
    \includegraphics[width=\textwidth]{Figures/figM(1,1)Energy(1024)} 
  \endminipage
  \minipage[t]{0.33 \textwidth}
    \centering
    \includegraphics[width=\textwidth]{Figures/figM(1,1)RW(1024)}
   \endminipage 
   \minipage[t]{0.33 \textwidth}
     \centering
     \includegraphics[width=\textwidth]{Figures/figM(1,1)Our(1024)}
   \endminipage \medskip \hfill
   \minipage[t]{0.33 \textwidth}
     \centering
     \includegraphics[width=\textwidth]{Figures/figM(1,1)Energy(2048)} 
   \endminipage
   \minipage[t]{0.33 \textwidth}
     \centering
     \includegraphics[width=\textwidth]{Figures/figM(1,1)RW(2048)}
   \endminipage 
   \minipage[t]{0.33 \textwidth}
     \centering
     \includegraphics[width=\textwidth]{Figures/figM(1,1)Our(2048)}
   \endminipage \medskip \hfill 
   \minipage[t]{0.33 \textwidth}
     \centering
     \includegraphics[width=\textwidth]{Figures/figM(1,1)Energy(10240)} 
   \endminipage
   \minipage[t]{0.33 \textwidth}
     \centering
     \includegraphics[width=\textwidth]{Figures/figM(1,1)RW(10240)}
   \endminipage 
   \minipage[t]{0.33 \textwidth}
     \centering
     \includegraphics[width=\textwidth]{Figures/figM(1,1)Our(10240)}
    \endminipage \hfill 
    \caption{Evolution of the (1,1)-component of the momentum flux tensor, $\Pi_{1,1}$, as a function of time together with $99.9\%$ confidence intervals. In the different rows we show the results for different initial numbers of computational particles, $m_0$, per ensemble. We used $N =$ 500 ensembles in all the panels. We show the results for the energy scheme (left column), energy and central heat flux scheme (middle column), and pressure tensor and central heat flux scheme (right column).}
\end{figure}

\begin{figure}[!htbp] 
   \label{fig1: h2} 
   \minipage[t]{0.33 \textwidth}
     \centering
     \includegraphics[width=\textwidth]{Figures/fighEnergy(256)} 
       \endminipage 
   \minipage[t]{0.33 \textwidth}
     \centering
      \includegraphics[width=\textwidth]{Figures/fighRW(256)}
       \endminipage 
    \minipage[t]{0.33 \textwidth}
     \centering
      \includegraphics[width=\textwidth]{Figures/fighOur(256)}
       \endminipage \medskip \hfill
       \minipage[t]{0.33 \textwidth}
     \centering
     \includegraphics[width=\textwidth]{Figures/fighEnergy(512)} 
       \endminipage
   \minipage[t]{0.33 \textwidth}
     \centering
      \includegraphics[width=\textwidth]{Figures/fighRW(512)}
       \endminipage 
    \minipage[t]{0.33 \textwidth}
     \centering
      \includegraphics[width=\textwidth]{Figures/fighOur(512)}
       \endminipage \medskip \hfill 
 \minipage[t]{0.33 \textwidth}
     \centering
     \includegraphics[width=\textwidth]{Figures/fighEnergy(1024)} 
       \endminipage
   \minipage[t]{0.33 \textwidth}
     \centering
      \includegraphics[width=\textwidth]{Figures/fighRW(1024)}
       \endminipage 
    \minipage[t]{0.33 \textwidth}
     \centering
      \includegraphics[width=\textwidth]{Figures/fighOur(1024)}
       \endminipage \medskip \hfill
 \minipage[t]{0.33 \textwidth}
     \centering
     \includegraphics[width=\textwidth]{Figures/fighEnergy(2048)} 
       \endminipage
   \minipage[t]{0.33 \textwidth}
     \centering
      \includegraphics[width=\textwidth]{Figures/fighRW(2048)}
       \endminipage 
    \minipage[t]{0.33 \textwidth}
     \centering
      \includegraphics[width=\textwidth]{Figures/fighOur(2048)}
       \endminipage \medskip \hfill 
 \minipage[t]{0.33 \textwidth}
     \centering
     \includegraphics[width=\textwidth]{Figures/fighEnergy(10240)} 
       \endminipage
   \minipage[t]{0.33 \textwidth}
     \centering
      \includegraphics[width=\textwidth]{Figures/fighRW(10240)}
       \endminipage 
    \minipage[t]{0.33 \textwidth}
     \centering
      \includegraphics[width=\textwidth]{Figures/fighOur(10240)}
       \endminipage \hfill 
     \caption{Evolution of the second component of the raw heat flux, $\mathbf{h}_2$, as a function of time together with $99.9\%$ confidence intervals. In the different rows we show the results for different initial numbers of computational particles, $m_0$, per ensemble. We used $N =$ 500 ensembles in all the panels. We show the results for the energy scheme (left column), energy and central heat flux scheme (middle column), and pressure tensor and central heat flux scheme (right column).}
\end{figure}

\begin{figure}[!htbp] 
   \label{fig2: M11}
   \minipage[t]{0.33 \textwidth}
     \centering
     \includegraphics[width=\textwidth]{Figures/fig2M(1,1)Energy(256)} 
       \endminipage 
   \minipage[t]{0.33 \textwidth}
     \centering
      \includegraphics[width=\textwidth]{Figures/fig2M(1,1)RW(256)}
       \endminipage 
    \minipage[t]{0.33 \textwidth}
     \centering
      \includegraphics[width=\textwidth]{Figures/fig2M(1,1)Our(256)}
       \endminipage \medskip \hfill
       \minipage[t]{0.33 \textwidth}
     \centering
     \includegraphics[width=\textwidth]{Figures/fig2M(1,1)Energy(512)} 
       \endminipage
   \minipage[t]{0.33 \textwidth}
     \centering
      \includegraphics[width=\textwidth]{Figures/fig2M(1,1)RW(512)}
       \endminipage 
    \minipage[t]{0.33 \textwidth}
     \centering
      \includegraphics[width=\textwidth]{Figures/fig2M(1,1)Our(512)}
       \endminipage \medskip \hfill
 \minipage[t]{0.33 \textwidth}
     \centering
     \includegraphics[width=\textwidth]{Figures/fig2M(1,1)Energy(1024)} 
       \endminipage
   \minipage[t]{0.33 \textwidth}
     \centering
      \includegraphics[width=\textwidth]{Figures/fig2M(1,1)RW(1024)}
       \endminipage 
    \minipage[t]{0.33 \textwidth}
     \centering
      \includegraphics[width=\textwidth]{Figures/fig2M(1,1)Our(1024)}
       \endminipage \medskip \hfill
 \minipage[t]{0.33 \textwidth}
     \centering
     \includegraphics[width=\textwidth]{Figures/fig2M(1,1)Energy(2048)} 
       \endminipage
   \minipage[t]{0.33 \textwidth}
     \centering
      \includegraphics[width=\textwidth]{Figures/fig2M(1,1)RW(2048)}
       \endminipage 
    \minipage[t]{0.33 \textwidth}
     \centering
      \includegraphics[width=\textwidth]{Figures/fig2M(1,1)Our(2048)}
       \endminipage \medskip \hfill
 \minipage[t]{0.33 \textwidth}
     \centering
     \includegraphics[width=\textwidth]{Figures/fig2M(1,1)Energy(10240)} 
       \endminipage
   \minipage[t]{0.33 \textwidth}
     \centering
      \includegraphics[width=\textwidth]{Figures/fig2M(1,1)RW(10240)}
       \endminipage 
    \minipage[t]{0.33 \textwidth}
     \centering
      \includegraphics[width=\textwidth]{Figures/fig2M(1,1)Our(10240)}
       \endminipage \hfill
     \caption{Evolution of the (1,1)-component of the momentum flux tensor, $\Pi_{1,1}$, as a function of time, together with $99.9\%$ confidence intervals, for different initial numbers of computational particles, $m_0$, per ensemble and different number of ensembles, $N$. For these results the total number of computational particles, $m_0 \times N$, was kept constant.  We show the results for the energy scheme (left column), energy and central heat flux scheme (middle column), and pressure tensor and central heat flux scheme (right column).}
\end{figure} 

\begin{figure}[!htbp] 
   \label{fig2: h2}
   \minipage[t]{0.33 \textwidth}
     \centering
     \includegraphics[width=\textwidth]{Figures/fig2hEnergy(256)} 
       \endminipage 
   \minipage[t]{0.33 \textwidth}
     \centering
      \includegraphics[width=\textwidth]{Figures/fig2hRW(256)}
       \endminipage 
    \minipage[t]{0.33 \textwidth}
     \centering
      \includegraphics[width=\textwidth]{Figures/fig2hOur(256)}
       \endminipage \medskip \hfill
       \minipage[t]{0.33 \textwidth}
     \centering
     \includegraphics[width=\textwidth]{Figures/fig2hEnergy(512)} 
       \endminipage
   \minipage[t]{0.33 \textwidth}
     \centering
      \includegraphics[width=\textwidth]{Figures/fig2hRW(512)}
       \endminipage 
    \minipage[t]{0.33 \textwidth}
     \centering
      \includegraphics[width=\textwidth]{Figures/fig2hOur(512)}
       \endminipage \medskip \hfill
 \minipage[t]{0.33 \textwidth}
     \centering
     \includegraphics[width=\textwidth]{Figures/fig2hEnergy(1024)} 
       \endminipage
   \minipage[t]{0.33 \textwidth}
     \centering
      \includegraphics[width=\textwidth]{Figures/fig2hRW(1024)}
       \endminipage 
    \minipage[t]{0.33 \textwidth}
     \centering
      \includegraphics[width=\textwidth]{Figures/fig2hOur(1024)}
       \endminipage \medskip \hfill
 \minipage[t]{0.33 \textwidth}
     \centering
     \includegraphics[width=\textwidth]{Figures/fig2hEnergy(2048)} 
       \endminipage
   \minipage[t]{0.33 \textwidth}
     \centering
      \includegraphics[width=\textwidth]{Figures/fig2hRW(2048)}
       \endminipage 
    \minipage[t]{0.33 \textwidth}
     \centering
      \includegraphics[width=\textwidth]{Figures/fig2hOur(2048)}
       \endminipage \medskip \hfill
 \minipage[t]{0.33 \textwidth}
     \centering
     \includegraphics[width=\textwidth]{Figures/fig2hEnergy(10240)} 
       \endminipage
   \minipage[t]{0.33 \textwidth}
     \centering
      \includegraphics[width=\textwidth]{Figures/fig2hRW(10240)}
       \endminipage 
    \minipage[t]{0.33 \textwidth}
     \centering
      \includegraphics[width=\textwidth]{Figures/fig2hOur(10240)}
       \endminipage \hfill
     \caption{Evolution of the second component of the raw heat flux, $\mathbf{h}_2$, as a function of time, together with $99.9\%$ confidence intervals, for different initial numbers of computational particles, $m_0$, per ensemble and different number of ensembles, $N$. For these results the total number of computational particles, $m_0 \times N$, was kept constant.  We show the results for the energy scheme (left column), energy and central heat flux scheme (middle column), and pressure tensor and central heat flux scheme (right column).}
\end{figure}